%% file: manuscript.tex
\documentclass[11pt]{article}
\usepackage{latexsym}
\usepackage{amsmath}
\usepackage{amsmath}
\usepackage{amssymb}
\usepackage{amsthm}
\usepackage{epsfig}
\usepackage{xcolor}
\usepackage{graphicx}
\usepackage{bm}
\usepackage{enumitem}
\usepackage{mathtools}
\mathtoolsset{showonlyrefs}
\usepackage{graphicx}
\usepackage{tikz}
\usepackage{mathrsfs}
\usepackage[toc,page]{appendix}
\usepackage{graphicx}
\usepackage{bbm}
\usepackage{tikz}
\usetikzlibrary{decorations.pathreplacing,arrows}
\usepackage{ytableau}
\usepackage{tkz-berge}

\usepackage[linesnumbered, ruled, vlined]{algorithm2e}
\usepackage[top=1in, bottom=1in, left=1in, right=1in]{geometry}

\newtheorem{theorem}{Theorem}[section]
\newtheorem{proposition}[theorem]{Proposition}
\newtheorem{lemma}[theorem]{Lemma}
\newtheorem{corollary}[theorem]{Corollary}
\newtheorem{conjecture}[theorem]{Conjecture}
\theoremstyle{definition}

\theoremstyle{remark}
\newtheorem{remark}[theorem]{Remark}

\usepackage{array}
\usepackage{multirow}

\usepackage{hyperref}
\hypersetup{
    colorlinks,
    linkcolor={blue!80!black},
    citecolor={green!50!black},
}
\colorlet{linkequation}{blue}
\usepackage{authblk}

\input{def}

\def\scL{\mathscr{L}}

\def\scE{\mathscr{E}}
\def\scF{\mathscr{F}}
\def\scS{\mathscr{S}}

\def\sfS{\mathsf{S}}
\def\Id{\mathrm{I}}

\title{The sharp SAT/UNSAT phase transition in random ellipsoid fitting}

\author[1]{Theodor Misiakiewicz}
\author[1]{Garrett G. Wen}

\affil[1]{Department of Statistics and Data Science, Yale University}

\date{\today}

\allowdisplaybreaks

\begin{document}

\maketitle 

\begin{abstract}
    Let $x_1,\ldots,x_n$ be independent standard Gaussian vectors in $\R^d$. An \emph{ellipsoid fit} is a matrix $S \succeq 0$ such that $x_i^\top S x_i =d$ for every $i$, so that all the points lie on the boundary of the centered ellipsoid $\{ x : x^\top S x = d\}$. Saunderson, Parrilo and Willsky conjectured that, as $n,d \to \infty$, this semidefinite feasibility problem undergoes a sharp transition at $n \sim d^2/4$. We prove this conjecture. If $\lim \sup n/d^2 = \alpha^* <1/4$, then, with probability tending to one, an ellipsoid fit exists; moreover, one can choose $S$ with all eigenvalues in a fixed interval $[\lambda_- , \lambda_+] \subset (0,\infty)$ depending only on $\alpha^*$. Conversely, if $\lim \inf n/d^2  > 1/4$, then, with probability tending to one, no ellipsoid fit exists, without any spectral restriction.  
    
    Our proof builds on the Gaussian-equivalence framework developed by Bandeira and Maillard (2025) and closes the two gaps left open in their work: establishing exact fitting and removing the operator-norm constraint. On the satisfiable side, the new ingredients are a head-tail decomposition of the dual vector, exact correction of the sparse head constraints, and a Gaussian comparison principle for the low-influence tail. On the unsatisfiable side, we split a candidate into a low-rank spectral head and a Schatten-3 diffuse bulk, Gaussianize the bulk conditionally on the head, and apply a projected Gordon escape argument. The threshold is governed by the statistical dimension $d(d+1)/4$ of the positive semidefinite cone.
\end{abstract}

\section{Introduction}

\subsection{The ellipsoid fitting problem}

Given $n$ independent standard Gaussian points $x_1,\dots,x_n \sim \cN(0, I_d)$ in dimension $d$, an elementary question of random geometry asks the following: \emph{when do all $n$ points lie on the boundary of a common centered ellipsoid?}  Writing $\sfS^d$ for the space of $d \times d$ real symmetric matrices and $S \succeq 0$ for positive semidefiniteness, the question asks for the feasibility of the semidefinite program
\begin{equation}\label{eq:P}
x_i^{\top} S x_i = d \quad (1 \le i \le n), \qquad S \succeq 0,
\qquad S \in \sfS^d .
\tag{P}
\end{equation}
The eigenvectors of a solution $S$ are the principal axes of the fitted ellipsoid\footnote{Note that we allow degenerate elliptic cylinders with axes ${\rm ker}(S)$. In fact, Maillard and Kunisky \cite{maillard2024fitting} predicts that, at the transition threshold, typical solutions of \eqref{eq:P} have half of their semiaxes diverging.} $\{x \in \R^d : x^\top S x = d\}$, and its eigenvalues $(\lambda_j)_{j \le d}$ determine the semiaxis lengths $r_j = \sqrt{d} \lambda_j^{-1/2}$. Since $\|x_i\|_2^2 = d\,(1 + \widetilde O_{\P}(d^{-1/2}))$, the sphere itself $S = I_d$ very nearly fits, and \eqref{eq:P} asks whether the residual fluctuations of order $\sqrt d$ can be absorbed \emph{exactly} by an anisotropic quadratic.

Two regimes are immediate.  If $n \le d$, the samples are almost surely linearly independent, and an invertible change of coordinates sending $x_i$ to $\sqrt d\,e_i$ produces a positive definite fit.  If $n > \dim \sfS^d = d(d+1)/2$, the linear system in \eqref{eq:P} is almost surely inconsistent and no fit exists.  The feasibility transition thus lies between these two bounds.  The following sharp prediction was put forward by Saunderson, Parrilo and Willsky \cite{saunderson2011subspace,saunderson2013diagonal} on the basis of strong numerical evidence; see also \cite[Conjecture~1.1]{potechin2023near} for the now-standard two-sided formulation.

\begin{conjecture}[The ellipsoid fitting conjecture]\label{conj:main}
Let $p(n,d) := \P\{\exists S \succeq 0 :\ x_i^{\top} S x_i = d \ \ \forall i \in [n]\}$, where $x_1,\dots,x_n \stackrel{\mathrm{iid}}{\sim} \cN(0,I_d)$. For every $\varepsilon > 0$, as $d \to \infty$,
\[
\limsup_{d\to\infty} \frac{n}{d^2} \le \frac{1-\varepsilon}{4}
\;\Longrightarrow\; p(n,d) \to 1,
\qquad\qquad
\liminf_{d\to\infty} \frac{n}{d^2} \ge \frac{1+\varepsilon}{4}
\;\Longrightarrow\; p(n,d) \to 0 .
\]
\end{conjecture}

The conjecture has attracted sustained attention over the last decade, both for its own sake---it is among the cleanest examples of a sharp phase transition for a random semidefinite program---and for its consequences in statistics and theoretical computer science, which we recall in Section~\ref{sec:origins}. A sequence of works \cite{saunderson2011subspace,saunderson2012diagonal,saunderson2013diagonal,ghosh2020sum,kane2023nearly,potechin2023near,bandeira2024fitting,hsieh2023ellipsoid,tulsiani2025ellipsoid} pushed the satisfiability region from $n = O(d^{6/5-\eps})$ up to $n \leq c d^2$ for a small absolute constant $c$, while on the other side nothing better than the trivial dimension-count bound $n > d(d+1)/2$ was known.  Recently, Maillard and Kunisky \cite{maillard2024fitting} predicted the transition at $n = d^2/4$ using the non-rigorous replica method, and Bandeira and Maillard \cite{bandeira2025exact} proved that an approximate version of the problem, in which the constraints may be violated by a small average error and the fit has bounded operator norm, does transition sharply at $1/4$. The exact conjecture, with zero error and no spectral restrictions, remained open on both sides.

The purpose of this paper is to prove Conjecture \ref{conj:main}. Section~\ref{sec:main-results} states our main result, while Section~\ref{sec:origins} reviews the origins of the conjecture and prior progress toward its resolution. Section~\ref{sec:heuristic} presents the Gaussian comparison heuristic and the results of Bandeira and Maillard \cite{bandeira2025exact} that motivate our approach. Section~\ref{sec:proof-method} provides an overview of our proof strategy. Finally, Section~\ref{sec:notation} gathers notations and outlines the organization of the remainder of the paper.

\subsection{Main results}
\label{sec:main-results}

Our main result establishes the sharp feasibility threshold for the ellipsoid fitting problem \eqref{eq:P}.

\begin{theorem}[The sharp ellipsoid fitting transition]\label{thm:main}
Let $x_1,\dots,x_n \stackrel{\mathrm{iid}}{\sim} \cN(0,I_d)$ and let $p(n,d) = \Prob\{\exists S \succeq 0 :\ x_i^{\top} S x_i = d \ \forall i\}$.
\begin{enumerate}
\item[{\rm (a)}] \textup{(Satisfiable phase.)} If $\ \limsup_{d \to \infty} n/d^2 < 1/4$, then $p(n,d) \to 1$.
\item[{\rm (b)}] \textup{(Unsatisfiable phase.)} If $\ \liminf_{d \to \infty} n/d^2 > 1/4$, then $p(n,d) \to 0$.
\end{enumerate}
\end{theorem}

Below the threshold, our proof gives in fact a stronger existence result: it produces a \emph{well-conditioned} fit, whose axes are within a constant factor of those of the sphere.

\begin{theorem}[Well-conditioned fits below the threshold]\label{thm:conditioned}
For every $\alpha^* < 1/4$ there are constants $0 < \lambda_- \le \lambda_+ < \infty$, depending only on $\alpha^*$, with the following property.  If $\limsup_{d\to\infty} n/d^2 \le \alpha^*$, then with probability tending to one there exists an ellipsoid fit $S$ with
\[
\lambda_- \Id \;\preceq\; S \;\preceq\; \lambda_+ \Id .
\]
The fit may moreover be chosen with $\Tr S = d$.
\end{theorem}

 The replica analysis of Maillard and Kunisky \cite{maillard2024fitting} predicts that for every fixed $\alpha <1/4$, the typical solution (drawn uniformly over the set of feasible fits) has a spectral density supported on a compact subset of $(0,\infty)$. Theorem~\ref{thm:conditioned} establishes the existence of such a well-conditioned fit but does not characterize a typical solution.

As pointed out by Bandeira, Maillard, Mendelson, and Paquette \cite{bandeira2024fitting}, the ellipsoid fitting problem admits an equivalent dual formulation in terms of a semidefinite program, which seems interesting in its own right. When the samples span $\R^d$, feasibility of \eqref{eq:P} is complementary to the existence of a balanced positive definite combination of random rank-one matrices. Theorem \ref{thm:main} therefore also establishes the corresponding threshold for this dual problem.

\begin{corollary}[Threshold for balanced positive definite combinations]\label{cor:alternative}
Let $x_1,\dots,x_n \stackrel{\mathrm{iid}}{\sim} \cN(0,I_d)$ and consider the event
\begin{equation}\label{eq:alt-cor-event}
\cE_{n,d} \;=\; \Bigl\{\exists\, y \in \R^n:\ \textstyle\sum_{i=1}^n y_i = 0\ \text{ and } \ \sum_{i=1}^n y_i\, x_i x_i^{\top} \succ 0 \Bigr\}.
\end{equation}
If $\limsup n/d^2 < 1/4$ then $\Prob(\cE_{n,d}) \to 0$; if $\liminf n/d^2 > 1/4$ then $\Prob(\cE_{n,d}) \to 1$.
\end{corollary}

Alternative \eqref{eq:alt-cor-event} is the classical Gordan--Stiemke duality for the self-dual cone $\Sdp$ (see, e.g., \cite[Section~5.8]{boyd2004convex} for background); we include a short proof of this equivalence in Section \ref{sec:cor-alternative} for reader's convenience (see also \cite[Corollary 1.3]{bandeira2024fitting}).

The proof of Theorem \ref{thm:main} builds on the Gaussian-comparison approach of Bandeira and Maillard \cite{bandeira2025exact}. We compare the ellipsoid fitting problem \eqref{eq:P} with a surrogate problem obtained by partially replacing the data matrices $x_ix_i^\top - \Id$ with covariance-matched Gaussian matrices $G_i \sim \GOE d$. The resulting problem can then be analyzed using Gordon's min-max and escape theorems.  On the satisfiable side, we avoid the difficulties faced by previous attempts, which aimed to construct an ellipsoid fit explicitly, by working directly in the dual of a strengthened version of \eqref{eq:P}. This reduces the proof of satisfiability to showing that the minimum over the dual unit sphere is bounded away from zero. On the unsatisfiable side, we remove the norm constraint imposed in \cite{bandeira2025exact} by conditioning on the low-rank spikes of the candidate matrix and Gaussianizing only its bulk. Section~\ref{sec:heuristic} discusses the motivation for this approach and the underlying Gaussian comparison heuristic, while Section~\ref{sec:proof-method} gives a high-level overview of the proof.

\subsection{Origins and prior progress on the conjecture}
\label{sec:origins}

We give here a brief overview of the origin and earlier progress on the ellipsoid fitting conjecture. We refer to \cite{potechin2023near,maillard2024fitting,bandeira2025exact} for more detailed accounts.

The conjecture originated in work of Saunderson and coauthors \cite{saunderson2011subspace,saunderson2012diagonal,saunderson2013diagonal} on \emph{minimum trace factor analysis} (MTFA) for decomposing a matrix $X = D +L$, with $D$ diagonal and $L \succeq 0$ low-rank. They showed that $U = {\rm range}(L) \in \R^{N \times r}$ is recoverable by MTFA if and only if its orthogonal complement $U^\perp$ admits a matrix $V \in \R^{N \times (N-r)}$, with $ {\rm range} (V) = U^\perp$, whose columns satisfy the ellipsoid fitting property \cite[Proposition 3.1]{saunderson2012diagonal}. Thus, for a Haar-random rank-$r$ subspace $U$, Theorem \ref{thm:main} implies that MTFA succeeds with high probability if $ N - r \geq(2 + \eps) \sqrt{N}$ and fails with high probability if $N - r \leq (2 - \eps) \sqrt{N}$.

Ellipsoid fitting has since been connected to several other problems in statistics and theoretical computer science. Saunderson et al. \cite{saunderson2012diagonal} and Potechin et al. \cite{potechin2023near} showed that it is equivalent to \emph{vector discrepancy}, a canonical SDP relaxation of discrepancy ${\rm disc} (A) =\min_{\sigma \in \{\pm 1\}^n} \| A \sigma \|_\infty$ \cite{nikolov2013komlos}. For a Gaussian matrix $A \in \R^{m \times N}$, the probability that this SDP has value zero equals the probability that $N$ Gaussian vectors in $\R^{N-m}$ admits an ellipsoid fit \cite[Theorem 38]{potechin2023near}. Thus, when $N - m \geq (2+\eps)  \sqrt{N}$, Theorem \ref{thm:main} implies that the SDP fails with high probability to certify a nontrivial discrepancy lower bound. Podosinnikova et al. \cite{podosinnikova2019overcomplete} found a related connection in overcomplete \emph{independent component analysis} (ICA), showing that an SDP-based algorithm succeeds when a variant of ellipsoid fitting is feasible. Although its points are neither identically distributed nor constrained to have equal radii, and hence fall outside \eqref{eq:P}, their experiments suggest a similar threshold \(n\sim d^2/4\).  Finally, Ghosh et al. \cite{ghosh2020sum} studied sum-of-squares lower bounds for the Sherrington--Kirkpatrick Hamiltonian; their pseudocalibration analysis of the planted-affine-planes problem implies ellipsoid fitting for \(n\leq d^{3/2-\eps}\). We refer to \cite{potechin2023near} for further connections to the literature.

Most prior progress on showing the satisfiable side of Conjecture \ref{conj:main} has relied on exhibiting an explicit candidate solution $S^\star$ to the linear system of equations \eqref{eq:P} and proving $S^\star \succeq 0$ with high probability when $n/d^2$ is small enough, a delicate random matrix problem. Saunderson et al.'s least-squares analysis gave feasibility for $n \leq d^{6/5 - \eps}$ \cite{saunderson2011subspace,saunderson2013diagonal}; the planted-affine-planes machinery gave \(n \leq d^{3/2-\eps}\) \cite{ghosh2020sum}; Potechin, Turner, Venkat and Wein analyzed the minimum-Frobenius-norm, or pseudoinverse, interpolant and reached \(n \leq d^2/\mathrm{polylog}\, d\) \cite{potechin2023near}; and Kane and Diakonikolas obtained the explicit scale $n \leq c d^2/\log^4 d$ with an identity-perturbation construction \cite{kane2023nearly}. Feasibility at a positive density \(n \leq c d^2\) was then proved in three independent works: by Hsieh, Kothari, Potechin and Xu via graphical matrix decompositions \cite{hsieh2023ellipsoid}, by Bandeira, Maillard, Mendelson and Paquette via concentration of the Gram matrix of the quadratic feature vectors \cite{bandeira2024fitting}, and by Tulsiani and Wu via empirical covariance estimation \cite{tulsiani2025ellipsoid}. Throughout this development, the only bound on the unsatisfiable side remained the dimension count \(n > d(d+ 1)/2\).

Maillard and Kunisky developed a non-rigorous replica analysis of the ellipsoid fitting problem \cite{maillard2024fitting}. Their Claim 1 predicts a SAT/UNSAT transition of \eqref{eq:P} at $\alpha = n/d^2 = 1/4$ and for each fixed $\alpha <1/4$, a deterministic limiting spectral distribution in a compact interval $[\lambda_- (\alpha), \lambda_+(\alpha)] \subset (0,\infty)$ for the typical exact fit, that is, a fit drawn uniformly from the solution set. As $\alpha \uparrow 1/4$, the predicted law degenerates to the critical distribution $\mu_C = \tfrac{1}{2}\delta_0 + \tfrac{4 \sqrt{9\pi^2 - 4x^2}}{9\pi^3} \boldsymbol{1}_{[0,3\pi/2]} (x)\de x$, an atom of mass one half at the origin plus a quarter-circle bulk. At criticality, half of the semiaxes of the typical fitting ellipsoid diverge, and the fit flattens toward an elliptical cylinder. The same analysis predicts where canonical interpolants fail \cite[Claim 2]{maillard2024fitting}: the minimum-nuclear-norm solution is predicted to remain positive semidefinite throughout the satisfiable phase, whereas the minimum-Frobenius-norm solution and the minimizer of $\| S - \Id \|_F$ are predicted to lose positive semidefiniteness at $\alpha = 1/10$, and the minimizer of $\| S -\Id\|_\op$ near $\alpha \approx 0.1892$.  Our proof abandons explicit candidates altogether and certifies existence of an ellipsoid fit via a dual separation over the full solution set; see Section \ref{sec:proof-method}.

The nonrigorous replica analysis of \cite{maillard2024fitting} predicts that the free energy of the ellipsoid fitting problem \eqref{eq:P} is universal and coincides with that of a Gaussian surrogate model. Building on this prediction, Bandeira and Maillard \cite{bandeira2025exact} established a universality principle for the minimal fitting error and proved Conjecture \ref{conj:main} for an approximate version of the ellipsoid fitting problem. This Gaussian comparison heuristic and the analysis of \cite{bandeira2025exact} motivate our approach; we describe both in detail in the next section.

\subsection{The Gaussian heuristic and the statistical dimension}\label{sec:heuristic}

The sharp threshold \(n\sim d^2/4\) in Conjecture \ref{conj:main} is suggested by a Gaussian comparison heuristic \cite{bandeira2025exact}. It is useful to review this heuristic in detail, as it underlies our proof, whose main challenge is to bridge the gap between the Gaussian surrogate and the original problem.

Consider the centered, normalized rank-one measurements
\begin{equation}\label{eq:W-def-intro}
W_i:=\frac{x_ix_i^\top-\Id}{\sqrt d}.
\end{equation}
Then a matrix $S\in \sfS^d_+$ solves \eqref{eq:P} if and only if $\ip{W_i}{S}=(d-\Tr S)/\sqrt d$ for all $i$. Equivalently, when the samples span $\R^d$, \eqref{eq:P} is equivalent (up to rescaling) to the following linear system of equations in the space of symmetric matrices: 
\begin{equation}\label{eq:P-sym}
 \sfP_{\boldsymbol{1}^\perp} (\ip{W_i}{S})_{i =1}^n =0, \qquad S \succeq 0, \quad S \neq 0,
\end{equation}
where \(\sfP_{\boldsymbol{1}^\perp}\) denotes the orthogonal projection in $\R^n$ onto the complement of \(\boldsymbol{1}\). Wick's formula gives
\begin{equation}\label{eq:moment-match}
\E\ip{W_i}{A}=0,\qquad
\E\bigl[\ip{W_i}{A}\ip{W_i}{B}\bigr]=\frac{2}{d}\ip{A}{B}
\qquad(A,B\in\Sd).
\end{equation}
Thus, the first two moments of the random matrix $W_i$ agree with those of $G_i$, a Gaussian orthogonal ensemble matrix with $(G_i)_{ss} \sim \cN(0,2/d)$ and $(G_i)_{st} = (G_i)_{ts} \sim \cN(0,1/d)$ for $s< t$, which we denote $G_i\sim\GOE{d}$. Equivalently, $G_i = \sqrt{\tfrac{2}{d}}\Gamma_i$ where $\Gamma_i$ is a standard Gaussian vector in $\sfS^d$. 

Now replace the rows $W_i$ in \eqref{eq:P-sym} by independent $G_i \sim \GOE{d}$ matrices. The kernel of the constraints is a Haar-random subspace of $\sfS^d$ and problem \eqref{eq:P-sym} becomes: \emph{does a uniformly random linear subspace of $\sfS^d$ of codimension $n-1$ intersect the cone $\sfS^d_+$?} For such a genuinely isotropic question, conic integral geometry provides a complete answer \cite{bandeira2025exact}. Gordon's min-max and escape theorems \cite{gordon1985some,gordon1988milman}, the width calculus of convex programs \cite{chandrasekaran2012convex} and the phase transition results in \cite{amelunxen2014living} show that a uniformly random linear subspace of codimension $n$ meets a closed convex cone $C$ (with high probability, and sharply) if and only if $n$ is below the \emph{statistical dimension} $\delta(C) = \E \| \Pi_C Z \|_\F^2$, where $\Pi_C$ is the metric projection onto $C$ and $Z$ is a standard Gaussian in the ambient space. For the positive semidefinite cone, Moreau's decomposition and the symmetry $(\sfS^d_+)^\circ = - \sfS^d_+$ give exactly (see Lemma \ref{lem:psd-width})
\begin{equation}
  \delta(\sfS^d_+) = \frac{1}{2} \dim \sfS^d= \frac{d(d+1)}{4}= \frac{d^2}{4} + O(d).
\end{equation}
Hence the Gaussian surrogate problem has a sharp transition at $n = d^2/4$: the positivity constraint halves the number of degrees of freedom, moving the threshold from the dimension count $d(d+1)/2$ down to $d(d+1)/4$. Equivalently, in the language used throughout this paper, the Gaussian width of the spherical cap $\sfS_+^d \cap \sphereF$ (with $\sphereF$ the Frobenius unit sphere of $\sfS^d$) equals
\begin{equation}
  w(\sfS_+^d \cap \sphereF) = \E \sup_{S \in \sfS_+^d \cap \sphereF}\ip{Z}{S} = (1+o(1))\sqrt{\delta(\sfS^d_+)} = \frac{d}{2} + O(1),
\end{equation}
and Gordon's inequalities localize the transition where this width crosses $\E\|g_n\|_2 \approx \sqrt{n}$ where $g_n \sim \cN(0,I_n)$, that is, at $\sqrt{n} = d/2$.

Conjecture \ref{conj:main}, then, is a \emph{universality} statement: for this particular conic feasibility problem, the rank-one Gaussian-chaos measurements in \eqref{eq:W-def-intro} should exhibit the same sharp transition threshold as their Gaussian counterparts. This assertion is far from obvious: (1) the kernel of \eqref{eq:P-sym} is not uniformly oriented; (2) the coordinates of \(W_i\) are second-order Gaussian chaoses and are neither independent nor uniformly subgaussian; and (3) the conjecture concerns \emph{exact feasibility} and \emph{all possible witnesses}, requiring any comparison with the Gaussian model to hold uniformly over exponentially many directions.

A decisive step towards formalizing this approach was taken by Bandeira and Maillard \cite{bandeira2025exact}. They established a universality principle for the minimal fitting error: for suitable bounded losses $\varphi$, the ground-state energy 
\begin{equation}
\inf_{0 \preceq S \preceq M I_d}\frac{1}{n} \sum_{i=1}^n \varphi (\ip{X_i}{S} - \sqrt{d})
\end{equation}
 has the same asymptotic behavior as $d\to\infty$ for the rank-one chaos ensemble $X_i=W_i$ and its Gaussian surrogate $X_i=G_i$. Combined with a sharp analysis of the Gaussian problem using Gordon's inequalities, this principle yields the following result. If $\limsup n/d^2<1/4$, then, for every $r\in[1,4/3)$ and every fixed tolerance $\eps>0$, with high probability there exists a matrix $S$ whose spectrum lies in a fixed interval $[\lambda_-,\lambda_+]\subset(0,\infty)$ and such that $\tfrac{1}{n} \sum_i  |(x_i^\top S x_i - d)/\sqrt{d} |^r \le \eps$. Above the threshold, they establish, with high probability, a strictly positive loss gap over ${0\preceq S\preceq M I_d}$, thereby ruling out exact fits with bounded operator norm.
 
 These results fall short of Conjecture \ref{conj:main} in two respects. First, the limit $\eps\to0$ is taken only after $d\to\infty$, so the argument yields approximate rather than exact fitting. Second, the impossibility result applies only to matrices with bounded operator norm, whereas Conjecture \ref{conj:main} concerns all ellipsoids, without such a restriction. The present proof closes these two gaps. On the satisfiable side, we work directly with the dual problem, one of the possible routes suggested in \cite{bandeira2024fitting,bandeira2025exact} for resolving the full conjecture. On the unsatisfiable side, we condition on the low-rank spectral head and perform the Gaussian comparison only on the diffuse bulk.

\subsection{Proof method and prior work on universality}
\label{sec:proof-method}

We give below a high-level overview of the proof strategy behind Theorem \ref{thm:main}.

\paragraph*{The satisfiable phase.} Below the threshold, we may fix $n/d^2 \to \alpha \in (0,1/4)$ after padding. We consider a strengthening of the feasibility problem \eqref{eq:P} where we impose $\Tr S = d$ and restrict the condition number of $S$. Specifically, for a fixed spectral box $\widetilde \cB:=\left\{S\in\Sd:\tfrac{1}{2\kappa_1\sqrt{d}}\Id\preceq S \preceq \tfrac{2\kappa_1}{\sqrt{d}}\Id\right\}$, we prove that with high probability
\begin{equation}\label{eq:centered-P}
  0\in\cW(\cB), \qquad \cW(S)=\bigl(\ip{W_i}{S}\bigr)_{i=1}^n.
\end{equation}
Any $S \in \cB$ with $\cW(S) = 0$ has $x_i^\top S x_i = \Tr S > 0$ for all $i$ and therefore $\widehat S=dS/\Tr S$ solves \eqref{eq:P} with additional properties
\begin{equation}
  \frac{1}{4\kappa_1^2}\Id\preceq\widehat S \preceq4\kappa_1^2\Id, \qquad \Tr \widehat S = d.
\end{equation}
Now $\cW (\cB) \subseteq \R^n$ is compact and convex, so \eqref{eq:centered-P} fails if and only if an hyperplane separates it from the origin (see Lemma \ref{lem:convex-separation}): there would be a unit vector $y \in \R^n$ with $\sup_{S \in \widetilde \cB} \ip{\cW^* y}{S} < 0$ where $\cW^* y = \sum_i y_i W_i$. The proof therefore reduces to a uniform dual estimate: it is sufficient to show that with high probability, simultaneously for every unit $y \in \R^n$,
\begin{equation}\label{eq:dual-uniform-estimate}
  h(y) := \sup_{S \in \widetilde \cB}  \sum_{i=1}^n y_i\ip{W_i}{S} >0.
\end{equation} 
This modifies the dual matrix alternative \eqref{eq:alt-cor-event} in two ways. First, centering eliminates the constraint $y^\top \boldsymbol 1 = 0$. Second, restricting to the spectral constraint $S \in \widetilde \cB$ provides enough slack to perturb $S$ to resolve a sparse set of constraints exactly without leaving the positive cone $\sfS^d_+$.

For Gaussian rows, \eqref{eq:dual-uniform-estimate} follows directly from Gordon's min--max margin inequality:
\begin{equation}\label{eq:dual-uniform-estimate-surrogate}
  \inf_{\| y \|_2 = 1} \sup_{S \in \widetilde \cB \cap  \sphereF} \sum_{i=1}^n y_i \sqrt d \ip{G_i}{S} \gtrsim w (\widetilde \cB \cap \sphereF) -  \E \| g_n \|_2 - t
\end{equation}
with probability $1 - e^{-t^2 /C}$. Hence, the margin is positive with high probability whenever $w (\widetilde \cB \cap \sphereF) > \sqrt{n}= d\sqrt{\alpha}$, which happens for sufficiently large constant $\kappa_1$ (only depending on $\alpha$) since $w (\widetilde \cB \cap \sphereF) / d \to w(\sfS_+^d \cap \sphereF)/d \sim 1/2 > \sqrt{\alpha}$ as $\kappa_1 \to \infty$ (see Lemma \ref{lem:widthcone}).

The main difficulty resides in transferring the Gordon margin from $G_i$ to $W_i$ \emph{uniformly in $y$}. Indeed, when $y$ is supported on only a few coordinates, the law of $\sum_i y_i \ip{W_i}{S}$ can differ substantially from that of its Gaussian counterpart $\sum_i y_i\ip{G_i}{S}$, and one cannot hope for a uniform comparison over all unit vectors $y$.  We therefore decompose each $y$ into a sparse head $J$ of at most $C_0 d \log d$ heavy coordinates and a low-influence tail satisfying $\| y_{J^c} \|_\infty \leq \theta \| y_{J^c} \|_2 /\sqrt{d}$. We enfore the head constraints exactly 
\begin{equation}
\cW_J (S) = (\ip{W_j}{S})_{j \in J} = 0,
\end{equation}
 and compare \eqref{eq:dual-uniform-estimate} with \eqref{eq:dual-uniform-estimate-surrogate} only over the restricted set 
 \begin{equation}
 S \in \widetilde \cB \cap {\rm ker} \cW_j \cap \sphereF.
 \end{equation} Restricting the supremum can only decrease $h(y)$. Proposition \ref{prop:heads} shows that this restriction reduces the Gaussian width by only \(o(d)\). The key mechanism is a deterministic correction based on the inverse Gram matrix of the second-order chaos, whose singular values remain of order \(d\) throughout the regime \(n/d^2<1/2\). This follows from a sharp two-sided spectral-edge result for polynomial-scaling kernel random matrices, recently established by Kogan, Namdy, and Huang \cite{kogan2025extremal}. For the low-influence tail, we develop a Lindeberg principle for second-order chaos (Proposition \ref{lem:interpolation}) and apply it to a smooth approximation of the min--max problem in \eqref{eq:dual-uniform-estimate-surrogate}. The resulting error scales as $\sum_{i \in J^c} |y_i|^3\leq \theta /\sqrt{d}$. The remainder of the proof combines a delicate random-net argument over the possible heads with a truncation of the quadratic chaos, yielding \(e^{-cd^2}\) concentration strong enough to take a uniform bound over all possible heads.

\paragraph*{The unsatisfiable phase.} Above the threshold, we may fix $n/d^2 \to \gamma \in (1/4,1/2)$ after deleting observations if necessary. Following the strategy of \cite{bandeira2025exact}, we prove that, with high probability, no exact ellipsoid fit exists by showing that the fitting error is uniformly bounded away from zero over all admissible $S$. To this end, define the empirical risk
\begin{equation}\label{eq:loss-intro}
  \widetilde \Loss_n(S,b):=\frac1n\sum_{i=1}^n\varphi(x_i^\top S x_i - \Tr S - b), \qquad \varphi(t)=1-e^{-t^2}.
\end{equation}
In particular, if there exists an \(S\) satisfying \eqref{eq:P}, then \(  \widetilde \Loss_n(S',b)=0\) for some \(S'\succeq0\) with \(\|S'\|_F=1\) and some \(b\in\mathbb R\). Moreover, with high probability, we may restrict attention to \(|b|\le C\), where \(C>0\) is a sufficiently large constant. It therefore suffices to prove that, with high probability, \(  \widetilde \Loss_n(S,b)\) is uniformly bounded away from zero over all \(S\succeq0\) satisfying \(\|S\|_F=1\) and all \(|b|\le C\).

The argument of \cite{bandeira2025exact} establishes such a uniform positive gap under the additional restriction\footnote{We note that \cite{bandeira2025exact} used a different normalization. The particular restriction \(\|S\|_{\op}\le d^{-1/4}\) is obtained by adapting their argument to the homogeneous formulation and the Frobenius-net argument used in our proof.}
\(\|S\|_{\op}\le d^{-1/4}\). It does so by comparing \eqref{eq:loss-intro} with the empirical risk in a Gaussian surrogate model, in which $x_i^\top S x_i - \Tr S$ is replaced by $\sqrt{d} \ip{G_i}{S}$. Here, we extend this comparison to the full class of matrices
\(S\succeq0\) with \(\|S\|_F=1\), without imposing an operator-norm bound. The key idea is to condition on the low-rank spiked component containing the eigendirections whose eigenvalues exceed \(d^{-1/4}\).  More precisely, let \(\varepsilon_d=d^{-1/4}\) and decompose \(S\) into its spectral head and bulk:
\[
H:=S\mathbf 1_{(\varepsilon_d,\infty)}(S),
\qquad
B:=S-H.
\]
Since \(\|S\|_F=1\), the head has rank at most \(\varepsilon_d^{-2}=d^{1/2}\). Meanwhile, the bulk satisfies 
\[
\|B\|_{\op}\le\varepsilon_d=d^{-1/4},
\qquad
\Tr(B^3)\le\varepsilon_d \|B\|_F^2\le d^{-1/4},
\]
and is diffuse in Schatten \(3\)-norm. Writing \((g_i,z_i)\) for the components of \(x_i\) in the subspaces associated with \(H\) and \(B\), respectively, we obtain the decomposition
\begin{equation}
  x_i^\top Sx_i-\Tr S-b = 
  g_i^\top Hg_i-\Tr H-b
  +z_i^\top Bz_i-\Tr B.
\end{equation}
We condition on the head variables \(g_i\) and introduce a hybrid model in which $z_iz_i^\top - \Id_m$ is replaced by $\sqrt{m} G_i$, where $m = {\rm rank} (B)$ and $G_i \sim \GOE{m}$. The corresponding empirical risk is
\begin{equation}\label{eq:loss-intro-hybrid}
    \widetilde \Loss_n^h (H,B,b):=\frac1n\sum_{i=1}^n\varphi(g_i^\top H g_i - \Tr H  - b + \sqrt{m} \ip{G_i}{B}).
\end{equation}

An adaptation of the free-entropy interpolation principle of
\cite{bandeira2025exact} shows that the minimum empirical risk in the original model is close, in conditional expectation, to that in the hybrid model; see Proposition~\ref{prop:translated-univ}. Crucially, this comparison is uniform over all possible values of the offsets contributed by the head variables. Gordon’s escape-through-a-mesh theorem then yields a uniform positive lower bound on the hybrid empirical risk \eqref{eq:loss-intro-hybrid} over all admissible triples \((H,B,b)\). Finally, we transfer this positive risk gap from the hybrid model back to the original model \eqref{eq:loss-intro}. We first carry out the comparison on a suitable net of head components. McDiarmid’s inequality upgrades the conditional-expectation estimate to an event of probability at least \(1-e^{-cd^2}\), which is strong enough to take a union bound over a net of \(e^{O(d^{3/2})}\) possible heads. Lipschitz continuity, together with the feature-edge bound of \cite{kogan2025extremal}, then extends the conclusion from the net to the entire admissible class.

\paragraph*{Relation to prior work on universality.} Universality principles replace a complicated random design by a Gaussian design with matching two first moments. For data with independent entries, Bayati et al.~\cite{bayati2015universality} proved universality of a certain phase transition arising in polytope geometry and compressed sensing, and Oymak and Tropp \cite{oymak2018universality} proved universality laws for restricted singular values and embedding thresholds for a family of randomized dimension reduction maps. A more recent line of work, going beyond independent entries, establishes Gaussian equivalence for high-dimensional empirical risk minimization and random-feature models under suitable pointwise or uniform central-limit hypotheses \cite{goldt2020modeling,hu2022universality,montanari2022universality,goldt2022gaussian,han2023universality,montanari2023universality,gerace2024gaussian,hu2024asymptotics}. Recent work also identifies mechanisms by which such universality can break down in structured classification and convex
ERM \cite{pesce2023gaussian,mai2025breakdown,yaakoubi2026characterization}. Bandeira and Maillard \cite{bandeira2025exact} adapted a free-entropy interpolation argument to the rank-one quadratic ensemble and proved universality over spectrally bounded candidate sets. Related matrix-valued universality results in the quadratic scaling appear in \cite{xu2025fundamental,erba2026nuclear}.

Our conditional decompositions are more directly inspired by the recent work of the authors jointly with Hu, Lu, and Fan on random-feature ERM in the quadratic regime $n,p\asymp d^2$ \cite{wen2025does}, where degree-two Wiener-chaos components dominate. At this scaling, a fully Gaussian surrogate can fail when the target depends on a fixed low-dimensional projection of the covariates; the correct hybrid surrogate retains the associated low-dimensional non-Gaussian component and Gaussianizes only its high-dimensional complement \cite{wen2025does}. Although the ellipsoid fitting problem is quite different, both parts of our proof follow the same conditional-Gaussianization principle: isolate the structured directions that obstruct a direct Gaussian comparison and Gaussianize only the diffuse remainder. In the satisfiable phase, these exceptional directions correspond to the large coordinates of the dual vector $y$; we enforce the associated constraints exactly and apply Gaussian comparison only to the low-influence tail. In the unsatisfiable phase, the structured component is instead a candidate-dependent low-rank spectral head of $S$; we condition on its contribution and only Gaussianize the Schatten--$3$ diffuse bulk.

\subsection{Notation and organization}\label{sec:notation}

For $p\ge1$, $\sfS^p$ denotes the space of real symmetric $p\times p$ matrices and $D_p=p(p+1)/2$ its dimension.  It is equipped with the Frobenius inner product $\ip{A}{B}=\Tr(AB)$ and norm $\|\cdot\|_{\F}$; $\Symp{p}$ is its positive semidefinite cone.  We write $\|\cdot\|_{\op}$ and $\|\cdot\|_*$ for the operator and nuclear norms, and $\Tr|A|^k=\sum_j|\lambda_j(A)|^k$.  The set $\sphereF=\{S\in\Sd:\|S\|_{\F}=1\}$ is the Frobenius unit sphere, and $\psdsphere{p}=\Symp{p}\cap\{S:\|S\|_{\F}=1\}$. A standard Gaussian matrix $Z\in\Sym{p}$ has independent $\cN(0,1)$ coordinates in a (Frobenius) orthonormal basis.  Thus $Z_{aa}\sim \cN(0,1)$ and $Z_{ab}\sim \cN(0,1/2)$ for $a<b$.  We use the normalization $\GOE{p}=\sqrt{2/p}\,Z$, which corresponds to the standard Gaussian Orthogonal Ensemble (GOE).  

For $J\subseteq[n]$, set $J^c=[n]\setminus J$.  If $v=(v_i)_{i\in[n]}$ is any scalar-, vector-, or matrix-valued indexed family, then $v_J=(v_i)_{i\in J}$ denotes its restriction; in particular, $X_J$ is the matrix with columns $(x_i)_{i\in J}$. Constants $c,C>0$ are absolute unless subscripted and may change from line to line.  The notation $o(\cdot),O(\cdot)$ refers to the asymptotic variable of the statement---usually $d$, and $p$ or $m$ in dimension-generic results.  Likewise, \emph{with high probability} means
with probability $1-o(1)$ in that variable.  Throughout,
\[
W_i=\frac{x_ix_i^\top-\Id}{\sqrt d},
\qquad
\cQ_x(S)=\bigl(x_i^\top Sx_i-\Tr S\bigr)_{i=1}^n =\sqrt d \, \cW (S) = \sqrt d\,\bigl(\ip{W_i}{S}\bigr)_{i=1}^n.
\]

Section~\ref{sec:prelim} develops the Gaussian-width, comparison, feature-edge, and data estimates shared by the two arguments. Section~\ref{sec:sat} proves the satisfiable phase, and Section~\ref{sec:unsat} proves the unsatisfiable phase.  The two proof sections may be read independently after Section~\ref{sec:prelim}.

\section{Preliminaries}\label{sec:prelim}

This section collects tools shared by the two proofs.  We compute the Gaussian width of the positive semidefinite cone and state the two forms of Gordon's inequality used later.  We then establish the two-sided feature edge and the required Gaussian data estimates. Finally, we introduce the homogeneous formulation of \eqref{eq:P} and prove the dual matrix alternative behind Corollary \ref{cor:alternative}.

\subsection{The width of the positive semidefinite cone}\label{sec:width}

Recall that for a closed convex cone $C$ in a Euclidean space $E$ and a standard Gaussian $Z$ on $E$, the statistical dimension is $\delta(C) = \E \|\Pi_C Z\|^2$, where $\Pi_C$ is the metric projection onto $C$, and the Gaussian width is defined as
\begin{equation}
 w(C \cap \S_E ) = \E \sup_{x \in C \cap \S_E } \ip{Z}{x} ,
\end{equation}
where $\S_E$ is the unit sphere in $E$. 

Recall that $\sfS^d$ denotes the space of $d \times d$ real symmetric matrices and $\Sdp$ its positive semidefinite cone. Set $D_d = \dim \sfS^d = d (d+1)/2$ and let $\psdsphere{d} = \Symp{d} \cap \{S : \|S\|_{\F} = 1\}$ denote the unit Frobenius sphere in $\Symp{d}$.  We first recall the classical Gaussian-width calculation for $\sfS^d_+$ \cite{chandrasekaran2012convex,amelunxen2014living,bandeira2025exact}. We include a proof for completeness, based on \cite[Proposition~10.2]{amelunxen2014living}.

\begin{lemma}[Gaussian width of $\sfS_+^d$]\label{lem:psd-width}
Let $d \ge 1$, let $Z$ be a standard Gaussian in $\Sym{d}$, and write $Z = Z_+ - Z_-$ for its spectral decomposition into positive and negative parts.  Then,
\begin{equation}\label{eq:statdim}
\delta(\Symp{d}) = \E \|Z_+\|_{\F}^2 = \frac{D_d}{2} = \frac{d(d+1)}{4} .
\end{equation}
Consequently, 
\begin{equation}\label{eq:width-upper}
\sqrt{\frac{D_d}{2} - 1}
\le w\bigl(\psdsphere{d}\bigr) \le \sqrt{\frac{D_d}{2}}\,,
\qquad
w\bigl(\psdsphere{d} \bigr) = \frac d2 + O(1) .
\end{equation}
\end{lemma}

\begin{proof}
 The metric projection of a symmetric matrix onto $\Symp d$ is its positive part.  Moreau's decomposition applied to the cone $C = \Symp d$ and its polar $C^\circ = -\Symp d$ gives $Z = \Pi_C Z + \Pi_{C^\circ} Z$ with $\ip{\Pi_C Z}{\Pi_{C^\circ}Z} = 0$, whence $\delta(C) + \delta(C^\circ) = \E\|Z\|_{\F}^2 = D_d$.  By the symmetry $Z \overset{d}{=} -Z$ we get $\delta(C^\circ) = \delta(-C) = \delta(C)$, and \eqref{eq:statdim} follows.  Proposition 10.2 in~\cite{amelunxen2014living} states that for every convex cone $C$, the Gaussian width and the statistical dimension are related as $w (C  \cap \S_E )^2 \leq \delta (C) \leq w(C  \cap \S_E)^2 +1$.  Applying this to $C = \Symp d$ gives the bounds in \eqref{eq:width-upper}.
\end{proof}

Lemma~\ref{lem:psd-width} is the source of the constant $1/4$.  The cone $\sfS^d_+$ has width $d/2+O(1)$ and Gordon's inequalities compare this width with $\sqrt n$, giving the critical scale $n=d^2/4$.

In the satisfiable phase, we will further restrict the cap $\psdsphere{d}$ to \emph{strictly} positive definite matrices with uniformly bounded condition number. For $\kappa \ge 1$ define the closed convex cone
\begin{equation}\label{eq:cone-def}
\cK_\kappa \;=\; \bigl\{S \in \Sdp :\ \lambda_{\max}(S) \le \kappa\,
\lambda_{\min}(S)\bigr\} .
\end{equation}
Every nonzero element of $\cK_\kappa$ is positive definite. 

The following result was proved in \cite[Lemma 3.6]{bandeira2025exact}. It shows that the Gaussian width of $\psdsphere{d}$ restricted to $\cK_\kappa$ becomes arbitrarily close, at leading order, to $w\bigl(\psdsphere{d}\bigr)$ as $\kappa \to \infty$.  We include a quantitative proof because it also bounds the condition number of the fit constructed in Section \ref{sec:sat}; see Remark~\ref{rem:kappa-blowup}.

\begin{lemma}[Width of bounded-condition positive cones]\label{lem:widthcone}
For every $a < 1/2$, there is a fixed $\kappa = \kappa(a) < \infty$ such that
\[
w\bigl(\cK_\kappa \cap \sphereF\bigr) \ge a d
\]
for all sufficiently large $d$.
\end{lemma}

\begin{proof}
The construction shifts the positive part of a Gaussian matrix by a small multiple of the identity.  The shift enforces a bounded condition number while losing arbitrarily little width.

Let $Z$ be a standard Gaussian in $\Sd$ and $Z_+$ its positive part. Choose once and for all a fixed $C > \sqrt2$, strictly above the limiting operator-norm edge for this normalization ($\|Z\|_{\op}/\sqrt d \to \sqrt 2$ almost surely; see, e.g., \cite[Theorem~2.1.22]{anderson2010introduction}).  Fix $s > 0$ and, on the event
$\{\|Z\|_{\op} \le C\sqrt d\}$, set
\[
Q_s = Z_+ + s \sqrt d\, \Id, \qquad P_s = \frac{Q_s}{\|Q_s\|_{\F}} .
\]
Then $P_s \in \cK_{\kappa_s} \cap \sphereF$ with $\kappa_s = (C + s)/s$.  The Wigner semicircle law and the operator-norm convergence give, in probability,
\begin{equation}\label{eq:semicircle-facts}
\frac{\|Z_+\|_{\F}^2}{d^2} \longrightarrow \frac14, \qquad
\frac{\Tr Z_+}{d^{3/2}} \longrightarrow c_+ := \frac{2\sqrt2}{3\pi},
\qquad
\frac{\Tr Z}{d^{3/2}} \longrightarrow 0 ,
\end{equation}
where $c_+$ is the first moment of the positive half of the semicircle law of radius $\sqrt 2$.  The value of the candidate $P_s$ is exactly
\begin{equation}\label{eq:cand}
\ip{Z}{P_s} = \frac{\|Z_+\|_{\F}^2 + s\sqrt d\, \Tr Z}
{\bigl(\|Z_+\|_{\F}^2 + 2 s \sqrt d\, \Tr Z_+ + s^2 d^2\bigr)^{1/2}} ,
\end{equation}
using $\ip{Z}{Z_+} = \|Z_+\|_{\F}^2$ and $\ip{Z}{\Id} = \Tr Z$.  Hence, for fixed $s > 0$, the quantity \eqref{eq:cand} divided by $d$ converges in probability to
\begin{equation}\label{eq:qs-def}
q(s) := \frac{1/4}{\sqrt{\,1/4 + 2 c_+ s + s^2\,}} .
\end{equation}
The absolute value of \eqref{eq:cand} is at most $\|Z\|_{\F}$, and $\sup_d \E (\|Z\|_{\F}/d)^2 < \infty$, so the candidates are uniformly integrable.  On $\{\|Z\|_{\op} > C\sqrt d\}$, use instead the fixed element $\Id/\sqrt d \in \cK_{\kappa_s} \cap \sphereF$; since the probability of that event tends to zero, $\E[\|Z\|_{\F} \mathbf 1\{\|Z\|_{\op} > C\sqrt d\}] = o(d)$ by Cauchy--Schwarz.  We conclude
\begin{equation}\label{eq:width-liminf}
\liminf_{d\to\infty} \frac 1d\,
\E \sup_{P \in \cK_{\kappa_s} \cap \sphereF} \ip{Z}{P}
\; \ge \; q(s) .
\end{equation}
The right-hand side of \eqref{eq:width-liminf} tends to $1/2$ as $s \downarrow 0$.  Given $a < 1/2$, choose a fixed $s > 0$ small enough that $q(s) > a$, and take $\kappa = \kappa_s = (C+s)/s$.
\end{proof}

\subsection{Gordon's comparison inequalities}\label{sec:gordon}

We use Gordon's Gaussian min--max theorem \cite{gordon1985some,gordon1988milman} in two guises: a lower tail for min--max statistics (used on the satisfiable side) and the escape-through-a-mesh estimate (used on the unsatisfiable side). Convexity of the index set is required in neither.  We give a full proof of the first, in the precise form needed; the second is classical.

\begin{lemma}[Gordon min--max lower tail]\label{lem:gordon}
Let $E$ be a Euclidean space of finite dimension, let $T \subseteq E$ be a nonempty compact subset of the unit sphere of $E$, let $\Gamma_1, \dots, \Gamma_N$ be independent standard Gaussian vectors in $E$, and let $g_N \sim \cN(0, I_N)$.  Then for every $t > 0$,
\begin{equation}\label{eq:gordon-minmax}
\Prob\left\{
\inf_{\|z\|_2 = 1}\, \sup_{R \in T}\,
\sum_{i=1}^N z_i \ip{\Gamma_i}{R} \; \le\;  w(T) - \E\|g_N\|_2 - t
\right\} \le 2\, e^{-t^2/4} .
\end{equation}
\end{lemma}

\begin{proof}
The proof compares finite nets after adding one scalar Gaussian to equalize the variances.  We then pass to the compact index sets and remove the added scalar by symmetry.

Identify $E$ with $\R^M$ through an orthonormal basis and let $A$ be the $N \times M$ matrix whose $i$-th row is $\Gamma_i$.  Introduce an independent scalar $\gamma_0 \sim \cN(0,1)$ and, for $z \in \mathbb{S}^{N-1}$, $R \in T$, consider the two centered Gaussian processes
\[
X_{z,R} = z^\top A R + \gamma_0,
\qquad
Y_{z,R} = \ip{h}{R} + \ip{g_N}{z},
\]
where $h \sim \cN(0, I_M)$ and $g_N$ are independent.  Both processes have variance $2$.  If $a = \ip{z}{z'}$ and $b = \ip{R}{R'}$, then
\begin{equation}\label{eq:gordon-cov}
\E X_{z,R} X_{z',R'} - \E Y_{z,R} Y_{z',R'}
= 1 + ab - a - b = (1-a)(1-b) \;\ge\; 0 ,
\end{equation}
with equality whenever $z = z'$.  We use the following finite-index form of Gordon's comparison theorem \cite{gordon1985some} (see also \cite[Appendix A]{thrampoulidis2014gaussian} for additional background): if $(U_{ij})$ and $(V_{ij})$ are finite centered Gaussian arrays with equal coordinate variances such that
\[
\E U_{ij} U_{ik} \le \E V_{ij} V_{ik},
\qquad
\E U_{ij} U_{\ell k} \ge \E V_{ij} V_{\ell k} \quad (i \neq \ell),
\]
then $\Prob\{\min_i \max_j U_{ij} < u\} \le \Prob\{\min_i \max_j V_{ij} < u\}$ for every $u$.  Apply this with $U = X$, $V = Y$ on finite nets of $\S^{N-1}$ and $T$, the net points of $\S^{N-1}$ playing the role of $i$: by \eqref{eq:gordon-cov} the covariance difference vanishes for a common minimizing index ($a = 1$) and is nonnegative for distinct ones, which is exactly the required pair of inequalities.  Thus the finite-net min--max values $P_{X,k}, P_{Y,k}$ (along a sequence of successively finer deterministic nets, indexed by $k$) satisfy the stated comparison.

Uniform sample-path continuity on the compact product gives $P_{X,k}\to P_X$ and $P_{Y,k}\to P_Y$ almost surely, where $P_X=\inf_z\sup_R X_{z,R}$ and similarly for $P_Y$.  Passing the finite-net comparison first through $k\to\infty$ and then through thresholds $u+\varepsilon\downarrow u$ gives
\begin{equation}\label{eq:gordon-compact}
\Prob\Bigl\{\inf_z \sup_R X_{z,R} \le u\Bigr\}
\;\le\;
\Prob\Bigl\{\inf_z \sup_R Y_{z,R} \le u\Bigr\}
\qquad (u \in \R).
\end{equation}
The auxiliary min--max computes exactly:
\[
\inf_{\|z\|_2=1} \sup_{R\in T}
\bigl(\ip{h}{R} + \ip{g_N}{z}\bigr)
= \sup_{R \in T} \ip{h}{R} - \|g_N\|_2 ,
\]
whose expectation is $w(T) - \E\|g_N\|_2$.  As a function of the pair $(h, g_N)$ it is $\sqrt2$-Lipschitz, so its lower deviations at distance $t$ have probability at most $e^{-t^2/4}$ by the Gaussian concentration inequality \cite[Theorem~5.6]{boucheron2013concentration}.  Finally, if $P(A) = \inf_z \sup_R z^\top A R$, then $P_X=P(A)+\gamma_0$. Independence and symmetry of $\gamma_0$ give
\[
\tfrac12 \Prob\{P(A) \le u\}
= \Prob\{P(A) \le u,\ \gamma_0 \le 0\}
\le \Prob\{P(A) + \gamma_0 \le u\} ,
\]
and combining with \eqref{eq:gordon-compact} at $u = w(T) - \E\|g_N\|_2 - t$ proves \eqref{eq:gordon-minmax}.
\end{proof}

\begin{lemma}[Escape through a mesh {\cite{gordon1988milman}}]\label{lem:escape}
Let $A$ be an $s \times M$ matrix with independent $\cN(0,1)$ entries and let $T$ be a nonempty compact subset of the unit sphere of $\R^M$.  Then for every $t > 0$,
\begin{equation}\label{eq:gordon-escape}
\Prob\left\{
\inf_{x \in T} \|A x\|_2 \;\le\; \E\|g_s\|_2 - w(T) - t
\right\}
\le e^{-t^2/2} ,
\qquad g_s \sim \cN(0, I_s).
\end{equation}
\end{lemma}

Lemma~\ref{lem:escape} also follows from the argument used for Lemma~\ref{lem:gordon}, with the roles of the two index families exchanged, at the cost of slightly worse constants; we use the classical form \eqref{eq:gordon-escape} as stated.  We will use repeatedly the elementary bounds
\begin{equation}\label{eq:chi-mean}
\sqrt{s-1} \;\le\; \E\|g_s\|_2 \;=\;
\sqrt2\, \frac{\Gamma((s+1)/2)}{\Gamma(s/2)} \;\le\; \sqrt s .
\end{equation}

\begin{lemma}[Rectangular Gaussian matrices; see
{\cite[Theorem~7.3.1]{vershynin2018high}}]\label{lem:gauss-rect}
If $A \in \R^{N \times M}$ has independent $\cN(0, \sigma^2)$ entries, then for every $t \ge 0$,
\begin{equation}\label{eq:gauss-rect}
\Prob\bigl\{\|A\|_{\op} > \sigma(\sqrt N + \sqrt M + t)\bigr\}
\le 2 e^{-t^2/2} .
\end{equation}
\end{lemma}

\subsection{The spectral edge of the quadratic feature map}\label{sec:feature}

The two main proofs use the quadratic feature map both to solve a small set of constraints and to control perturbations.  Both tasks require its singular values to remain of order $d$.  We prove this uniformly throughout $n/d^2<1/2$, a range containing the feasibility transition at $1/4$.

For $p \ge 1$ and i.i.d.\ $y_1, \dots, y_n \sim \cN(0, I_p)$, define
\begin{equation}\label{eq:Q-def}
\cQ_y : \Sym p \to \R^n, \qquad
\cQ_y(T) = \bigl(y_i^\top T y_i - \Tr T\bigr)_{i=1}^n ,
\end{equation}
and let $\cQ_y \cQ_y^*$ denote its $n \times n$ Gram matrix, $(\cQ_y\cQ_y^*)_{ij} = (y_i^\top y_j)^2 - \|y_i\|_2^2 - \|y_j\|_2^2 + p$. In the Frobenius-orthonormal basis $E_{aa}=e_ae_a^\top$ and $E_{ab}=(e_ae_b^\top+e_be_a^\top)/\sqrt2$ for $a<b$, the row-coordinate matrix of $\cQ_y$ is $[D\ V]$, where
\begin{equation}\label{eq:DV-def}
D_{ia} = y_{ia}^2 - 1 \quad (a \le p),
\qquad
V_{i, (a,b)} = \sqrt2\, y_{ia} y_{ib} \quad (a < b).
\end{equation}
Thus $\cQ_y\cQ_y^*=DD^\top+VV^\top$.  The next two lemmas control the square-free block from both sides and the diagonal block from above.

\begin{lemma}[Square-free quadratic feature edge]\label{lem:squarefree-edge}
Let $K\subset(0,1/2)$ be compact, let $n/p^2\in K$, and let $V$ be the square-free block in \eqref{eq:DV-def}.  Along every sequence for which
$n/p^2\to\rho_*\in K$,
\begin{equation}\label{eq:VV-edges}
\lambda_{\min}\Bigl(\frac{VV^\top}{p^2}\Bigr) \longrightarrow
\bigl(1-\sqrt{2\rho_*}\bigr)^2,
\qquad
\lambda_{\max}\Bigl(\frac{VV^\top}{p^2}\Bigr) \longrightarrow
\bigl(1+\sqrt{2\rho_*}\bigr)^2
\end{equation}
in probability.  Consequently, there are constants $0<c_K\le C_K<\infty$ such that, uniformly over $n/p^2\in K$,
\[
\Prob\bigl\{c_Kp^2I_n\preceq VV^\top\preceq C_Kp^2I_n\bigr\}
\longrightarrow1.
\]
\end{lemma}

\begin{proof}
Put $M_2=\binom p2$ and let $X_2$ be the $n \times M_2$ matrix $(X_2)_{i,(a,b)} = y_{ia} y_{ib}$, so $V = \sqrt 2\, X_2$, and put
\begin{equation}\label{eq:R2-def}
R_2 = \frac{1}{p \sqrt n}\, X_2 X_2^\top .
\end{equation}
Write $\rho = \rho(p) = n/p^2 \in K$; because $\max K<1/2$, in particular $n<M_2$ for all large $p$.  The extreme-eigenvalue theorem of Kogan, Nandy and Huang for random kernel matrices with polynomial scaling \cite[Theorem~1.2(2a)]{kogan2025extremal}, in its square-free degree-two case, applies: the Gaussian coordinates are centered with unit variance and satisfy the moment growth $\E|y_{11}|^k \le k^{Ck}$ for all $k \ge 1$, and along any sequence with $n/p^2 \to \rho_* \in K$ the two extremal eigenvalues of the off-diagonal part converge almost surely,
\begin{equation}\label{eq:KNH-edge}
\lambda_{\min}\bigl(R_2 - \diag R_2\bigr) \longrightarrow \sqrt{\rho_*} -
\sqrt 2,
\qquad
\lambda_{\max}\bigl(R_2 - \diag R_2\bigr) \longrightarrow \sqrt{\rho_*} +
\sqrt 2 .
\end{equation}
The diagonal is estimated directly: we have the exact formula
\begin{equation}\label{eq:R2-diag}
(R_2)_{ii} = \frac{\bigl(\sum_a y_{ia}^2\bigr)^2 - \sum_a
y_{ia}^4}{2\, p \sqrt n} .
\end{equation}
By the $\chi^2$ tail bound and a union bound over $n = O(p^2)$ rows, $\max_{i}\bigl| \|y_i\|_2^2/p - 1 \bigr| \to 0$ in probability; by a union bound over the $np = O(p^3)$ coordinates, $\max_{i,a} y_{ia}^2 = O_{\Prob}(\log p)$; hence $\max_i \sum_a y_{ia}^4 \le (\max_{i,a} y_{ia}^2)(\max_i \|y_i\|_2^2) = O_{\Prob}(p \log p) = o_\P (p^2)$.  Substituting in \eqref{eq:R2-diag} and using $\sqrt n / p = \sqrt\rho$,
\begin{equation}\label{eq:diag-limit}
\max_{i \le n} \Bigl| (R_2)_{ii} - \frac{1}{2\sqrt\rho} \Bigr|
\longrightarrow 0 \quad \text{in probability.}
\end{equation}
Combining \eqref{eq:KNH-edge} and \eqref{eq:diag-limit} yields (along the same sequences, in probability) the limits in \eqref{eq:VV-edges}, after the algebra $2\sqrt{\rho}\,\bigl[\tfrac{1}{2\sqrt\rho} + \sqrt\rho \mp \sqrt2\bigr] = 1 + 2\rho \mp 2\sqrt{2\rho} = (1 \mp \sqrt{2\rho})^2$.  The lower edge in \eqref{eq:VV-edges} is bounded below by $\tfrac12 (1 - \sqrt{2 \rho_+})^2 > 0$ with probability tending to one, where $\rho_+ = \max K < 1/2$.  The corresponding upper bound is uniform as well.  Compactness of $K$ gives the final assertion.
\end{proof}

\begin{lemma}[Operator norm of the diagonal quadratic block]
\label{lem:diagonal-feature-block}
Under the assumptions of Lemma~\ref{lem:squarefree-edge}, let $D_{ia}=y_{ia}^2-1$ be the diagonal block in \eqref{eq:DV-def}.  There is $C_K<\infty$ such that, uniformly over $n/p^2\in K$,
\[
\Prob\{\|D\|_{\op}^2\le C_Kn\}\longrightarrow1.
\]
\end{lemma}

\begin{proof}
We truncate the entries and apply matrix Bernstein to the sample covariance $\widetilde D^\top\widetilde D$.  Put $u_{ia} = y_{ia}^2 - 1$, fix a large constant $L$, and set
\begin{equation}\label{eq:trunc-def}
\mu_p = \E\bigl[u_{11} \mathbf 1\{|u_{11}| \le L \log p\}\bigr],
\qquad
\widetilde u_{ia} = u_{ia} \mathbf 1\{|u_{ia}| \le L\log p\} - \mu_p .
\end{equation}
The $\widetilde u_{ia}$ are independent and centered, satisfy $|\widetilde u_{ia}| \le C \log p$, and have uniformly bounded fourth moments.  Gaussian tails and a union bound over $np = O(p^3)$ entries show that no truncation occurs with probability $1 - o(1)$, provided $L$ is large enough, while Cauchy--Schwarz on the tail gives $|\mu_p| \le C p^{-cL}$, which can be made smaller than any fixed negative power of $p$ by increasing $L$.

Let $\widetilde u_i = (\widetilde u_{ia})_{a \le p}$ and set $A_i=\widetilde u_i\widetilde u_i^\top$ and $M_i=A_i-\E A_i$.  Then $\|M_i\|_{\op} \le C p \log^2 p$, and $A_i^2 = \|\widetilde u_i\|_2^2 A_i$; independence and centering make every off-diagonal entry of $\E A_i^2$ vanish, and its $a$-th diagonal entry is $\E \widetilde u_{ia}^4 + \sum_{c \neq a} \E \widetilde u_{ia}^2\, \E\widetilde u_{ic}^2 \le C p$.  Moreover $\E A_i = \sigma_p^2 I_p$ with $\sigma_p^2 \le C$, hence $\E M_i^2 = \E A_i^2 - \sigma_p^4 I_p \preceq \E A_i^2$, and
\begin{equation}\label{eq:bernstein-var}
\Bigl\| \sum_{i=1}^n \E M_i^2 \Bigr\|_{\op} \le C n p .
\end{equation}
The self-adjoint matrix Bernstein inequality \cite[Theorem~6.6.1]{tropp2015introduction} applies with $\| M_i \|_\op \leq R = C p \log^2 p$, $\sigma^2 = \|\sum_i \E M_i^2\|_{\op} \le C np$, and $t = n$:
\begin{equation}\label{eq:bernstein-applied}
\Prob\Bigl\{\Bigl\|\sum_{i=1}^n M_i\Bigr\|_{\op} > n\Bigr\}
\le 2p \exp\Bigl\{- c\, \frac{n^2}{np + np\log^2 p}\Bigr\}
\le 2p\, e^{-c p / \log^2 p} = o(1) ,
\end{equation}
using $n \asymp p^2$.  Since $\widetilde D^\top \widetilde D = \sum_i M_i + n \sigma_p^2 I_p$ for the truncated matrix $\widetilde D = (\widetilde u_{ia})$, the event in \eqref{eq:bernstein-applied} gives $\|\widetilde D\|_{\op}^2 \le C n$.  On the no-truncation event, $D - \widetilde D = \mu_p \one_n \one_p^\top$, of operator norm $|\mu_p| \sqrt{np} = o(\sqrt n)$ after increasing $L$ once more.  Hence $\|D\|_{\op}^2 \le C n$ with probability $1 - o(1)$.  The estimates are uniform for $\rho_-p^2\le n\le\rho_+p^2$, where $\rho_-=\min K$ and $\rho_+=\max K$.
\end{proof}

\begin{proposition}[Two-sided feature edge]\label{prop:feature-edge}
Let $K \subset (0, 1/2)$ be compact and let $n = n(p)$ satisfy $n/p^2 \in K$ for all large $p$.  There are constants $0 < c_K \le C_K < \infty$ such that, with probability tending to one as $p \to \infty$,
\begin{equation}\label{eq:feature-two-sided}
c_K\, p^2\, \Id_n \;\preceq\; \cQ_y \cQ_y^* \;\preceq\; C_K\, p^2\, \Id_n .
\end{equation}
In particular, with probability tending to one,
\begin{equation}\label{eq:feature-lip}
\sup_{T \neq 0} \frac{\|\cQ_y(T)\|_2}{\sqrt n\, \|T\|_{\F}} \le C_K' .
\end{equation}
Moreover, if $\cG(T) = (\sqrt p \Tr(G_i T))_{i=1}^n$ with $G_i \stackrel{\mathrm{iid}}{\sim} \GOE p$, then $\|\cG\|_{\F \to 2} \le C_K' \sqrt n$ with probability at least $1 - 2e^{-n/2}$.  All constants and all $o(1)$ probabilities are uniform over sequences with $n/p^2 \in K$.
\end{proposition}

\begin{proof}
Write $\rho_-:=\min K$ and $\rho_+:=\max K$. On the events in Lemmas~\ref{lem:squarefree-edge} and~\ref{lem:diagonal-feature-block}, the decomposition $\cQ_y \cQ_y^* = DD^\top + VV^\top$ satisfies
\[
\cQ_y \cQ_y^* \;\succeq\; VV^\top \;\succeq\; \tfrac12 (1 - \sqrt{2\rho_+})^2\, p^2\, I_n ,
\qquad
\|\cQ_y\cQ_y^*\|_{\op} \le C n + \bigl(1 + \sqrt{2\rho_+}\bigr)^2 p^2 (1 + o(1)) \le C_K p^2 ,
\]
which is \eqref{eq:feature-two-sided}; and \eqref{eq:feature-lip} follows since $\|\cQ_y\|_{\F\to2}^2 = \|\cQ_y\cQ_y^*\|_{\op} \le C_K p^2 \le (C_K/\rho_-)\, n$.  Uniformity over $K$ holds by compactness: indeed, if the failure probability did not tend to $0$ uniformly, there would be $p_j \to \infty$ and $n_j$ with $n_j/p_j^2 \in K$ witnessing this; passing to a subsequence with $n_j / p_j^2 \to \rho_* \in K$ and applying the two lemmas along it yields a contradiction.

For the Gaussian map, in Frobenius coordinates, $\cG(T)_i = \ip{\sqrt p\, G_i}{T}$ and $\sqrt p\, G_i = \sqrt2\, Z_i$ with $Z_i$ standard Gaussian in $\Sym p$; thus $\cG$ is represented by an $n \times D_p$ matrix with independent $\cN(0,2)$ entries.  The rectangular Gaussian estimate (Lemma~\ref{lem:gauss-rect}) with $\sigma = \sqrt2$ and $t = \sqrt n$ gives $\|\cG\|_{\F\to2} \le \sqrt2\,(2\sqrt n + \sqrt{D_p}) \le C_K' \sqrt n$ with probability at least $1 - 2e^{-n/2}$, since $D_p \le p^2 \le n/\rho_-$.
\end{proof}

We record the normalization used on the satisfiable side.  With $p = d$, $W_i = (x_i x_i^\top - \Id)/\sqrt d$ and
\begin{equation}\label{eq:W-map-def}
\cW : \Sd \to \R^n, \qquad \cW(S) = \bigl(\ip{W_i}{S}\bigr)_{i=1}^n = \frac{1}{\sqrt d}\, \cQ_x(S),
\end{equation}
Proposition~\ref{prop:feature-edge} reads as follows.

\begin{corollary}[Feature edge, satisfiable-side normalization]\label{cor:feature-sat}
If $n/d^2 \to \rho_* \in (0, 1/2)$, there are constants $0 < c \le C < \infty$ (depending only on $\rho_*$) such that, with probability tending to one,
\begin{equation}\label{eq:feature-sat}
c\, d\, I_n \;\preceq\; \cW \cW^* \;\preceq\; C\, d\, I_n , \qquad\text{in particular}\qquad
\|\cW^*\|_{2 \to \F} = \|\cW\|_{\F \to 2} \le C^{1/2} \sqrt d .
\end{equation}
\end{corollary}

\subsection{Elementary estimates on Gaussian samples}\label{sec:data}

We record below the elementary estimates on Gaussian samples used in the proofs. 

\begin{lemma}[Gaussian and chaos estimates]\label{lem:data}
Suppose $c_1 d^2\le n\le C_1 d^2$ for fixed $0<c_1\le C_1<\infty$, and fix $C_0 < \infty$; put $K_d = \lfloor C_0\, d \log d\rfloor$.  There is a constant $C$ only depending on $c_1, C_1, C_0$, such that, with probability tending to one, simultaneously,
\begin{align}
\max_{i \le n} \|x_i\|_2 &\le 2\sqrt d ,
\label{eq:radius}\\
\max_{i \le n,\, a \le d} x_{ia}^2 &\le C \log d ,
\label{eq:coord-max}\\
\max_{i \neq j} \bigl|\ip{W_i}{W_j}\bigr| &\le C \log d ,
\label{eq:coherence}\\
\max_{J \subseteq [n],\ |J| \le K_d} \|X_J\|_{\op}^2 &\le C\, d
(\log d)^2 .
\label{eq:sparse-norms}
\end{align}
\end{lemma}

\begin{proof}
Estimates  \eqref{eq:radius} and \eqref{eq:coord-max} follow from standard Gaussian concentration and a union bound. For \eqref{eq:coherence}, fix $i \neq j$ and write
\begin{equation}\label{eq:coherence-formula}
\ip{W_i}{W_j}
= \frac{(x_i^\top x_j)^2 - \|x_i\|_2^2 - \|x_j\|_2^2 + d}{d} .
\end{equation}
Conditionally on $x_i$, the scalar $x_i^\top x_j \sim \cN(0, \|x_i\|_2^2)$; on the radius event \eqref{eq:radius}, $\Prob\{|x_i^\top x_j| > A \sqrt{d \log d} \mid x_i\} \le 2 d^{-A^2/8}$. Choose $A$ so that a union bound over $O(d^4)$ ordered pairs is $o(1)$; formula \eqref{eq:coherence-formula} together with $\|x_i\|_2^2, \|x_j\|_2^2 \le 4d$ then gives \eqref{eq:coherence}. For \eqref{eq:sparse-norms}, fix $J$ with $|J| = k$. The Gaussian matrix estimate \eqref{eq:gauss-rect} gives $\Prob\{\|X_J\|_{\op} > \sqrt d + \sqrt k + t\} \le 2 e^{-t^2/2}$.  There are at most $(en/k)^k$ such sets and, since $K_d = o(n)$, $\sum_{k \le K_d} \binom nk \le \exp\{C K_d \log(en/K_d)\}$.  Taking $t = A \sqrt{K_d \log(en/K_d)}$ for every $k \le K_d$ and choosing $A$ large makes the union failure probability $\le 2\exp\{(C - A^2/2) K_d \log(en/K_d)\} = o(1)$; finally $(\sqrt d + \sqrt{K_d} + t)^2 = O(d \log^2 d)$.
\end{proof}

\begin{lemma}[Hanson--Wright inequality {\cite{rudelson2013hanson}}]\label{lem:hw}
There is a universal $c > 0$ such that for $x \sim \cN(0, I_d)$ and $P \in \Sd$, for all $t \ge 0$,
\[
\Prob\bigl\{\, |x^\top P x - \Tr P| > t \,\bigr\}
\;\le\; 2 \exp\Bigl\{- c \min\Bigl(\frac{t^2}{\|P\|_{\F}^2},\,
\frac{t}{\|P\|_{\op}}\Bigr)\Bigr\} .
\]
\end{lemma}

\subsection{Homogeneous witness and the balanced matrix alternative}\label{sec:cor-alternative}

We include below the proof of the equivalence between infeasibility of \eqref{eq:P} and the existence of a positive definite balanced combination of the $x_i x_i^\top$.   Recall $\cQ_x(S) = (x_i^\top S x_i - \Tr S)_{i=1}^n$, and let $P_{\one^\perp} = I_n - n^{-1} \one\one^\top$ denote the projection of
$\R^n$ onto $\one^\perp$.

\begin{lemma}[Homogeneous formulation and the alternative]\label{lem:homogeneous}
Assume that $x_1, \dots, x_n$ span $\R^d$ (an almost sure event when $n \ge d$).  Consider the linear map
\[
\cD_X : \Sd \longrightarrow \one^\perp, \qquad
\cD_X(S) = P_{\one^\perp}\bigl(x_i^\top S x_i\bigr)_{i=1}^n .
\]
\begin{enumerate}
\item[{\rm (i)}] Problem \eqref{eq:P} is feasible if and only if $\ker \cD_X \cap \Sdp \neq \{0\}$; equivalently, if and only if there are $S \succeq 0$ with $\|S\|_{\F} = 1$ and $b \in \R$ such that
\begin{equation}\label{eq:witness}
\cQ_x(S) = b \one .
\end{equation}
\item[{\rm (ii)}] Exactly one of the following two alternatives holds:
\begin{align}
&\ker \cD_X \cap \Sdp \neq \{0\};
\label{eq:alt-a}\\[2pt]
&\exists\, y \in \one^{\perp} \ \text{ such that } \
\sum_{i=1}^n y_i\, x_i x_i^{\top} \succ 0 .
\label{eq:alt-b}
\end{align}
\end{enumerate}
\end{lemma}

\begin{proof}
(i) If $S$ solves \eqref{eq:P}, then $S \neq 0$, $S \succeq 0$ and $\cD_X(S) = 0$.  Conversely, let $0 \neq S \succeq 0$ with $\cD_X(S) = 0$; then the numbers $x_i^\top S x_i$ share a common value $c \ge 0$.  If $c = 0$, then $S^{1/2} x_i = 0$ for every $i$, and the spanning assumption forces $S = 0$, a contradiction; hence $c > 0$ and $d S / c$ solves \eqref{eq:P}.  For the second formulation: given a solution $S$ of \eqref{eq:P}, the pair $(S/\|S\|_{\F}, (d - \Tr S)/\|S\|_{\F})$ satisfies \eqref{eq:witness}; conversely, a pair $(S, b)$ satisfying \eqref{eq:witness} with $\|S\|_{\F} = 1$ has all values $x_i^\top S x_i$ equal to $c := \Tr S + b \ge 0$, and as before $c > 0$, so $dS/c$ solves \eqref{eq:P}.

(ii) The two alternatives cannot both hold: if $0 \neq S \succeq 0$ with $\cD_X(S) = 0$ and $y \in \one^\perp$ with $M_y := \sum_i y_i x_i x_i^\top \succ 0$, then
\[
0 < \ip{M_y}{S} = \sum_i y_i\, x_i^\top S x_i = c \sum_i y_i = 0,
\]
a contradiction (here $c$ is the common value above).  Suppose now that \eqref{eq:alt-a} fails, i.e.\ $\ker \cD_X \cap \Sdp = \{0\}$.  Consider
\[
K \;:=\; \cD_X\bigl(\{S \succeq 0,\ \Tr S = 1\}\bigr) \subseteq \one^\perp.
\]
Let $k_0$ be the point of $K$ closest to the origin, so $k_0 \neq 0$. Differentiating $t \mapsto \|k_0 + t(k - k_0)\|_2^2$ at $t = 0$ gives $\ip{k_0}{k} \ge \|k_0\|_2^2 > 0$ for all $k \in K$.  Setting $y := k_0 \in \one^\perp$, this reads: for every $S \succeq 0$ with $\Tr S = 1$,
\[
0 < \ip{y}{\cD_X(S)} = \ip{\cD_X^* y}{S}, \qquad
\cD_X^* y = \sum_{i=1}^n y_i\, x_i x_i^\top ,
\]
where we used $P_{\one^\perp} y = y$.  This implies that $\cD_X^* y$ is strictly positive definite, which is \eqref{eq:alt-b}.
\end{proof}

\section{The satisfiable phase}\label{sec:sat}

This section proves the following theorem, which contains Theorem~\ref{thm:main}.(a) and Theorem~\ref{thm:conditioned}.

\begin{theorem}[Satisfiable phase]\label{thm:sat}
Fix $\alpha^*<1/4$, let $n=n(d)$ satisfy $\limsup_{d\to\infty}n/d^2\le\alpha^*$, and let $x_1, \dots, x_n$ be independent $\cN(0, I_d)$ vectors.  There is a constant $\kappa_1 \ge 1$, depending only on $\alpha^*$, such that with probability tending to one, there exists $S \in \Sd$ with
\[
x_i^\top S x_i = d \quad (1 \le i \le n),
\qquad \Tr S = d, \qquad 
\frac{1}{4\kappa_1^2}\, \Id \;\preceq\; S \;\preceq\; 4 \kappa_1^2\, \Id .
\]
\end{theorem}

Without loss of generality, we assume throughout this section that 
\begin{equation}\label{eq:alpha-assume}
n=\lfloor\alpha d^2\rfloor,\qquad
\frac{n}{d^2}\longrightarrow\alpha\in(0,1/4).
\end{equation}
Indeed, feasibility is monotone under deletion of measurements. For all large $d$, append independent Gaussian points until the sample has $\lfloor\alpha d^2\rfloor$ points, with $\alpha \in ( \alpha^* , 1/4)$. A solution of the padded system solves the original one.

\subsection{Outline and proof of Theorem~\ref{thm:sat}}\label{sec:sat-setup}

Choose $\eta>0$ so that
\begin{equation}\label{eq:eta-choice}
\sqrt\alpha+4\eta<\frac12.
\end{equation}
Lemma~\ref{lem:widthcone}, applied with $a=\sqrt\alpha+4\eta$, gives a fixed $\kappa_0>1$ such that, for all large $d$,
\begin{equation}\label{eq:width-kappa0}
w\bigl(\cK_{\kappa_0}\cap\sphereF\bigr)
\ge(\sqrt\alpha+4\eta)d.
\end{equation}

Recall
\[
W_i=\frac{x_ix_i^\top-\Id}{\sqrt d},\qquad
\cW(S)=\bigl(\ip{W_i}{S}\bigr)_{i=1}^n.
\]
Set $\kappa_1=2\kappa_0$. We prove the stronger centered assertion
\begin{equation}\label{eq:sat-target}
\cB:=\left\{S\in\Sd:\frac{1}{2\kappa_1}\Id\preceq S
\preceq2\kappa_1\Id\right\},
\qquad 0\in\cW(\cB).
\end{equation}
Indeed, if $\cW(S)=0$ for some $S\in\cB$, then $x_i^\top Sx_i=\Tr S>0$ for every $i$.  Therefore $\widehat S=dS/\Tr S$ is an exact fit, and
\begin{equation}\label{eq:implication-sat-condition-eigenvalues}
\frac{1}{4\kappa_1^2}\Id\preceq\widehat S \preceq4\kappa_1^2\Id.
\end{equation}
The centered assertion is stronger than merely requiring the quadratic measurements to share a common value, which only forces $\cW(S)$ to lie in ${\rm span}(\one)$. It also gives the trace identity $\Tr(\widehat S) = d$ after rescaling. 

The proof uses the following dual formulation to establish \eqref{eq:sat-target}, and hence the feasibility of \eqref{eq:P}, with high probability.

\begin{lemma}[Convex separation]\label{lem:convex-separation}
Let
\[
K=\cW(\cB),\qquad
h(y)=\sup_{S\in\cB}\ip{\cW^*y}{S}
=\sup_{k\in K}\ip{y}{k}.
\]
Then $0\in K$ if and only if $h(y)\ge0$ for every unit
$y\in\R^n$.  In particular, $h(y)>0$ for every unit $y$ is sufficient
for $0\in K$.  Strict positivity for every
unit $y$ is equivalent to $0\in\operatorname{int}K$.
\end{lemma}

\begin{proof}
If $0\in K$, then $h(y)\ge0$.  Conversely, suppose $0\notin K$ and let
$k_0$ be the point of the compact convex set $K$ closest to the origin.
For every $k\in K$, differentiation along the segment from $k_0$ to $k$
gives $\ip{k_0}{k}\ge\|k_0\|_2^2$.  With
$e=k_0/\|k_0\|_2$, this yields
$h(-e)=-\inf_{k\in K}\ip{e}{k}<0$.

For the final assertion, $0\in\operatorname{int}K$ implies
$rB_2^n\subseteq K$ for some $r>0$, and hence $h(y)\ge r$ on the unit
sphere.  In the other direction, continuity and compactness give
$r:=\min_{\|y\|_2=1}h(y)>0$.  By positive homogeneity, $h(y)\ge r\|y\|_2$ for every $y\in\R^n$; hence, for every $x\in rB_2^n$,
\[
\langle y,x\rangle\le r\|y\|_2\le h(y)\qquad\text{for all }y\in\R^n.
\]
Since $K=\bigcap_{y\in\R^n}\{x:\langle y,x\rangle\le h(y)\}$, it follows that $rB_2^n\subseteq K$, and therefore $0\in\operatorname{int}K$.
\end{proof}

Thus it suffices to prove, with probability tending to one, that $h(y)>0$ for every unit vector $y\in\R^n$. As described in Section \ref{sec:proof-method}, we decompose each such dual vector into a head $y_J = (y_j)_{j \in J}$ of at most $|J| \leq C_0d \log d$ heavy coordinates and a low-influence tail: after normalization, every coordinate of $y_{J^c}$ is at most $\theta/\sqrt{d}$, for a fixed constant $\theta$. We impose the head equations $\ip{W_j}{S} = 0$, $j \in J$, exactly and use Gaussian comparison to bound the tail contribution from below.

For $J\subseteq[n]$, write
\[
\cW_J(S)=\bigl(\ip{W_i}{S}\bigr)_{i\in J}.
\]
The following proposition says that solving any sparse collection of head equations exactly preserves almost all of the Gaussian width of the bounded-condition cone.

\begin{proposition}[Uniform head-section width]\label{prop:heads}
Fix $C_0<\infty$ and set $K_d=\lfloor C_0d\log d\rfloor$.  With probability tending to one, simultaneously for every $J\subseteq[n]$ with
$|J|\le K_d$,
\begin{equation}\label{eq:head-width}
w\bigl(\cK_{\kappa_1}\cap\ker\cW_J\cap\sphereF\bigr)
\ge w\bigl(\cK_{\kappa_0}\cap\sphereF\bigr)-o(d)
\ge(\sqrt\alpha+3\eta)d.
\end{equation}
The $o(d)$ term may depend on the fixed $C_0$ but is uniform in $J$. In particular, the compact set
\begin{equation}\label{eq:DJ-def}
\cD_J:=\sqrt d\, \bigl(\cK_{\kappa_1}\cap\ker\cW_J\cap\sphereF\bigr)
\end{equation}
is nonempty and satisfies
\begin{equation}\label{eq:DJ-box}
\cW_J(S)=0,\qquad
\kappa_1^{-1} \Id \preceq S\preceq\kappa_1 \Id
\quad(S\in\cD_J).
\end{equation}
\end{proposition}

Section~\ref{sec:sat-heads} proves Proposition~\ref{prop:heads}.  Starting from a support maximizer in the smaller cone $\cK_{\kappa_0}$, it solves the equations indexed by $J$ through the inverse feature Gram matrix. The spectral bound in Corollary \ref{cor:feature-sat} controls the Frobenius norm of the correction, while the off-diagonal Gram and sparse-column estimates in Lemma \ref{lem:data} make its operator norm $o(d^{-1/2})$, so that the corrected solution remains within the larger cone $\cK_{\kappa_1}$.

The next proposition gives a positive margin for every normalized tail direction whose individual coordinates have low influence.

\begin{proposition}[Uniform low-influence tail margin]\label{prop:lowinf}
There are fixed constants $\theta,c_*>0$, depending only on $\alpha$, such that the following holds for every fixed $C_0<\infty$.  With $K_d=\lfloor C_0d\log d\rfloor$, with probability tending to one, simultaneously for every $J\subseteq[n]$ with $|J|\le K_d$ satisfying
\eqref{eq:head-width},
\begin{equation}\label{eq:lowinf}
\inf_{\substack{\operatorname{supp}z\subseteq J^c,\ \|z\|_2=1\\
\|z\|_\infty\le\theta/\sqrt d}} \ \sup_{S\in\cD_J} \sum_{i\notin J}z_i\ip{W_i}{S} \ge c_*d.
\end{equation}
An infimum over an empty index set is $+\infty$.
\end{proposition}

Sections~\ref{sec:sat-gordon}--\ref{sec:sat-lowinf} prove Proposition~\ref{prop:lowinf}.  Gordon's inequality first gives the corresponding margin for covariance-matched Gaussian surrogates.  A smoothed second-chaos Lindeberg argument transfers the margin, in conditional expectation, to the quadratic chaos rows.  Radial clipping then gives an $e^{-cd^2}$ concentration bound, which beats the union bound over the $e^{O(d\log^2d)}$ possible heads.

It remains to reduce an arbitrary dual direction to one covered by Proposition \ref{prop:lowinf}.

\begin{lemma}[Head--tail decomposition]\label{lem:regularity}
Fix $\theta>0$.  There is $C_\theta<\infty$ such that for every unit vector $y\in\R^n$ there is $J\subseteq[n]$ with $|J|\le C_\theta d\log d$ for which, writing $r=\|y_{J^c}\|_2$, one of the following holds:
\begin{align}
&r\le d^{-2};\label{eq:tiny-tail}\\
&r>0\quad\text{and}\quad \left\|\frac{y_{J^c}}r\right\|_\infty \le\frac{\theta}{\sqrt d}.
\label{eq:spread-tail}
\end{align}
\end{lemma}

\begin{proof}
Start with $J=\varnothing$ and current remainder norm $r=1$.  While $r>d^{-2}$, add to $J$ every remaining coordinate with
\begin{equation}\label{eq:peel}
|y_i|>\frac{\theta r}{2\sqrt d}.
\end{equation}
If $q$ coordinates are added in one round, then $q\theta^2r^2/(4d)\le r^2$, so $q\le4d/\theta^2$.  Let $r'$ be the new remainder norm.  If $r'\ge r/2$, stop: every remaining coordinate is at most $\theta r/(2\sqrt d)\le\theta r'/\sqrt d$, which gives \eqref{eq:spread-tail}.  If $r'<r/2$, continue.  After $\ell$ non-stopping rounds the remainder norm is at most $2^{-\ell}$, so after at most $\lceil2\log_2d\rceil+1$ rounds it is at most $d^{-2}$. Each round adds at most $4d/\theta^2$ coordinates, and the size bound and the dichotomy follow.
\end{proof}

\begin{proof}[Proof of Theorem~\ref{thm:sat}]
Assume the reduction to \eqref{eq:alpha-assume} after padding.  Take $\theta,c_*$ from Proposition~\ref{prop:lowinf} and choose $C_0\ge C_\theta$ in Proposition~\ref{prop:heads}.  Intersect the high-probability events of Corollary~\ref{cor:feature-sat}, Proposition~\ref{prop:heads}, and Proposition~\ref{prop:lowinf}.  By \eqref{eq:DJ-box}, every section $\cD_J$ with $|J|\le K_d$ is contained in the fixed box $\cB$.

We show that $h(y)>0$ for every unit $y\in\R^n$.  Apply
Lemma~\ref{lem:regularity} and write
\begin{equation}\label{eq:uv-def}
u=y_J,\qquad v=y_{J^c},\qquad M_u=\cW_J^*u,\qquad M_v=\cW_{J^c}^*v.
\end{equation}
In the spread-tail case \eqref{eq:spread-tail}, put $z=v/r$.  Every $S\in\cD_J$ eliminates the head rows $\ip{M_u}{S} = 0$, and Proposition~\ref{prop:lowinf} gives
\begin{equation}\label{eq:spread-case}
h(y)\ge\sup_{S\in\cD_J}\ip{M_u+M_v}{S} =r\sup_{S\in\cD_J}\sum_{i\notin J}z_i\ip{W_i}{S} \ge r c_*d>0.
\end{equation}
In the tiny-tail case \eqref{eq:tiny-tail}, choose $S_0\in\cD_J$.  Since $r\le d^{-2}$, $\|u\|_2\ge1/2$ for large $d$.  The lower spectral edge in Corollary \ref{cor:feature-sat}, restricted to the principal submatrix on $J$, gives
\begin{equation}\label{eq:Mu-lower}
\|M_u\|_{\F}^2 =u^\top\cW_J\cW_J^*u\ge cd\|u\|_2^2\ge c'd.
\end{equation}
Write $\sgn(M_u)$ for the matrix $M_u$ with eigenvalues replaced by their signs. Set (with $\sgn (0 ) = 0$)
\begin{equation}\label{eq:Sprime}
\delta_0=\frac{1}{4\kappa_1},\qquad S'=S_0+\delta_0\sgn(M_u),
\end{equation}
so that $( \tfrac{1}{2\kappa_1} + \delta_0) \Id\preceq S_0 \preceq (\kappa_1 + \delta_0) \Id $ for $\kappa_1 \geq 1$ and $S'\in\cB$, and
\begin{equation}\label{eq:head-inner}
\ip{M_u}{S'}=\delta_0\|M_u\|_* \ge\delta_0\|M_u\|_{\F}\ge c''\sqrt d.
\end{equation}
The upper feature edge in Corollary \ref{cor:feature-sat} gives
\begin{equation}\label{eq:Mv-bounds}
\|M_v\|_{\F}\le C\sqrt d\,r,\qquad \|M_v\|_*\le\sqrt d\,\|M_v\|_{\F}\le Cdr,
\end{equation}
and therefore
\begin{equation}\label{eq:tail-inner}
|\ip{M_v}{S'}| \le\|S'\|_{\op}\|M_v\|_* \le Cdr\le C/d.
\end{equation}
Thus $h(y)\ge c''\sqrt d-C/d>0$ in the tiny-tail case as well.

Lemma~\ref{lem:convex-separation} now gives $0\in\cW(\cB)$. Choose $S\in\cB$ with $\cW(S)=0$ and set $\widehat S=dS/\Tr S$.  The calculation \eqref{eq:implication-sat-condition-eigenvalues} shows that  $(4\kappa_1^2)^{-1}\Id \preceq\widehat S \preceq4\kappa_1^2 \Id$.  This proves the theorem.
\end{proof}

\begin{remark}[Condition-number bound supplied by the proof]
\label{rem:kappa-blowup}
In the proof of Lemma~\ref{lem:widthcone}, $q(s)=\tfrac12-2c_+s+O(s^2)$ as $s\downarrow0$.  The explicit choice there gives $\kappa(a)=O((1/2-a)^{-1})$.  With $a=\sqrt\alpha+4\eta$ and $\kappa_1=2\kappa_0$, the fit constructed above has condition number at most $4\kappa_1^2$. Taking $\eta$ to be a small fixed fraction of $1/2-\sqrt\alpha$ yields
\[
\operatorname{cond}(S)=O\bigl((1/4-\alpha)^{-2}\bigr)
\qquad(\alpha \uparrow1/4).
\]
This bounds the fit produced by the proof, not every fit.  The replica prediction of \cite{maillard2024fitting} concerns typical solutions and does not exclude atypical well-conditioned (or badly-conditioned) fits.
\end{remark}

\begin{remark}[Where the sharp constant $1/4$ enters]\label{rem:where-quarter}
The degree-two feature map remains uniformly well conditioned all the way up to $n/d^2 < 1/2$ (Proposition~\ref{prop:feature-edge}).  The sharp constant $1/4$ enters only through the comparison
\[
w\bigl(\Sdp \cap \sphereF\bigr) \sim \frac d2
\qquad\text{versus}\qquad \E\|g_n\|_2 \sim \sqrt\alpha\, d
\]
in \eqref{eq:gordon-tail-margin}.  The bounded-condition cone approximates the first width while retaining the spectral slack needed for the uniform head corrections.
\end{remark}

\subsection{Uniform removal of all sparse heads}\label{sec:sat-heads}

We now prove Proposition~\ref{prop:heads}. We first record the elementary spectral fact behind \eqref{eq:DJ-box}: if $R \in \cK_{\kappa_1} \cap \sphereF$, then $d \lambda_{\min}(R)^2 \leq  \|R\|_{\F}^2 =1 \leq d\lambda_{\max}(R)^2$ forces
\begin{equation}\label{eq:P-eigen-generic}
\frac{1}{\kappa_1 \sqrt d}\, \Id \;\preceq\; R \;\preceq\;
\frac{\kappa_1}{\sqrt d}\, \Id ,
\end{equation}
which gives \eqref{eq:DJ-box} after multiplication by $\sqrt d$.  The same argument applies to elements of $\cK_{\kappa_0} \cap \sphereF$.

We work on the event $\cE_{\mathrm{data}}$ on which both Corollary~\ref{cor:feature-sat} and Lemma~\ref{lem:data} hold.  Its probability is $1-o(1)$.  Let $Z$ be a standard Gaussian in $\Sd$, independent of the data. Choose
\begin{equation}\label{eq:P-max}
P = P(Z) \in \arg\max_{R \in \cK_{\kappa_0} \cap \sphereF} \ip{Z}{R} .
\end{equation}
The index set $\cK_{\kappa_0} \cap \sphereF$ is fixed and compact, and $(Z,R)\mapsto\ip{Z}{R}$ is continuous.  The measurable maximum theorem \cite[Theorem~18.19]{aliprantis2006infinite} therefore gives a Borel maximizer $P=P(Z)$.  In particular, $P$ remains independent of the data. The spectral estimate \eqref{eq:P-eigen-generic}, with $\kappa_0$ in place of $\kappa_1$, gives
\begin{equation}\label{eq:P-eigen}
\lambda_{\min}(P) \ge \frac{1}{\kappa_0 \sqrt d}, \qquad \|P\|_{\op} \le \frac{\kappa_0}{\sqrt d} .
\end{equation}
Set $r_i := \ip{W_i}{P} = (x_i^\top P x_i - \Tr P)/\sqrt d$.

\begin{lemma}[Exact correction and support transfer]
\label{lem:head-correction-bounds}
Fix $C_0<\infty$ and set $K_d=\lfloor C_0d\log d\rfloor$. Fix $Z \in \sfS^d$ and let $P\in\cK_{\kappa_0}\cap\sphereF$ be a support maximizer \eqref{eq:P-max}. Put $R_d=\max_{i\le n}|r_i|$, and suppose $R_d\le C_A\sqrt{\log d/d}$ and work on $\cE_{\mathrm{data}}$.  For each nonempty $J\subseteq[n]$ with $|J|\le K_d$, let $Q_J = \cW_J \cW_J^*$ and define
\begin{equation}\label{eq:aJ-def}
a_J = Q_J^{-1} r_J, \qquad
E_J = -\cW_J^* a_J , \qquad \widetilde P_J = \frac{P + E_J}{\|P + E_J\|_{\F}};
\end{equation}
for $J=\varnothing$, set $a_J=0$ and $E_J=0$.  For all sufficiently large $d$, simultaneously over $|J|\le K_d$,
\begin{equation}\label{eq:PJ-close}
\widetilde P_J\in \cK_{\kappa_1}\cap\ker\cW_J\cap\sphereF,
\end{equation}
and
\begin{equation}\label{eq:width-transfer}
\sup_{S \in \cK_{\kappa_1} \cap \ker\cW_J \cap \sphereF} \ip{Z}{S}
\;\ge\;
\sup_{R \in \cK_{\kappa_0} \cap \sphereF} \ip{Z}{R}
\;-\; C\, \|Z\|_{\F}\, \frac{\log d}{\sqrt d}  .
\end{equation}
\end{lemma}

\begin{proof}
The lower edge in \eqref{eq:feature-sat} passes to principal submatrices,
so $Q_J\succeq cdI_{|J|}$ and is invertible.  The definition
\eqref{eq:aJ-def} gives $\cW_J(P+E_J)=0$, and thus $\widetilde P_J \in \ker \cW_J$.  Uniformly over
$|J|\le K_d$,
\begin{align}
\|a_J\|_2 &\le \frac{\|r_J\|_2}{cd}
\le \frac{C \sqrt{K_d}\, R_d}{d}
= O\Bigl(\frac{\log d}{d}\Bigr),
\label{eq:aJ-l2}\\
\|a_J\|_1 &\le \sqrt{K_d}\, \|a_J\|_2
= O\Bigl(\frac{\log^{3/2} d}{\sqrt d}\Bigr) .
\label{eq:aJ-l1}
\end{align}
Writing 
\begin{equation} 
  (Q_J)_{ii} (a_J)_i = r_i - \sum_{j \in J, j \neq i} (Q_J)_{ij}(a_J)_j,
\end{equation}  
and using $(Q_J)_{ii} = \| W_i\|_\F^2 \geq cd$ from Corollary \ref{cor:feature-sat} and $| (Q_J)_{ij}| \leq C \log d$ from Lemma \ref{lem:data} give
\begin{equation}\label{eq:aJ-linf}
\|a_J\|_\infty
\le \frac{C}{d}\bigl(R_d + \log d\, \|a_J\|_1\bigr)
= O\Bigl(\frac{\log^{5/2} d}{d^{3/2}}\Bigr) .
\end{equation}
The Frobenius error satisfies
\begin{equation}\label{eq:EJ-frob}
\|E_J\|_{\F}^2 = a_J^\top Q_J a_J \le C d\, \|a_J\|_2^2,
\qquad
\|E_J\|_{\F} = O\Bigl(\frac{\log d}{\sqrt d}\Bigr) .
\end{equation}
Expanding the $W_i$ gives
\begin{equation}\label{eq:EJ-formula}
E_J = -\frac{1}{\sqrt d}
\Bigl( X_J \diag(a_J) X_J^\top - \Bigl(\sum_{i \in J} (a_J)_i\Bigr)
\Id \Bigr).
\end{equation}
Therefore \eqref{eq:sparse-norms}, \eqref{eq:aJ-l1}, and
\eqref{eq:aJ-linf} imply
\begin{equation}\label{eq:EJ-op}
\|E_J\|_{\op}
\le \frac{1}{\sqrt d}\Bigl(\|a_J\|_\infty \|X_J\|_{\op}^2 +
\|a_J\|_1\Bigr)
= O\Bigl(\frac{\log^{9/2} d}{d}\Bigr)
= o\bigl(d^{-1/2}\bigr),
\end{equation}
uniformly in $J$.

Write $a=\lambda_{\min}(P)\ge1/(\kappa_0\sqrt d)$ and
\[
\delta_d = \sup_{|J| \le K_d} \frac{\|E_J\|_{\op}}{a}
\;\le\; \kappa_0 \sqrt d\, \sup_{|J| \le K_d} \|E_J\|_{\op} = o(1) .
\]
Since $\lambda_{\max}(P)\le\kappa_0a$, Weyl's inequality gives
\[
\frac{\lambda_{\max}(P + E_J)}{\lambda_{\min}(P + E_J)}
\;\le\; \frac{\kappa_0 + \delta_d}{1 - \delta_d} \;<\; \kappa_1
\]
for all sufficiently large $d$; in particular, $P+E_J$ is positive
definite and $\widetilde P_J\in \cK_{\kappa_1}\cap\ker\cW_J\cap\sphereF$.  Since
$\|E_J\|_{\F}=o(1)$ and $\|P\|_{\F}=1$,
$|\|P+E_J\|_{\F}-1|\le\|E_J\|_{\F}$, and the triangle
inequality applied to
\[
\frac{P + E_J}{\|P + E_J\|_{\F}} - P
= \frac{E_J}{\|P+E_J\|_{\F}}
+ \Bigl(\frac{1}{\|P+E_J\|_{\F}} - 1\Bigr) P
\]
gives $\|\widetilde P_J-P\|_{\F}=O(\log d/\sqrt d)$.  Since $P$ is the maximizer in
\eqref{eq:P-max}, bounding the left-hand side of \eqref{eq:width-transfer} below by $\langle Z , \widetilde P_J \rangle$ proves \eqref{eq:width-transfer}.
\end{proof}

\begin{proof}[Proof of Proposition~\ref{prop:heads}]
Conditionally on $P$, Lemma~\ref{lem:hw} with \eqref{eq:P-eigen} gives, for every $t > 0$,
\begin{equation}\label{eq:HW-applied}
\Prob\bigl\{|r_i| > t \,\big|\, P\bigr\}
\;\le\; 2\exp\bigl\{-c \min(d t^2,\ d t)\bigr\} ,
\end{equation}
absorbing $\kappa_0$ into $c$.  Thus, for every prescribed $A > 0$ there is
$C_A$ with
\begin{equation}\label{eq:Rd}
R_d := \max_{i \le n} |r_i| \;\le\; C_A \sqrt{\frac{\log d}{d}}
\end{equation}
outside a joint $(W, Z)$-event of probability at most $d^{-A}$. On $\cE_{\mathrm{data}}$ and the event in \eqref{eq:Rd}, Lemma~\ref{lem:head-correction-bounds} gives \eqref{eq:width-transfer} simultaneously for every $|J|\le K_d$.

We now pass from this joint event to conditional widths.
Let $q(W)$ denote the conditional probability, over $Z$, that
\eqref{eq:Rd} fails.  The joint estimate gives $\E_W\, q(W) \le d^{-A}$,
hence, by Markov's inequality, $q(W) \le d^{-A/2}$ outside a data event of
probability at most $d^{-A/2}$.  Fix a data realization
$W \in \cE_{\mathrm{data}}$ with $q(W) \le d^{-A/2}$.  Since $q(W) < 1$,
at least one $Z_0$ satisfies \eqref{eq:Rd}; applying Lemma \ref{lem:head-correction-bounds} to this $Z_0$ shows that every section in \eqref{eq:head-width} is nonempty.  On the conditional good set,
\eqref{eq:width-transfer} holds; on its complement, both support functions
in \eqref{eq:width-transfer} have absolute value at most $\|Z\|_{\F}$,
because both index sets lie on $\sphereF$.  Taking conditional
expectations over $Z$ and using $\E\|Z\|_{\F} = O(d)$,
$\E\|Z\|_{\F}^2 = D_d = O(d^2)$ and Cauchy--Schwarz, we get, uniformly in
$J$,
\[
w\bigl(\cK_{\kappa_1} \cap \ker\cW_J \cap \sphereF\bigr)
\;\ge\; w\bigl(\cK_{\kappa_0} \cap \sphereF\bigr)
- O(\sqrt d \log d) - O\bigl(d \sqrt{q(W)}\bigr)
\;=\; w\bigl(\cK_{\kappa_0} \cap \sphereF\bigr) - o(d),
\]
choosing $A > 4$.  Together with \eqref{eq:width-kappa0}, this proves
\eqref{eq:head-width} on an event of probability $1 - o(1)$, and
\eqref{eq:DJ-box} follows from \eqref{eq:P-eigen-generic}.
\end{proof}

\subsection{The Gaussian margin on tail sections}\label{sec:sat-gordon}

Let
\begin{equation}\label{eq:G-def}
G_i := \sqrt{\tfrac2d}\, \Gamma_i, \qquad \Gamma_i \text{ i.i.d.\ standard
Gaussian in } \Sd , \qquad i \in [n],
\end{equation}
be the covariance-matched Gaussian rows from \eqref{eq:moment-match}, independent of all $W_i$. Equivalently, $G\sim\GOE d$. Fix a head $J$ with $|J| \le K_d$, and $W_J$ that satisfies the width bound \eqref{eq:head-width}. The set $\cD_J$ of \eqref{eq:DJ-def} is then a fixed nonempty compact set.  Let $N = n - |J|$.  Apply Lemma~\ref{lem:gordon} with $T = \cD_J/\sqrt d \subseteq\sphereF$.  By \eqref{eq:head-width},
\[
w(T)\ge(\sqrt\alpha+3\eta)d,
\]
while $N\le n$ and \eqref{eq:alpha-assume} give
\[
\E\|g_N\|_2\le  \sqrt\alpha\,d.
\]
Taking $t=\eta d$ in Gordon's bound therefore leaves a margin $2\eta d$ for the standard rows $\Gamma_i$.  The scaling identity
\begin{equation}\label{eq:scaling-cancel}
\ip{G_i}{S} = \Bigl\langle \sqrt{\tfrac2d}\,\Gamma_i,\ \sqrt d\, R \Bigr\rangle
= \sqrt2\, \ip{\Gamma_i}{R} ,
\end{equation}
multiplies this margin by $\sqrt2$.  Hence, with
$a_0:=2\sqrt2\,\eta$,
\begin{equation}\label{eq:gordon-tail-margin}
\Prob\left\{
\inf_{\substack{\|z\|_2 = 1\\ \operatorname{supp} z \subseteq J^c}}\
\sup_{S \in \cD_J}\
\sum_{i \notin J} z_i \ip{G_i}{S}
\;<\; a_0\, d
\ \middle|\ W_J \right\}
\;\le\; 2\, e^{-\eta^2 d^2 / 4}.
\end{equation}
This is the margin that the remainder of the section transfers from the Gaussian rows $G_i$ to the chaos rows $W_i$, at the price of restricting to \emph{low-influence} directions $z$ and of soft-max smoothing.

\subsection{A second-chaos Lindeberg principle for soft min--max statistics}\label{sec:sat-interp}

For a trilinear form $T$ on $\Sd$, define its \emph{projective operator
norm} by
\begin{equation}\label{eq:pi-op-def}
\|T\|_{\pi,\op}
= \inf\left\{
\sum_\ell |c_\ell| \prod_{j=1}^3 \|B_{\ell j}\|_{\op} :\;
T[H_1, H_2, H_3]
= \sum_\ell c_\ell \prod_{j=1}^3 \Tr(B_{\ell j} H_j)
\right\},
\end{equation}
the infimum being over finite representations with $B_{\ell j} \in \Sd$ (such representations always exist in finite dimension), and analogously for bilinear forms.  This is the norm in which third derivatives of Gibbs functionals of the measurements are naturally controlled, and it is exactly adapted to the rank-one structure of the chaos rows, as the proof of the next lemma shows.

\begin{lemma}[Second-chaos interpolation]\label{lem:interpolation}
Let $N \ge 1$, let $W_1, \dots, W_N$ be independent copies of $(x x^\top - \Id)/\sqrt d$, and let $G_1, \dots, G_N$ be independent copies of \eqref{eq:G-def}, independent of the $W_i$'s.  Suppose $F : (\Sd)^N \to \R$ is $C^3$ with bounded first three Fr\'echet derivatives, and put
\[
M_3(F) = \sup_{A \in (\Sd)^N} \sum_{i=1}^N \|D_i^3 F(A)\|_{\pi,\op},
\qquad
R_{N,d} = \E \max_{1 \le i \le N} \|x_i\|_2^2 .
\]
Then
\begin{equation}\label{eq:interp-main}
\bigl| \E F( \{W_i \}) - \E F(\{G_i \}) \bigr|
\;\le\; \frac{4 R_{N,d}}{3\, d^{3/2}}\, M_3(F) .
\end{equation}
Moreover
\begin{equation}\label{eq:RNd}
R_{N,d} \;\le\; d + 2\sqrt{d \log N} + 2 \log N ,
\end{equation}
so that for every fixed $C_{\mathrm{row}} < \infty$, if $N \le C_{\mathrm{row}} d^2$, then for all sufficiently large $d$,
\begin{equation}\label{eq:interp-poly}
\bigl| \E F(\{W_i \}) - \E F(\{G_i \}) \bigr| \;\le\; \frac{C}{\sqrt d}\, M_3(F),
\end{equation}
with $C$ an absolute constant.
\end{lemma}

We interpolate the two expectations via the continuous path
\[
Z_i(t)=\sqrt t\,W_i+\sqrt{1-t}\,G_i,\qquad t\in[0,1],
\]
so that 
\[
\E F(\{W_i \}) - \E F(\{G_i \}) = \E F(\{Z_i(1)\}) - \E F(\{Z_i(0)\}) = \int_0^1 \frac{\de}{\de t} \E F(\{Z_i(t)\})\, \de t .
\]
Differentiating $\E F(Z(t))$ produces one term from each ensemble. The second-chaos integration-
by-parts identity below expands the quadratic term into a covariance contribution and a third-
derivative remainder, while ordinary Gaussian integration by parts expands the GOE term into the
same covariance contribution with the opposite sign. The covariance terms cancel. Lemma~\ref{lem:max-radius} controls the radial weights in the remaining sum.

\begin{lemma}[One-row second-chaos integration by parts]
\label{lem:chaos-ibp}
Fix $t\in(0,1)$ and a row $i$ in the setting of Lemma~\ref{lem:interpolation}.  Put $U_a=e_ax_i^\top+x_ie_a^\top$ and $V_{ab}=e_ae_b^\top+e_be_a^\top$, and define
\[
\cR_i(t)=\frac12\sum_{a,b=1}^d D_i^3F(\{Z_i(t) \})[U_a,U_b,V_{ab}].
\]
Then, with $\Psi=\nabla_iF(\{Z_i(t)\})$,
\begin{equation}\label{eq:ibp}
\frac{1}{2 \sqrt t}\E \ip{W_i}{\Psi} = \frac{1}{d}\, \E \operatorname{div}_i \Psi + \frac{\sqrt t}{2d^{3/2}}\, \E\, \cR_i(t) ,
\end{equation}
where, for any Frobenius-orthonormal basis $(E_\nu)_{\nu \le D_d}$ of $\Sd$,
\[
\operatorname{div}_i \Psi = \sum_{\nu=1}^{D_d} \ip{E_\nu}{D_i
\Psi[E_\nu]}.
\]
Moreover,
\begin{equation}\label{eq:Ri-bound}
||\cR_i(t)| \;\le\; 4\, \|x_i\|_2^2\, \bigl\|D_i^3 F(\{Z_i(t)\})\bigr\|_{\pi,\op} .
\end{equation}
\end{lemma}

\begin{proof}
Apply the identity $\E[(x_a x_b - \delta_{ab}) f(x)] = \E[\partial_a \partial_b f(x)]$ to the entries of $\Psi$, with $x=x_i$, and use
\[
D_x Z_i[e_a] = \sqrt{\tfrac td} \bigl( e_ax_i^\top+x_ie_a^\top \bigr) = \sqrt{\tfrac td} \, U_a,
\qquad
D_x^2 Z_i[e_a, e_b] = \sqrt{\tfrac td}\,\bigl(e_a e_b^\top + e_b
e_a^\top\bigr) = \sqrt{\tfrac td} \, V_{ab}.
\]
The chain rule gives
\begin{equation}\label{eq:ibp-expand}
\frac{1}{\sqrt d}\sum_{a,b} \E\bigl[\partial_a\partial_b \Psi_{ab}\bigr]
= \frac{\sqrt t}{d}\sum_{a,b} \E\bigl[D_i\Psi[V_{ab}]\bigr]_{ab}
+ \frac{t}{d^{3/2}} \sum_{a,b} \E\bigl[D_i^2\Psi[U_a,
U_b]\bigr]_{ab} .
\end{equation}
For $a = b$, $V_{aa} = 2E_{aa}$; for $a < b$, $V_{ab} = V_{ba} = \sqrt2\, E_{ab}$.  Grouping the two ordered off-diagonal pairs gives the exact identity
\begin{equation}\label{eq:div-identity}
\sum_{a,b} \bigl[D_i \Psi[V_{ab}]\bigr]_{ab} = 2 \sum_{\nu=1}^{D_d} \ip{E_\nu}{D_i\Psi[E_\nu]} = 2 \operatorname{div}_i \Psi .
\end{equation}
For every symmetric matrix $M$, $M_{ab} = \frac12 \ip{M}{V_{ab}}$ including for $a = b$; since $D_i^2 \Psi$ is the third derivative of $F$, the second sum in \eqref{eq:ibp-expand} is precisely $\frac12\sum_{a,b} D_i^3F[U_a, U_b, V_{ab}] = \cR_i(t)$.  Equations \eqref{eq:ibp-expand}--\eqref{eq:div-identity} prove \eqref{eq:ibp}, with all diagonal and off-diagonal factors accounted for.

To bound the remainder, take first a pure tensor $T[H_1, H_2, H_3] = \prod_{j\le3} \Tr(A_j H_j)$ with symmetric $A_j$.  A direct contraction using $\Tr(A U_a) = 2 (A x)_a$ and $\Tr(A V_{ab}) = 2 A_{ab}$ gives
\[
\frac12 \sum_{a,b} \Tr(A_1 U_a) \Tr(A_2 U_b) \Tr(A_3 V_{ab}) = 4\, x^\top A_1 A_3 A_2\, x ,
\]
whence $|\cR_i(t)| \le 4 \|x_i\|_2^2 \prod_j \|A_j\|_{\op}$ for pure tensors.  For an arbitrary trilinear form, take a representation in \eqref{eq:pi-op-def} within $\varepsilon$ of the infimum, apply the pure-tensor bound term by term, and send $\varepsilon\downarrow0$ to obtain \eqref{eq:Ri-bound}.
\end{proof}

\begin{lemma}[Expected maximum Gaussian radius]\label{lem:max-radius}
For independent $x_1,\dots,x_N\sim \cN(0,I_d)$, the quantity $R_{N,d}=\E\max_{i\le N}\|x_i\|_2^2$ satisfies \eqref{eq:RNd}.
\end{lemma}

\begin{proof}
For $Q = \|x\|_2^2 \sim \chi^2_d$ and $0 < \lambda < 1/2$,
\begin{equation}\label{eq:chi2-mgf}
\log \E\, e^{\lambda(Q - d)}
= \frac d2\bigl[-2\lambda - \log(1 - 2\lambda)\bigr]
\;\le\; \frac{d \lambda^2}{1 - 2\lambda} ,
\end{equation}
using $-\log(1-u) - u \le u^2/[2(1-u)]$ for $0 < u < 1$.  For $M_N = \max_{i \le N} \|x_i\|_2^2$, the log-sum-exp bound and Jensen give
\[
\E(M_N - d) \le \frac1\lambda \log \E\, e^{\lambda(M_N - d)}
\le \frac{\log N}{\lambda} + \frac{d\lambda}{1 - 2\lambda} .
\]
If $N > 1$, take $\lambda=\sqrt{\log N}/(\sqrt d+2\sqrt{\log N})$; the right-hand side is $2\sqrt{d\log N}+2\log N$.  For $N=1$, let $\lambda\downarrow0$.
\end{proof}

\begin{proof}[Proof of Lemma~\ref{lem:interpolation}]
Gaussian integration by parts in $G_i$, whose covariance is $(2/d) I_{\Sd}$, gives
\begin{equation}\label{eq:goe-ibp}
\frac{1}{2\sqrt{1-t}}\, \E \ip{G_i}{\nabla_i F(\{Z_i(t) \})}
= \frac1d\, \E \operatorname{div}_i \nabla_i F(\{Z_i(t) \}) .
\end{equation}
The divergence term in \eqref{eq:ibp} cancels \eqref{eq:goe-ibp} exactly, and for $0 < t < 1$ we obtain the exact remainder identity
\begin{equation}\label{eq:interp-derivative}
\frac{d}{dt} \E F(\{Z_i(t)\})
= \frac{\sqrt t}{2 d^{3/2}} \sum_{i=1}^N \E\, \cR_i(t) .
\end{equation}
Writing $a_i(A) = \|D_i^3 F(\{A_i\})\|_{\pi,\op}$, the bound \eqref{eq:Ri-bound} gives, pointwise in all $W,G$,
\begin{equation}\label{eq:radial-weight}
\sum_{i=1}^N \|x_i\|_2^2\, a_i(Z(t))
\;\le\; \Bigl(\max_{1\le i\le N} \|x_i\|_2^2\Bigr)
\sum_{i=1}^N a_i(\{Z_i(t)\})
\;\le\; \Bigl(\max_{1\le i\le N} \|x_i\|_2^2\Bigr) M_3(F(\{Z_i(t)\})).
\end{equation}
Combining \eqref{eq:interp-derivative}--\eqref{eq:radial-weight}, $|\frac{d}{dt} \E F(\{Z_i(t)\})| \le \frac{2\sqrt t}{d^{3/2}} R_{N,d}\, M_3(F)$, and integrating with $\int_0^1 2\sqrt t\, dt = 4/3$ proves \eqref{eq:interp-main}. Lemma~\ref{lem:max-radius} gives \eqref{eq:RNd}, and \eqref{eq:interp-poly} follows when $N\le C_{\mathrm{row}}d^2$.

All differentiations and interchanges above are justified on compact subintervals of $(0,1)$ by the bounded first three derivatives and Gaussian moments; bounded $DF$ makes $F$ globally Lipschitz, so $F(\{Z_i(t)\}) \to F(\{G_i\})$ as $t \downarrow 0$ and $F(\{Z_i(t)\}) \to F(\{W_i\})$ as $t \uparrow 1$ in $L^1$, and the endpoint limits used in integrating \eqref{eq:interp-derivative} are valid.
\end{proof}

We apply \eqref{eq:interp-poly} to a smooth approximation of
\[
\min_{y\in\cY}\max_{S\in\cS}\sum_i y_i\ip{A_i}{S}.
\]
The inner log-sum-exp smooths the maximum over $S$, and the outer negative log-sum-exp smooths the minimum over $y$.  The low-influence bound on $y$ will make the resulting third derivative small.  Fix $\beta,\gamma>0$, and let $\cY$ and $\cS$ be finite nonempty indexed families with
\[
\|y\|_2 = 1, \quad \|y\|_\infty \le \frac{\theta}{\sqrt d} \quad (y \in \cY),
\qquad \|S\|_{\op} \le \kappa_1 \quad (S \in \cS) .
\]
For $A = (A_i)_{i=1}^N \in (\Sd)^N$, define
\begin{equation}\label{eq:soft-def}
H_{y,S}(A) = \sum_i y_i \ip{A_i}{S},
\qquad
g_y(A) = \frac1\gamma \log \sum_{S \in \cS} e^{\gamma H_{y,S}(A)},
\qquad
\Phi(A) = -\frac1\beta \log \sum_{y \in \cY} e^{-\beta g_y(A)} .
\end{equation}

\begin{lemma}[Soft-extremum bounds]\label{lem:soft}
The functionals \eqref{eq:soft-def} satisfy
\begin{align}
\max_{S \in \cS} H_{y,S} \;\le\; g_y
\;&\le\; \max_{S \in \cS} H_{y,S} + \frac{\log |\cS|}{\gamma},
\label{eq:soft-sandwich-S}\\
\min_{y \in \cY} g_y - \frac{\log|\cY|}{\beta}
\;\le\; \Phi \;&\le\; \min_{y \in \cY} g_y ,
\label{eq:soft-sandwich-Y}\\
\sum_i \bigl\|D_i^3 \Phi\bigr\|_{\pi,\op}
\;&\le\; C \frac{\kappa_1^3\, (\beta + \gamma)^2 \theta }{\sqrt d}.
\label{eq:third-deriv-sum}
\end{align}
\end{lemma}

\begin{proof}
\eqref{eq:soft-sandwich-S}--\eqref{eq:soft-sandwich-Y} are the elementary log-sum-exp inequalities.  The derivative calculation has three terms: the inner third derivative, a covariance between the inner Hessian and gradient, and the third cumulant generated by the outer Gibbs law.  We bound each in the projective operator norm.  Fix the point at which the derivatives are evaluated.  Let $p_y$ be the inner Gibbs law on $\cS$ (with weights $\propto e^{\gamma H_{y,S}}$), let $\bar S_y = \E_{p_y} S$, and let $q$ be the outer Gibbs law on $\cY$ (with weights $\propto e^{-\beta g_y}$).  For directions $U, V, R \in \Sd$, write
\begin{equation}\label{eq:gibbs-derivs}
\begin{aligned}
\mu_i(y)[U] &= y_i \Tr(\bar S_y U),\\
B_i(y)[U,V] &= \gamma\, y_i^2\; \E_{p_y}\bigl[\Tr((S - \bar S_y)U)\, \Tr((S - \bar S_y)V)\bigr],\\
T_i(y)[U,V,R] &= \gamma^2\, y_i^3\; \E_{p_y} \prod_{Q \in \{U,V,R\}} \Tr\bigl((S - \bar S_y) Q\bigr) ;
\end{aligned}
\end{equation}
these are exactly $D_i g_y$, $D_i^2 g_y$, $D_i^3 g_y$.  Direct differentiation of the outer finite Gibbs sum gives
\begin{equation}\label{eq:D3Phi}
D_i^3 \Phi = \E_q T_i - \beta\, \operatorname{SymCov}_q(B_i, \mu_i) + \beta^2\, \operatorname{Cum}_{3,q}(\mu_i, \mu_i, \mu_i) ,
\end{equation}
where
\[
\begin{aligned}
\operatorname{SymCov}_q(B_i,\mu_i)[U,V,R] = {}& \operatorname{Cov}_q\bigl(B_i[U,V], \mu_i[R]\bigr) + \operatorname{Cov}_q\bigl(B_i[U,R], \mu_i[V]\bigr)\\
&+ \operatorname{Cov}_q\bigl(B_i[V,R], \mu_i[U]\bigr),
\end{aligned}
\]
and, with $M_i(y) = y_i \bar S_y$ and $\bar M_i = \E_q M_i(y)$,
\[
\operatorname{Cum}_{3,q}(\mu_i,\mu_i,\mu_i)[U,V,R]
= \E_q \prod_{Q \in \{U,V,R\}} \Tr\bigl((M_i(y) - \bar M_i) Q\bigr) .
\]
Since $\|S\|_{\op} \le \kappa_1$ and $\|S - \bar S_y\|_{\op} \le 2\kappa_1$, the explicit representations \eqref{eq:gibbs-derivs} give
\begin{equation}\label{eq:tensor-bounds}
\|M_i(y)\|_{\op} \le \kappa_1 |y_i|,
\qquad
\|B_i(y)\|_{\pi,\op} \le 4\gamma \kappa_1^2 y_i^2,
\qquad
\|T_i(y)\|_{\pi,\op} \le 8 \gamma^2 \kappa_1^3 |y_i|^3
\end{equation}
(the middle norm being the bilinear version of \eqref{eq:pi-op-def}).  For any one of the three placements in $\operatorname{SymCov}$, the definition of covariance, \eqref{eq:tensor-bounds}, and H\"older's inequality give
\[
\bigl\|\operatorname{Cov}_q(B_i, \mu_i)\bigr\|_{\pi,\op} \le 4\gamma\kappa_1^3\bigl(\E_q |y_i|^3 + \E_q y_i^2\, \E_q |y_i|\bigr) \le 8 \gamma \kappa_1^3\, \E_q |y_i|^3 ,
\]
and there are exactly three placements.  Likewise, Minkowski's inequality in $L^3(q)$ gives
\[
\bigl\|\operatorname{Cum}_{3,q}(\mu_i,\mu_i,\mu_i)\bigr\|_{\pi,\op}
\le \E_q \|M_i(y) - \bar M_i\|_{\op}^3
\le 8\, \E_q \|M_i(y)\|_{\op}^3
\le 8 \kappa_1^3\, \E_q|y_i|^3 .
\]
Together with $\|\E_q T_i\|_{\pi,\op} \le 8\gamma^2\kappa_1^3 \E_q|y_i|^3$, identity \eqref{eq:D3Phi} yields $\|D_i^3\Phi\|_{\pi,\op} \le C\kappa_1^3(\gamma^2 + \beta\gamma + \beta^2)\, \E_q |y_i|^3$.  Finally
\[
\sum_i \E_q |y_i|^3 = \E_q \sum_i |y_i|^3
\le \max_{y \in \cY} \sum_i |y_i|^3 \leq \max_{y \in \cY}  \| y \|_2^2 \| y \|_\infty \leq \theta/\sqrt{d},
\]
which proves \eqref{eq:third-deriv-sum}.
\end{proof}

\subsection{The uniform low-influence tail margin}\label{sec:sat-lowinf}

We now prove Proposition~\ref{prop:lowinf}.  We first choose finite nets of the capped dual sphere and of each random head section. The expectation of the soft functional is then transferred from Gaussian to quadratic-chaos rows. Clipping yields an $e^{-cd^2}$ concentration bound strong enough for a union
bound over all heads.  Finally, we remove the nets and recover \eqref{eq:lowinf}.

Because $\cD_J = \sqrt d\, (\cK_{\kappa_1}\cap\ker\cW_J\cap\sphereF )$ depends on the head $W_J$ and we later condition on $W_J$, its net must be a measurable function of the head.  We first construct a globally measurable extension: its value on bad heads is immaterial and chosen arbitrarily.  Fix a deterministic $J$, put $\Omega_J=(\Sd)^J$ and $K_* = \cK_{\kappa_1} \cap \sphereF$, and, for $w = (w_i)_{i \in J} \in \Omega_J$, set
\[
\Xi_J(w) = \bigl\{ R \in K_* :\ \ip{w_i}{R} = 0 \ \text{ for every } i
\in J \bigr\} ,
\]
so that $\sqrt d \ \Xi_J(W_J) = \cD_J$. The graph of $\Xi_J$ is closed in $\Omega_J \times K_*$, and $\Omega_J^0=\{w:\Xi_J(w)\ne\varnothing\}$ is closed by compactness of $K_*$.  On $\Omega_J^0$ this is a compact-valued Borel correspondence. The measurable maximum theorem, applied to $(w,Z,R)\mapsto\ip{Z}{R}$, makes its support function jointly Borel in $(w,Z)$. Integration in the independent Gaussian matrix $Z$ therefore shows that the section width
\[
\omega_J(w) = \E_Z\max_{R \in \Xi_J(w)} \ip{Z}{R}
\]
is Borel on $\Omega_J^0$.  Hence the set of good head configurations
\begin{equation}\label{eq:good-heads}
\fG_J = \bigl\{w \in \Omega_J^0 :\ \omega_J(w) \ge (\sqrt\alpha +
3\eta)\, d \bigr\}
\end{equation}
is Borel.  Define the globally measurable, nonempty compact-valued extension
\[
\widetilde\cD_J(w) =
\begin{cases}
\sqrt d\; \Xi_J(w), & w \in \fG_J,\\
\{\Id\}, & w \notin \fG_J ;
\end{cases}
\]
This extension has a Borel graph and nonempty compact sections, and it equals $\cD_J$ on $\fG_J$.

Fix $C_0<\infty$ and set $K_d=\lfloor C_0d\log d\rfloor$.  We leave $\theta>0$ temporarily unspecified.  Recall the raw Gaussian margin $a_0 = 2\sqrt2\,\eta$ from \eqref{eq:gordon-tail-margin}. For a tail of size $N=n-|J|$, let $G^*:\R^{J^c}\to\Sd$ be $G^*z=\sum_{i\notin J}z_iG_i$.  After identifying $\R^{J^c}$ with $\R^N$, in Frobenius coordinates it has the law of $\sqrt{2/d}\,A^\top$, with $A$ an $N\times D_d$ standard Gaussian matrix. Lemma~\ref{lem:gauss-rect} therefore supplies numerical constants $C_G,c>0$, uniform in $J$, such that
\begin{equation}\label{eq:Gstar-bound}
\Prob\bigl\{\|G^*\|_{\ell_2 \to \F} > C_G \sqrt d\bigr\}
\le 2e^{-c d^2}.
\end{equation}
On the upper feature-edge event in \eqref{eq:feature-sat}, fix $C_W$ so that $\|\cW^*\|_{2\to\F}\le C_W\sqrt d$.  Choose $\varrho\in(0,1)$ small enough so that
\begin{equation}\label{eq:varrho-choice}
C_G\, \varrho \le \frac{a_0}{8}, \qquad C_W\, \varrho \le \frac{a_0}{64} .
\end{equation}

For each $J$, consider the capped sphere
\[
\cY_J^{\mathrm{cap}}
=\{z : \operatorname{supp} z \subseteq J^c,\ \|z\|_2 = 1,\
\|z\|_\infty\le \theta/\sqrt d\}.
\]
For fixed $\theta>0$ and $|J|\le K_d$, this set is nonempty for all large $d$.  Choose a maximal $\varrho$-separated subset $\cY_J\subseteq\cY_J^{\mathrm{cap}}$; it is a deterministic $\varrho$-net.  To obtain a measurable net of $\cD_J$ with deterministic cardinality, take a deterministic $(\varrho\sqrt d/2)$-net $\{T_\ell : 1 \le \ell \le M_0\}$ of the ambient Frobenius ball $\{T\in\Sd:\|T\|_{\F}\le\sqrt d\}$ and project each $T_\ell$ measurably onto $\widetilde\cD_J(W_J)$.  Specifically, the measurable minimum theorem provides Borel selectors
\begin{equation}\label{eq:net-selector}
S_\ell(W_J) \in \arg\min_{S \in \widetilde\cD_J(W_J)} \|S - T_\ell\|_{\F} ,
\end{equation}
and we form the indexed family $\cS_J=(S_\ell(W_J))_{\ell=1}^{M_0}$.  On a good head, this is a $(\varrho\sqrt d)$-net of $\cD_J$ contained in $\cD_J$.  The usual packing--volume comparison gives $|\cY_J|\le(1+2/\varrho)^N$ and $M_0 \le (1 + 4/\varrho)^{D_d}$, whence
\begin{equation}\label{eq:net-entropy}
\log |\cY_J| + \log |\cS_J| \;\le\; C_\varrho\, d^2 .
\end{equation}

Set $c_0=a_0/2$ and choose the fixed constant $L$ large enough so that
\begin{equation}\label{eq:L-choice}
\frac{C_\varrho}{L} \;\le\; \frac{c_0}{16} .
\end{equation}
For the two nets above, use the soft functional \eqref{eq:soft-def} with
\begin{equation}\label{eq:beta-gamma}
\beta=\gamma=Ld.
\end{equation}
Let $\Phi_J(A;w)$ denote this functional with the indexed families $\cY_J$ and $\cS_J(w)$.  When $J$ and the realized head $w$ are fixed, we write simply $\Phi(A)$. Finally, set
\[
c_1=\frac{3c_0}{4},
\qquad
c_2=\frac{c_1}{2}.
\]
Thus $c_0$ is the discretized Gaussian margin, $c_1$ is the margin retained in expectation, and $c_2$ is the margin reserved for interpolation and concentration.

\begin{lemma}[Gaussian soft margin on a good head]
\label{lem:gaussian-soft-head}
Fix $J$ with $|J|\le K_d$ and condition on
$W_J=w\in\fG_J$.  Uniformly in the realized good head,
\begin{equation}\label{eq:net-margin-G}
\min_{y \in \cY_J} \max_{S \in \cS_J} \sum_{i \notin J} y_i
\ip{G_i}{S} \;\ge\; c_0\,d
\end{equation}
except with conditional probability at most
$e^{-c_{\mathrm{prob}}d^2}$.  Moreover, for all sufficiently large $d$,
\begin{equation}\label{eq:PhiG-mean}
\E_G\Phi_J(G;w)\ge c_1d.
\end{equation}
\end{lemma}

\begin{proof}
On the intersection of the event $\{\|G^*\|_{\ell_2\to\F}\le C_G\sqrt d\}$ with the Gordon margin event of \eqref{eq:gordon-tail-margin}, approximate, for each unit $y \in \cY_J$, a maximizing $S \in \cD_J$ by its net point in $\cS_J$; the loss is at most $\|G^* y\|_{\F}\, \varrho\sqrt d \le C_G \varrho\, d \le a_0 d/8$ by \eqref{eq:varrho-choice}.  Hence, with the margin constant $c_0$, \eqref{eq:net-margin-G} fails with conditional probability at most $e^{-c_{\mathrm{prob}}d^2}$ for all sufficiently large $d$, uniformly over the realized good head, where $c_{\mathrm{prob}} > 0$ is a separate fixed constant chosen no larger than $\eta^2/8$ (absorbing the prefactor in \eqref{eq:gordon-tail-margin}) and no larger than half the exponent in \eqref{eq:Gstar-bound}.

On the event \eqref{eq:net-margin-G}, \eqref{eq:soft-sandwich-S}--\eqref{eq:soft-sandwich-Y} give
\begin{equation}\label{eq:PhiG-highprob}
\Phi(G) \;\ge\; \min_{y \in \cY_J} \max_{S \in \cS_J} H_{y,S}(G) - \frac{\log|\cY_J|}{\beta} \;\ge\; \frac{15\, c_0}{16}\, d .
\end{equation}
To pass to expectations, note that for every $y, S$, $|H_{y,S}(G)| = |\ip{G^* y}{S}| \le \sqrt d\, \|G^*\|_{\ell_2\to\F}$, so the soft-extremum inequalities imply the a priori bound
\begin{equation}\label{eq:Phi-apriori}
|\Phi(G)| \;\le\; \sqrt d\, \|G^*\|_{\ell_2 \to \F} + \frac{\log|\cS_J|}{\gamma} + \frac{\log|\cY_J|}{\beta} .
\end{equation}
Integrating the full tail estimate of Lemma~\ref{lem:gauss-rect} gives $\E_G\|G^*\|_{\ell_2\to\F}^2=O(d)$. Consequently \eqref{eq:Phi-apriori} and \eqref{eq:net-entropy} give $\E_G|\Phi(G)|^2=O(d^2)$.  Let $\cE_G$ denote the event in \eqref{eq:net-margin-G} and put $p_d=\Prob_G(\cE_G^c)$. Cauchy-Schwarz and the above estimates give
\[
\begin{aligned}
\E_G\Phi(G)
&\ge \frac{15c_0}{16}d(1-p_d)
   -\E_G\bigl[|\Phi(G)|\mathbf1_{\cE_G^c}\bigr] \ge \frac{15c_0}{16}d-o(d) \ge c_1d
\end{aligned}
\]
for all sufficiently large $d$.
\end{proof}

\begin{corollary}[Transferred soft margin on a good head]
\label{cor:chaos-soft-head}
After $L$ is fixed, there is a fixed $\theta>0$, depending only on $\alpha$, such that, uniformly over $J$ and $W_J=w\in\fG_J$,
\begin{equation}\label{eq:PhiW-mean}
\E_W\Phi_J(W;w) \ge \frac{3c_1}{4}d \ge \frac{3c_2}{2}d.
\end{equation}
\end{corollary}

\begin{proof} For each fixed head $w$, the functional $\Phi_J(\cdot;w)$ is a finite composition of log-sum-exp functions with linear functionals and is therefore $C^\infty$.  Since the corresponding slopes $(y_iS)_i$ range over finite bounded families with $\|y\|_2=1$ and $\|S\|_{\op}\le\kappa_1$, its first three full Fr\'echet derivatives are globally bounded.  Thus all hypotheses of Lemma~\ref{lem:interpolation} are satisfied. Applying \eqref{eq:third-deriv-sum} with $\beta = \gamma = Ld$ gives
\begin{equation}\label{eq:M3-bound}
M_3(\Phi) = \sup_A \sum_{i \notin J} \|D_i^3 \Phi(A)\|_{\pi,\op}
\;\le\; C\kappa_1^3\, (2Ld)^2\, \frac{\theta}{\sqrt d}
\;\le\; C_{\kappa_1} L^2\, \theta\, d^{3/2} .
\end{equation}
Conditionally on $W_J$, the nets and hence $\Phi$ are fixed, while the tail vectors $x_i$, $i \notin J$, remain independent standard Gaussians; moreover $N \le n \le d^2$ for all large $d$ by \eqref{eq:alpha-assume}. Thus the polynomial-row bound \eqref{eq:interp-poly} applies conditionally, with a constant uniform in the realized good head (the quantity $R_{N,d}$ in Lemma~\ref{lem:interpolation} is an expectation over the still-unconditioned tail vectors only and does not depend on the value of $W_J$).  Equations \eqref{eq:interp-poly} and \eqref{eq:M3-bound} give
\begin{equation}\label{eq:interp-applied}
\Bigl| \E_W\bigl[\Phi(W) \,\big|\, W_J\bigr]
- \E_G\bigl[\Phi(G) \,\big|\, W_J\bigr] \Bigr|
\;\le\; C_{\kappa_1} L^2\, \theta\, d .
\end{equation}
After $L$ is fixed, choose $\theta > 0$ small enough so that the right-hand side of \eqref{eq:interp-applied} is at most $c_1 d/4$.  Then, with $c_2=c_1/2$, \eqref{eq:PhiG-mean} proves \eqref{eq:PhiW-mean}.
\end{proof}

The interpolation gives a positive conditional mean, but the union over $e^{O(d\log^2 d)}$ heads requires a faster tail bound, here $e^{-cd^2}$.  The map $x\mapsto xx^\top$ is not globally Lipschitz, so we first clip the rare large-radius rows.  Let $\pi$ be the metric projection of $\R^d$ onto the centered ball of radius $2\sqrt d$, and define
\begin{equation}\label{eq:clipped}
\overline W_i = \frac{\pi(x_i)\pi(x_i)^\top - \Id}{\sqrt d} .
\end{equation}

\begin{lemma}[Uniform clipped soft margin over all heads]
\label{lem:clipped-soft-heads}
Fix the $\theta$ of Corollary~\ref{cor:chaos-soft-head}.  Uniformly over $J$ and $W_J=w\in\fG_J$,
\begin{equation}\label{eq:conc}
\Prob\bigl\{\Phi_J(\overline W_{J^c};w)\le c_2d/2\mid W_J=w\bigr\}
\le e^{-cd^2}.
\end{equation}
Consequently,
\begin{equation}\label{eq:union-heads}
\Prob\Bigl\{\exists J\subseteq[n],\ |J|\le K_d:\
W_J\in\fG_J,\ \Phi_J(\overline W_{J^c};W_J)\le c_2d/2\Bigr\}
\;\le\; \sum_{k\le K_d}\binom nk e^{-cd^2}
\;=\;o(1).
\end{equation}
\end{lemma}

\begin{proof}
Differentiating \eqref{eq:soft-def} gives $D_i \Phi = \E_q\bigl[y_i\, \E_{p_y} S\bigr]$, whence
\begin{equation}\label{eq:DPhi-op}
\|D_i \Phi\|_{\op} \le \kappa_1\, \E_q |y_i| ,
\qquad
\sum_i \bigl(\E_q|y_i|\bigr)^2 \le \sum_i \E_q y_i^2 = 1 .
\end{equation}
Also, writing $uu^\top - vv^\top = u(u-v)^\top + (u-v)v^\top$,
\begin{equation}\label{eq:rank-one-lip}
\Bigl\|\frac{uu^\top - vv^\top}{\sqrt d}\Bigr\|_{*}
\;\le\; 4 \|u - v\|_2
\qquad \text{when } \|u\|_2, \|v\|_2 \le 2\sqrt d .
\end{equation}
Along a matrix segment, duality between operator and nuclear norms, Cauchy--Schwarz and \eqref{eq:DPhi-op} give
\begin{equation}\label{eq:Phi-lip}
|\Phi(A) - \Phi(A')|
\;\le\; \kappa_1 \Bigl(\sum_i \|A_i - A_i'\|_{*}^2\Bigr)^{1/2} .
\end{equation}
Since $\pi$ is $1$-Lipschitz, \eqref{eq:rank-one-lip} and \eqref{eq:Phi-lip} show that $\Phi(\overline W)$ is $4\kappa_1$-Lipschitz with respect to the Euclidean norm of the tail Gaussian coordinates.

We quantify the change of mean.  With $Q = \|x_i\|_2^2$, radial projection gives the exact identity $\|W_i - \overline W_i\|_{*} = (Q - 4d)_+/\sqrt d$, and the Chernoff bound $\Prob\{Q \ge 4d + t\} \le e^{-c(d+t)}$ yields
\begin{equation}\label{eq:clip-error-moment}
\E (Q - 4d)_+^2 \le C d^2 e^{-cd},
\qquad
\E \|W_i - \overline W_i\|_{*}^2 \le C d\, e^{-cd} .
\end{equation}
With at most $n = O(d^2)$ tail rows, \eqref{eq:Phi-lip}, Jensen's inequality and \eqref{eq:clip-error-moment} give, uniformly in $J$ and in its realized good head,
\begin{equation}\label{eq:clip-mean}
\Bigl| \E\bigl[\Phi(\overline W_{J^c}) \,\big|\, W_J\bigr]
- \E\bigl[\Phi(W_{J^c}) \,\big|\, W_J\bigr] \Bigr|
\;\le\; \kappa_1 \Bigl(\sum_{i \notin J} \E\|W_i - \overline
W_i\|_{*}^2\Bigr)^{1/2}
\;=\; O\bigl(d^{3/2} e^{-cd/2}\bigr) = o(1) .
\end{equation}
For all large $d$, \eqref{eq:PhiW-mean} and \eqref{eq:clip-mean} place the clipped conditional mean above $3c_2d/4$.  Gaussian concentration for a $4\kappa_1$-Lipschitz function \cite[Theorem~5.6]{boucheron2013concentration}, at deviation $c_2d/4$, proves \eqref{eq:conc}.

For deterministic $J$, let $\cH_J = \{W_J \in \fG_J\}$ and $\cF_J = \{\Phi_J(\overline W_{J^c};W_J) \le c_2 d/2\}$ (with $\cF_J = \varnothing$ if the capped sphere is empty).  The selectors above make these events measurable; moreover $\cH_J\in\sigma(W_J)$, and the tail rows remain independent conditional on $\sigma(W_J)$.  Consequently,
\begin{equation}\label{eq:head-tail-prob}
\Prob\bigl(\cH_J \cap \cF_J\bigr)
= \E\bigl[\mathbf 1_{\cH_J}\, \Prob(\cF_J \mid W_J)\bigr]
\;\le\; e^{-c d^2} .
\end{equation}
Since, eventually, $1 \le K_d \le n/2$ and
\begin{equation}\label{eq:binom-entropy}
\log \sum_{k \le K_d} \binom nk
\;=\; O\bigl(K_d \log(en/K_d)\bigr)
\;=\; O\bigl(d \log^2 d\bigr) = o(d^2) ,
\end{equation}
summing \eqref{eq:head-tail-prob} over the finitely many heads proves \eqref{eq:union-heads}.  The sigma-fields $\sigma(W_J)$ may vary with $J$; the displayed conditional expectation is applied separately to each head before the union bound.
\end{proof}

\begin{proof}[Proof of Proposition~\ref{prop:lowinf}]
Fix the $\theta$ supplied by Corollary~\ref{cor:chaos-soft-head}.  If the capped unit sphere in \eqref{eq:lowinf} is empty, the assertion is immediate.  Otherwise, work on the complement of the event in \eqref{eq:union-heads}.  The event $\cN = \{\max_{i \le n} \|x_i\|_2 \le 2\sqrt d\}$ has probability $1 - o(1)$ by \eqref{eq:radius}, and on it $\overline W_i = W_i$ for every row and every head simultaneously.  On the upper feature-edge event \eqref{eq:feature-sat}, every admissible $z$ (unit, supported in $J^c$, capped) lies within $\varrho$ of some $y \in \cY_J$, and
\begin{equation}\label{eq:net-to-all}
\sup_{S \in \cD_J}
\Bigl| \sum_{i \notin J} (z_i - y_i) \ip{W_i}{S} \Bigr|
\;\le\; \bigl\|\cW^*(z - y)\bigr\|_{\F}\cdot \sqrt d
\;\le\; C_W \varrho\, d ,
\end{equation}
using $\|S\|_{\F} = \sqrt d$ on $\cD_J$.  Finally, the signs in \eqref{eq:soft-sandwich-S}--\eqref{eq:soft-sandwich-Y} give the reverse passage from soft to hard extrema,
\begin{equation}\label{eq:soft-to-hard}
\min_{y \in \cY_J} \max_{S \in \cS_J} H_{y,S}
\;\ge\; \Phi - \frac{\log|\cS_J|}{\gamma} .
\end{equation}
On the complement of the event in \eqref{eq:union-heads}, and on $\cN$ and the feature-edge event, the definition of \eqref{eq:union-heads} together with \eqref{eq:L-choice} and \eqref{eq:soft-to-hard} shows that the net min-max is at least
\[
\frac{c_2 d}{2} - \frac{C_\varrho\, d}{L}
\;\ge\; \Bigl(\frac{3 c_0}{16} - \frac{c_0}{16}\Bigr) d
\;=\; \frac{a_0 d}{16} ,
\]
recalling $c_2 = c_1/2 = 3c_0/8$ and $c_0 = a_0/2$.  For any admissible $z$, choosing $y \in \cY_J$ within $\varrho$ and using $\cS_J \subseteq \cD_J$ together with \eqref{eq:net-to-all} and \eqref{eq:varrho-choice},
\[
\sup_{S \in \cD_J} \sum_{i \notin J} z_i \ip{W_i}{S}
\;\ge\; \frac{a_0 d}{16} - C_W \varrho\, d
\;\ge\; \frac{3 a_0 d}{64} .
\]
Thus \eqref{eq:lowinf} holds with, for example, $c_* := a_0/100$.  The constants were fixed in the order $(\eta, \kappa_0, \kappa_1)$, then $\varrho$, then $L$, then $\theta$, and none depends on the head-size constant $C_0$.
\end{proof}

\section{The unsatisfiable phase}\label{sec:unsat}

This section proves part (b) of Theorem~\ref{thm:main}.

\begin{theorem}[Unsatisfiable phase]\label{thm:unsat}
Let $n = n(d)$ satisfy $\liminf_{d\to\infty} n/d^2 > 1/4$, and let $x_1, \dots, x_n$ be independent $\cN(0, I_d)$ vectors.  Then
\[
\Prob\bigl\{\exists\, S \succeq 0 :\ x_i^\top S x_i = d \ \text{ for
every } 1 \le i \le n\bigr\} \longrightarrow 0 .
\]
\end{theorem}

It is enough to work at a fixed density strictly between $1/4$ and $1/2$. Choose
\begin{equation}\label{eq:gamma-choice}
\frac14<\gamma<
\min\left\{\liminf_{d\to\infty} \frac{n}{d^2} ,\frac12\right\}.
\end{equation}
For all large $d$, retain the first $\lfloor\gamma d^2\rfloor$ sample points.  If this subsystem is infeasible, then so is the full system. We therefore assume throughout this section that
\begin{equation}\label{eq:unsat-density}
n=\lfloor\gamma d^2\rfloor,\qquad
\frac{n}{d^2}\longrightarrow\gamma,\qquad
\frac14<\gamma<\frac12.
\end{equation}

\subsection{Outline and proof of Theorem \ref{thm:unsat}}\label{sec:unsat-setup}

The proof rests on the homogeneous formulation of \eqref{eq:P} in Lemma \ref{lem:homogeneous}.(i),  which we restate in the form used throughout this section.

\begin{lemma}[Homogeneous witnesses]\label{lem:homogeneous-witness}
Assume that $x_1,\ldots,x_n$ span $\R^d$.  An ellipsoid fit exists if and only if there are $S\in\Sd$ and $b\in\R$ such that
\begin{equation}\label{eq:witness-recall}
S\succeq0,\qquad \|S\|_{\F}=1,\qquad \cQ_x(S)=b\one,
\end{equation}
where $\cQ_x(S)=(x_i^\top S x_i-\Tr S)_{i=1}^n$.
\end{lemma}

The sample spans $\R^d$ almost surely for $n\ge d$, which holds eventually under \eqref{eq:unsat-density}.  Proposition~\ref{prop:feature-edge}, applied with $p=d$, gives an event of probability $1-o(1)$ on which every witness \eqref{eq:witness-recall} satisfies
\begin{equation}\label{eq:b-bound}
|b|=\frac{\|\cQ_x(S)\|_2}{\sqrt n}\le C_\gamma.
\end{equation}
Theorem \ref{thm:unsat} will follow by excluding all witnesses $(S,b)$ with probability tending to one. 

Introduce the empirical risk
\begin{equation}\label{eq:empirical-loss}
\Loss_n(u)=\frac1n\sum_{i=1}^n\varphi(u_i), \qquad \varphi(t)=1-e^{-t^2},
\end{equation}
such that $\Loss_n(\cQ_x(S)-b\one)=0$ if and only if $(S,b)$ is a witness. We show that, with high probability, $\Loss_n(\cQ_x(S)-b\one)$ is bounded away from zero uniformly over all admissible $(S,b)$ with $S\succeq0$, $\|S\|_{\F}=1$, and $|b|\le C_\gamma$. The loss function $\varphi$ is bounded by one, smooth with bounded derivatives, and satisfies
\begin{equation}\label{eq:phi-sandwich}
c_\varphi\min\{t^2,1\}\le\varphi(t)\le\min\{t^2,1\},
\qquad c_\varphi=1-e^{-1}.
\end{equation}
In particular, boundedness will prevent a few exceptional rows from dominating either the universality comparison or the bounded-difference estimate.

As described in Section \ref{sec:proof-method}, a (rescaled) candidate witness need not have bounded operator norm. We will separate $S$ into a low-rank spectral head $H$, carrying the large eigenvalues, and a diffuse bulk $B$.  Specifically, set
\begin{equation}\label{eq:epsd}
\varepsilon_d=d^{-1/4},\qquad k_d=\lceil d^{1/2}\rceil,
\end{equation}
and, for $S\succeq0$ with $\|S\|_{\F}=1$, define
\begin{equation}\label{eq:split-def}
H=S\mathbf1_{(\varepsilon_d,\infty)}(S),\qquad B=S-H,
\end{equation}
where $\mathbf1_{(\varepsilon_d,\infty)}$ is the spectral projector onto the eigenvalues of $S$ exceeding $\varepsilon_d$. Then
\begin{align}
&H,B\succeq0,\qquad HB=0,\qquad
\rank H\le\varepsilon_d^{-2}\le k_d,\qquad
\|B\|_{\op}\le\varepsilon_d,
\label{eq:split-props}\\
&\|H\|_{\F}^2+\|B\|_{\F}^2=1,\qquad
\Tr(B^3)\le\varepsilon_d\|B\|_{\F}^2\le d^{-1/4}.
\label{eq:split-norm}
\end{align}
Thus all large eigenvalues lie in a head of rank at most $\sqrt d$, while the bulk is diffuse in Schatten--$3$ norm. 

The quadratic map $\cQ_x$ decomposes into 
\begin{equation}\label{eq:quadratic-map-head-tail}
  \cQ_x(S)_i = x_i^\top S x_i - \Tr S = g_i^\top H g_i - \Tr H + z_i^\top B z_i - \Tr B.
\end{equation}
where $(g_i,z_i)$ are the components of $x_i$ aligned with $H$ and $B$ respectively.  More precisely, let $U$ be the range of $H$, set $r := {\rm rank} (H)$ and $m:=d-r$, and decompose each sample 
\begin{equation}
x_i = g_i + z_i \in U \oplus U^\perp.
\end{equation}
For convenience, we identify $(H,B)$ and $(g_i,z_i)$ with their respective coordinate sectors and set $H \in \sfS_r^+$, $B \in \sfS_m^+$, $g_i \in \R^r$, and $z_i \in \R^m$. 

For deterministic $U$, $g_i \sim_{\mathrm{iid}} \cN(0,I_r)$ and $z_i \sim_{\mathrm{iid}} \cN(0,I_m)$ independently.  Set $W_i = (z_i z_i^\top - I_m)/\sqrt{m}$. The witness residual becomes
\begin{equation}\label{eq:ellipsoid-model}
\cQ_x(S)_i - b = g_i^\top H g_i - \Tr H  - b + \sqrt{m}\,\Tr(W_i B).
\end{equation}

The proof conditions on the head contribution and compares the minimum of the empirical risk under ellipsoid bulk data with that under Gaussian bulk data. Specifically, we define the following hybrid model: let $g_i\sim_{\mathrm{iid}}\cN(0,I_r)$ and, independently, $G_i\sim_{\mathrm{iid}}\GOE m$, and define the head and bulk quadratic maps
\begin{equation}\label{eq:hybrid-map}
\begin{aligned}
\sfK(H,b)_i&=g_i^\top Hg_i-\Tr H-b,\;
&&H\in\Sym r,\ b\in\R, \\
\sfR(B)_i&=\sqrt m\,\Tr(G_iB),
&&B\in\Sym m.
\end{aligned}
\end{equation}
We compare the minimum of $\scL_n(\cQ_x(S)-b\one)$ with that of $\scL_n(\sfK(H,b)+\sfR(B))$ over all admissible $(H,B,b)$.  

The proof proceeds in three steps. First, Proposition \ref{prop:translated-univ} shows that the minimum empirical risk in the ellipsoid model \eqref{eq:ellipsoid-model} matches that of the hybrid model \eqref{eq:hybrid-map}, up to $o(1)$ in expectation, uniformly over all deterministic head contributions $\sfK(H,b)$.  Second, Proposition \ref{prop:hybrid-gap} shows that the minimum empirical risk in the hybrid model is bounded away from zero with high probability.  Third, Proposition \ref{prop:true-net-gap} transfers the risk gap in the hybrid model to a risk gap in the ellipsoid model over a suitable net of head contributions.  Combining these three steps proves Theorem \ref{thm:unsat}.

The first proposition compares a diffuse quadratic bulk with its covariance-matched Gaussian counterpart uniformly over arbitrary deterministic head contributions.

\begin{proposition}[Translated-loss bulk universality]
\label{prop:translated-univ} Fix $C_0<\infty$, $\kappa>0$, and $\varepsilon \in (0,1/4)$. Let $m\to\infty$ and let $n=n(m)$ satisfy $n/m^2 \in [\varepsilon , 1/2 - \varepsilon]$.  For each $m$, let $\cC_m$ be an arbitrary nonempty closed subset of
\begin{equation}\label{eq:diffuse-class}
\cB_m (C_0,\kappa) : =\left\{B\in\Sym m:\|B\|_{\F}\le C_0,\ 
\Tr|B|^3\le m^{-\kappa}\right\},
\end{equation}
and let $a_1,\ldots,a_n\in\R$ be arbitrary deterministic offsets. For $\{X_i\}_{i\le n} \subset \sfS^m$, define
\begin{equation}
  \scE_{\cC_m} (\{X_i\}) =\inf_{B\in\cC_m}\frac1n\sum_{i=1}^n
\varphi\bigl(a_i+\sqrt m\,\Tr(X_iB)\bigr).
\end{equation}
Let $z_1,\ldots,z_n\sim_{\mathrm{iid}}\cN(0,I_m)$ and set $W_i = (z_i z_i^\top - I_m)/\sqrt m$, and let $G_1,\ldots,G_n$ be independent $\GOE m$ matrices.  Then, as $m \to \infty$,
\begin{equation}\label{eq:univ-conclusion}
\sup_{(a_i),\,\cC_m}
\left|\E \scE_{\cC_m} (\{W_i\})-\E \scE_{\cC_m} (\{G_i\})\right| \longrightarrow0,
\end{equation}
where the supremum is over all offsets and closed subsets $\cC_m$ of \eqref{eq:diffuse-class}.
\end{proposition}

The proof follows the free-entropy interpolation of \cite[Section~4]{bandeira2025exact} on the Frobenius/Schatten-3 domain $\cB_m (C_0,\kappa)$ while keeping each row's deterministic translation throughout the replacement, and uses a Frobenius-net to transfer to the ground state. The details can be found in Section~\ref{sec:unsat-univ}.

The next proposition shows that the empirical risk in the hybrid model is bounded away from zero with high probability, uniformly over all admissible $(H,B,b)$.

\begin{proposition}[Hybrid model risk gap]\label{prop:hybrid-gap}
There are constants $e_\gamma,c_\gamma>0$, depending only on $\gamma$, such that for every $0\le r\le k_d$, with
probability at least $1-e^{-c_\gamma d^2}$,
\begin{equation}\label{eq:hybrid-gap}
\inf_{(H,B,b)}\
\Loss_n\bigl(\sfK(H,b)+\sfR(B)\bigr)\ge e_\gamma,
\end{equation}
where the infimum is over all $H\succeq 0$, $B\succeq 0$, and $b\in\R$ with $1/4\le\|H\|_{\F}^2+\|B\|_{\F}^2\le4$.
\end{proposition}

Section~\ref{sec:unsat-hybrid} proves this proposition using the loss function lower bound in \eqref{eq:phi-sandwich}:
\[
\scL_n(u) \geq \frac{1}{n} \sum_{i =1}^n c_\varphi \min\{u_i^2,1\} \geq c_\varphi \frac{ |I|}{n} + c_\varphi \frac{\| u_{I^c}\|_2^2}{n}, \qquad I = \{i: |u_i| \geq 1\}.
\]
The proof lower bounds the $\ell_2$-norm of the low-rank spectral head contribution via a uniform deviation inequality and the $\ell_2$-norm of the bulk contribution via Gordon's escape theorem, uniformly over all possible restrictions to $I^c$ with $|I| \leq \delta n$ for a suitably small $\delta \in (0,1)$. 

The third proposition combines Proposition \ref{prop:translated-univ} and Proposition \ref{prop:hybrid-gap} to transfer the risk gap from the hybrid model to the ellipsoid model, uniformly over a net of admissible heads $(H,b)$.  Fix for now a constant $\tau\in(0,1/10)$, whose final value will be chosen in the proof of Theorem \ref{thm:unsat}.  For $0\le r\le k_d$, let $\fU_r$ be a minimal $\tau$-net of $\Gr{r}{d}$ in projector operator norm.  The Grassmannian entropy bound \cite{szarek1983finite} gives 
\begin{equation}
|\fU_r|\le(C_{\mathrm{Gr}}/\tau)^{r(d-r)}. 
\end{equation} 
For $U\in\fU_r$, choose inside
\[
\{H\succeq0:\operatorname{range}H\subseteq U,\ \|H\|_{\F}\le1\}
\]
a $\tau$-net $\fH_U$ in Frobenius norm satisfying
\begin{equation}\label{eq:head-net-count}
|\fH_U|\le(1+2/\tau)^{r(r+1)/2}.
\end{equation}
Let $K_\gamma=C_\gamma+1$ and let $\frB_\tau$ be a $\tau$-net of $[-K_\gamma,K_\gamma]$, with $|\frB_\tau|\le1+2K_\gamma/\tau$.

For a net pair $(U,H)$, identify $U^\perp$ with $\R^m$, $m=d-r$, and define the admissible bulk set
\begin{equation}\label{eq:diffuse-annulus}
\cC_U(H)=\left\{B\succeq0:
\begin{array}{l}
\operatorname{range}B\subseteq U^\perp,\quad
\|B\|_{\op}\le\varepsilon_d,\\[2pt]
1/4\le\|H\|_{\F}^2+\|B\|_{\F}^2\le4
\end{array}\right\}.
\end{equation}
Pairs $(U,H)$ for which $\cC_U(H)$ is empty are omitted.  Every element of $\cC_U(H)$ satisfies
\begin{equation}\label{eq:B-schatten}
\|B\|_{\F}\le2,\qquad
\Tr|B|^3\le4d^{-1/4}.
\end{equation}
Uniformly for $r\le k_d$,
\begin{equation}\label{eq:uniform-aspect}
4d^{-1/4}\le(d-r)^{-1/5},\qquad
\frac{n}{(d-r)^2}\longrightarrow\gamma.
\end{equation}
Thus Proposition~\ref{prop:translated-univ} applies to every $\cC_U(H)$ with $C_0=2$ and $\kappa=1/5$.

Define
\[
\frV_d=\{(r,U):0\le r\le k_d,\ U\in\fU_r\}
\]
and
\[
\cI_d=\{(r,U,H,b):0\le r\le k_d,\ U\in\fU_r,\ H\in\fH_U,\
b\in\frB_\tau,\ \cC_U(H)\ne\varnothing\}.
\]
The net entropies satisfy
\begin{equation}\label{eq:total-entropy}
\log|\frV_d|+\log|\cI_d|=O_\tau(d^{3/2})=o(d^2).
\end{equation}

\begin{proposition}[Uniform risk gap on the head net]
\label{prop:true-net-gap}
For every fixed $\tau\in(0,1/10)$, with probability tending to one,
simultaneously for all $(r,U,H,b)\in\cI_d$,
\begin{equation}\label{eq:net-gap}
\inf_{B\in\cC_U(H)}
\Loss_n\bigl(\cQ_x(H+B)-b\one\bigr)\ge\frac{e_\gamma}{4}.
\end{equation}
\end{proposition}

Section~\ref{sec:unsat-nets} proves Proposition~\ref{prop:true-net-gap}. 

Finally, we need the following elementary estimate when rounding a witness to the closest head matrix on the $\tau$-net. 

\begin{lemma}\label{lem:transport}
If $U,V\in\Gr{r}{d}$ and $\|P_U-P_V\|_{\op}\le\tau<1/2$, there is an orthogonal matrix $O$ with $OU=V$ and $\|O-I\|_{\op}\le C\tau$. Consequently, for every $S\in\Sd$,
\begin{equation}\label{eq:transport-frob}
\|OSO^\top-S\|_{\F}\le C\tau\|S\|_{\F}.
\end{equation}
\end{lemma}

\begin{proof}
Let $\theta_1,\ldots,\theta_r$ be the principal angles between $U$ and
$V$, and choose corresponding principal orthonormal bases
$(u_i)_{i\le r}$ of $U$ and $(v_i)_{i\le r}$ of $V$.  Define $O$ by
rotating $u_i$ to $v_i$ through angle $\theta_i$ in each principal
two-dimensional plane, and by acting as the identity on the orthogonal
complement of these planes.  Then $O$ is orthogonal and $OU=V$. Writing $\theta_{\max}=\max_i\theta_i$, the standard principal-angle
identity gives
\[
\sin\theta_{\max} = \max_{i \in [r]} \sin \theta_i =\|P_U-P_V\|_{\op}\le\tau.
\]
Since the operator norm of a planar rotation through angle $\theta$ minus
the identity is $2\sin(\theta/2)$,
\[
\|O-I\|_{\op}
=2\sin(\theta_{\max}/2)
\le2\sin\theta_{\max}
\le2\tau.
\]
Finally, $OSO^\top-S=(O-I)SO^\top+S(O^\top-I)$  and therefore
\[
\|OSO^\top-S\|_{\F}
\le2\|O-I\|_{\op}\|S\|_{\F}
\le4\tau\|S\|_{\F},
\]
which proves \eqref{eq:transport-frob}.
\end{proof}

\begin{proof}[Proof of Theorem~\ref{thm:unsat}]
Fix $\tau\in(0,1/10)$ sufficiently small as specified below, and use the
corresponding deterministic net $\cI_d$.  Intersect the feature-edge event (Proposition \ref{prop:feature-edge}) with the event of Proposition~\ref{prop:true-net-gap}.  This intersection has probability tending to one. In particular, Proposition~\ref{prop:feature-edge} provides \eqref{eq:b-bound} and, simultaneously for all $S,T\in\Sd$,
\begin{equation}\label{eq:Q-lip}
\frac1{\sqrt n}\|\cQ_x(S)-\cQ_x(T)\|_2
\le C_\gamma\|S-T\|_{\F}.
\end{equation}

Suppose that a homogeneous witness exists on this event:
\begin{equation}\label{eq:witness-normalized}
S\succeq0,\qquad \|S\|_{\F}=1,\qquad
\cQ_x(S)=b\one,
\end{equation}
where $|b|\le C_\gamma<K_\gamma$.  Split $S=H+B$ as in
\eqref{eq:split-def}, put $r=\rank H\le k_d$, and let
$U=\operatorname{range}H$.

Choose $U_0\in\fU_r$ with $\|P_U-P_{U_0}\|_{\op}\le\tau$, and let $O$ be the orthogonal matrix from Lemma~\ref{lem:transport}.  Define
\begin{equation}\label{eq:S0-def}
S_0=OSO^\top=H_0+B_0,\qquad
H_0=OHO^\top,\quad B_0=OBO^\top.
\end{equation}
Then $H_0$ is supported on $U_0$, $B_0$ on $U_0^\perp$, and
\begin{equation}\label{eq:S0-props}
\|S_0-S\|_{\F}\le C\tau,\qquad
\|H_0\|_{\F}\le1,\qquad B_0\succeq0,\qquad
\|B_0\|_{\op}\le\varepsilon_d.
\end{equation}
Choose $H_1\in\fH_{U_0}$ and $\bar b\in\frB_\tau$ with
\begin{equation}\label{eq:round}
\|H_1-H_0\|_{\F}\le\tau,\qquad |\bar b-b|\le\tau.
\end{equation}
Since $H_0$ and $B_0$ occupy orthogonal blocks,
\begin{equation}\label{eq:norm-preserved}
\left|\|H_1\|_{\F}^2+\|B_0\|_{\F}^2-1\right|
\le2\tau.
\end{equation}
For $\tau<1/10$, this implies that $B_0 \in \cC_{U_0}(H_1)$. Put $\widetilde S=H_1+B_0$.  

By linearity, \eqref{eq:witness-normalized}, \eqref{eq:Q-lip}, and \eqref{eq:S0-props}--\eqref{eq:round},
\begin{equation}\label{eq:Qtilde-close}
\frac1{\sqrt n}
\|\cQ_x(\widetilde S)-\bar b\one\|_2
\le C_\gamma\|\widetilde S-S\|_{\F}+\tau
\le C_\gamma'\tau.
\end{equation}
The upper bound in \eqref{eq:phi-sandwich} gives
\begin{equation}\label{eq:small-loss}
\Loss_n\bigl(\cQ_x(\widetilde S)-\bar b\one\bigr)
\le(C_\gamma')^2\tau^2.
\end{equation}
Choose the fixed $\tau$ small enough that $(C_\gamma')^2\tau^2<e_\gamma/4$.  Then \eqref{eq:small-loss} contradicts Proposition~\ref{prop:true-net-gap} at the tuple $(r,U_0,H_1,\bar b)$ and the point $B_0\in\cC_{U_0}(H_1)$.  Thus no homogeneous witness exists on the intersection event.  Lemma~\ref{lem:homogeneous-witness} rules out a fit for the ellipsoid model and proves the theorem.
\end{proof}

\begin{remark}[Where the sharp constant $1/4$ enters]\label{rem:unsat-quarter}
The spectral head has rank $o(d)$ and deleting an arbitrarily small fixed fraction of rows leaves a Gaussian bulk of approximately $(1-\delta)\gamma d^2$ equations.  The Gaussian width of the positive semidefinite cone in $\Sym{d - o(d)}$ is $d/2 + o(d)$ (Lemma~\ref{lem:psd-width}).  The strict inequality required by Gordon's escape \eqref{eq:gap-mean} is exactly $(1-\delta)\gamma > 1/4$, which is achievable precisely when $\gamma > 1/4$.  All other ingredients --- small ball, VC, universality, McDiarmid, nets --- work at any density $\gamma \in (0, 1/2)$.
\end{remark}

\subsection{Bulk universality with translated losses}\label{sec:unsat-univ}

We now prove Proposition~\ref{prop:translated-univ}.  For convenience, we introduce the rescaled set
\begin{equation}\label{eq:replica-domain}
\cT_m(C_0,\kappa) := \sqrt{m} \,\cB_m (C_0,\kappa) =  \left\{T\in\Sym m:\ \|T\|_{\F}\le C_0\sqrt m, \ \Tr|T|^3\le m^{3/2-\kappa}\right\}.
\end{equation}
Note that $\cT_m(C_0,\kappa)$ is a symmetric convex domain.  Fix $\beta >0$ and row offsets $(a_i)_{i \leq n} \in \R^n$, and consider the uniform distribution over a finite subset $\cA_m \subset \cT_m(C_0,\kappa)$.  For $\{X_i\}_{i\le n} \subset \sfS^m$, we introduce the free entropy 
\begin{equation}\label{eq:free-entropy}
\scF_{\cA_m} (\{X_i\}) := \frac{1}{m^2} \log\left\{
\frac{1}{|\cA_m|} \sum_{B \in \cA_m} \exp\Bigl[-\beta \sum_{i=1}^n \varphi\bigl(a_i + \Tr(X_i B)\bigr)\Bigr] \right\} ,
\end{equation}
where the dependence on $\beta$ and $(a_i)$ is left implicit.  By the log-sum-exp inequality \eqref{eq:soft-sandwich-Y},
\begin{equation}\label{eq:finite-prior-soft-min}
\begin{aligned}
\min_{B\in\cA_m}\frac1n\sum_{i=1}^n
\varphi\bigl(a_i+\Tr(X_i B)\bigr)
\le
- \frac{m^2}{\beta n} \scF_{\cA_m} \le
\min_{B\in\cA_m}\frac1n\sum_{i=1}^n
\varphi\bigl(a_i+\Tr(X_i B)\bigr)
+\frac{\log|\cA_m|}{\beta n}.
\end{aligned}
\end{equation}

The next proposition bounds the difference in the expected free entropy \eqref{eq:free-entropy} between quadratic-chaos rows and Gaussian rows. Importantly, this comparison is \emph{uniform} over $1 \leq n \leq \Lambda m^2$, the row offsets $(a_i)_{i \leq n}$, and the finite subset $\cA_m \subseteq \cT_m (C_0,\kappa)$. It is a straightforward modification of the free entropy universality proved by Bandeira and Maillard \cite[Section 4]{bandeira2025exact}, and we only sketch the differences.

\begin{proposition}[Free entropy comparison]
\label{prop:finite-prior-free-entropy}
Fix $\beta>0$, $C_0<\infty$, $\kappa>0$, and $\Lambda<\infty$. Let $z_1,\ldots,z_n\stackrel{\mathrm{iid}}{\sim}\cN(0,I_m)$ and set $W_i = (z_i z_i^\top - I_m)/\sqrt m$, and let $G_1,\ldots,G_n$ be independent $\GOE m$ matrices, independent of the $W_i$. Then, as $m \to \infty$,
\begin{equation}\label{eq:FE-univ}
 \sup_{n,  (a_i),  \cA_m} \; \bigl|\E \scF_{\cA_m} (\{W_i\}) -\E \scF_{\cA_m}(\{G_i\}) \bigr|\longrightarrow0 .
\end{equation}
where the supremum is over all $1\le n\le\Lambda m^2$, $(a_i)_{i\le n}\in\R^n$, and finite subsets $\cA_m\subset \cT_m(C_0,\kappa)$.
\end{proposition}

\begin{proof}
The free-entropy universality theorem \cite[Theorem 4.6]{bandeira2025exact} is stated for a common row loss. Its proof is rowwise and remains valid for losses $\ell_i$ with uniform bounds on the norms entering their domination and finite-projection estimates. Take $\ell_i (u) = \beta \varphi (a_i +u)$: then $\| \ell^{(k)}_{i,m} \|_\infty \leq \beta \| \varphi^{(k)} \|_\infty$ uniformly over $i$ and $a_i$. The support condition follows from $\cA_m \subseteq \cT_m$, while their Schatten-3 pointwise-normality lemma applies uniformly on $\cT_m$. Finally $\scF_m \in [-C\beta,0]$, so their theorem may be applied with a bounded smooth test function equal to the identity on this interval. This gives the conclusion uniformly in the offsets and finite prior.
\end{proof}

Through a Frobenius $\delta$-net of $\cC_m \subseteq \cB_m (C_0,\kappa)$, the free-entropy comparison \eqref{eq:FE-univ} and the sandwich \eqref{eq:finite-prior-soft-min} yield a comparison of the ground states $\scE_{\cC_m} ( \{W_i\})$ and $\scE_{\cC_m} ( \{G_i\})$.

\begin{lemma}[Ground-state transfer]\label{lem:ground-state-transfer}
Under the assumptions of Proposition~\ref{prop:translated-univ}, there exists a constant $C>0$ depending only on $\varepsilon$, $C_0$, and $\kappa$, such that for
every fixed $\delta\in(0,1)$ and $\beta>0$,
\begin{equation}\label{eq:ground-state-transfer}
\limsup_{m\to\infty}\ \sup_{(a_i),\,\cC_m}
\bigl|\E \scE_{\cC_m} ( \{W_i\}) -\E \scE_{\cC_m} ( \{G_i\})\bigr|
\le C\delta+\frac{C}{\beta}
\log\Bigl(1+\frac{2C_0}{\delta}\Bigr),
\end{equation}
where the supremum is over all $(a_i)_{i\le n}\in\R^n$ and all nonempty closed subsets $\cC_m\subseteq\cB_m(C_0,\kappa)$.
\end{lemma}

\begin{proof}[Proof of Lemma~\ref{lem:ground-state-transfer}]
Fix $\delta \in (0,1)$ and let $\cA_m$ be a maximal $\delta$-separated subset of $\cC_m$; maximality and compactness imply that $\cA_m$ is a finite $\delta$-net.  The disjoint Frobenius balls of radius $\delta/2$ centered at its points lie in the Frobenius ball of radius $C_0 + \delta/2$, so with $D_m = m(m+1)/2$,
\begin{equation}\label{eq:net-count}
\log |\cA_m| \;\le\; D_m \log\Bigl(1 + \frac{2 C_0}{\delta}\Bigr) .
\end{equation}
For either ensemble write
\[
E(\{ X_i\};B) = \frac1n \sum_{i=1}^n \varphi\bigl(a_i + \sqrt m \Tr(X_i \,
B)\bigr),
\qquad
\scE_{\cA_m}( \{ X_i\}) = \min_{B \in \cA_m} E( \{ X_i\};B) .
\]
Denoting $\scS_{\cA_m} (\{X_i\})= -\frac{m^2}{\beta n} \scF_{\sqrt m \cA_m} (\{X_i\})$, \eqref{eq:finite-prior-soft-min} gives
\begin{equation}\label{eq:soft-min-sandwich}
\scE_{\cA_m}( \{ X_i\}) \le \scS_{\cA_m} (\{X_i\}) \le \scE_{\cA_m}( \{ X_i\}) + 
\frac{\log|\cA_m|}{\beta n} .
\end{equation}
Putting together equations \eqref{eq:FE-univ} in Proposition \ref{prop:finite-prior-free-entropy}, \eqref{eq:net-count}, and
\eqref{eq:soft-min-sandwich},
\begin{equation}\label{eq:net-univ}
\limsup_{m \to \infty}\ \sup_{(a_i),\,\cC_m}
\bigl| \E \scE_{\cA_m}( \{ W_i\}) - \E \scE_{\cA_m}( \{ G_i\}) \bigr|
\;\le\; \frac{C}{\beta} \log\Bigl(1 + \frac{2C_0}{\delta}\Bigr) .
\end{equation}

It remains to compare the minimum $\scE_{\cA_m}$ over the $\delta$-net $\cA_m$ with the infimum $\scE_{\cC_m}$ over $\cC_m$.  Let $\cE_{\rm ell}$ be the feature-edge event \eqref{eq:feature-lip}, and let $\cE_{\rm goe}$ be the Gaussian operator-norm event in Proposition~\ref{prop:feature-edge}, both with $p=m$.  Work on $\cE_{\rm ell}\cap\cE_{\rm goe}$.  Note that $\sqrt m \Tr(X_i (B - B'))$ equals $z_i^\top(B-B')z_i - \Tr(B - B')$ for $X = W$ and is the map $\cG$ of Proposition~\ref{prop:feature-edge} for $X = G$.  On the corresponding event, for $B, B' \in \cC_m$,
\begin{equation}\label{eq:energy-lip}
\begin{aligned}
|E(\{ X_i\};B) - E(\{ X_i\};B')|
&\le \frac{\|\varphi'\|_\infty}{n} \sum_{i=1}^n
\bigl|\sqrt m \Tr(X_i  (B - B'))\bigr|\\
&\le \frac{\|\varphi'\|_\infty}{\sqrt n}
\bigl\|\bigl(\sqrt m \Tr(X_i   (B-B'))\bigr)_i\bigr\|_2
\;\le\; C \|B - B'\|_{\F} .
\end{aligned}
\end{equation}
The probability of the event $\cE_{\rm ell}\cap\cE_{\rm goe}$ tends to one as $m \to \infty$ uniformly over $n/m^2 \in [\varepsilon, 1/2 - \varepsilon]$.  On that event, choosing for each $B \in \cC_m$ a net point within $\delta$ shows the two infima differ by at most $C\delta$; on the complementary event, both values lie in $[0,1]$.  Consequently
\begin{equation}\label{eq:net-to-set}
\sup_{(a_i),\,\cC_m}\ \max_{X\in\{W,G\}}
\E\Bigl| \scE_{\cA_m}( \{ X_i\}) - \inf_{B \in \cC_m} E(\{ X_i\};B) \Bigr|
\;\le\; C \delta + o(1) .
\end{equation}
The triangle inequality, together with
\eqref{eq:net-univ}--\eqref{eq:net-to-set}, proves
\eqref{eq:ground-state-transfer}.
\end{proof}

\begin{proof}[Proof of Proposition~\ref{prop:translated-univ}]
Apply Lemma~\ref{lem:ground-state-transfer}.  In \eqref{eq:ground-state-transfer}, first take $m\to\infty$ with $\beta$ and $\delta$ fixed.  We may then let $\beta\to\infty$ and finally $\delta\downarrow0$.  This proves \eqref{eq:univ-conclusion}; all three passages are uniform in the offsets and in the closed subset $\cC_m$.
\end{proof}

\subsection{Hybrid model risk lower bound}\label{sec:unsat-hybrid}

We now prove Proposition~\ref{prop:hybrid-gap}. Let $u = \sfK (H,b) + \sfR (B)$ denote the residual vector of the hybrid model \eqref{eq:hybrid-map}. The goal is to show that, with high probability, the empirical risk $\scL_n (u)$ is bounded away from zero uniformly over all admissible $(H,B,b)$:
\begin{equation}\label{eq:admissible-hybrid}
  H \succeq 0, \quad B \succeq 0, \quad b \in \R, \quad \frac14 \leq \| H \|_\F^2 + \| B \|_\F^2  \leq 4.
\end{equation}
 
The proof uses the lower-bound on the loss function $\varphi$ in \eqref{eq:phi-sandwich}
 \begin{equation}\label{eq:lower-bound-risk}
  \scL_n(u) = \frac{1}{n} \sum_{i=1}^n \varphi(u_i) \geq \frac{c_\varphi}{n} \sum_{i=1}^n \min(1,u_i^2),
 \end{equation}
 and distinguishes three cases. For any subset of indices $J \subseteq [n]$, write $u_J = (u_i)_{i \in J}$, $\sfK (H,b)_J = ( \sfK (H,b)_{i})_{i \in J}$, and $\sfR (B)_J = ( \sfR (B)_{i})_{i \in J}$. Let $\theta \in (0,1/4)$ and $\delta \in (0,1)$ be constants chosen later. For each admissible parameter, let $I = \{ i \in [n]: |u_i | >1\}$ be the set of indices where the hybrid-model residual is large.  The three cases are:
\begin{enumerate}
  \item $|I| \geq \delta n$. The risk lower bound \eqref{eq:lower-bound-risk} immediately gives $\scL_n(u) \geq c_\varphi |I | / n \geq c_\varphi \delta$.
  
    \item $|I| < \delta n$ and $\| B \|_\F > \theta$.  Put $J = [n]\setminus I$. Lemma \ref{lem:projected-bulk-escape} gives a projection $P_J$ off the head-feature range such that 
  \begin{equation}\label{eq:tail-dominating-bound}
    P_J \sfK (H,b)_J = 0, \qquad \| P_J \sfR (B)_J \|_2 \geq c_{\rm bulk} \sqrt{n} \| B \|_\F.
  \end{equation} 
  Thus, the risk lower bound \eqref{eq:lower-bound-risk} gives in this case
  \begin{equation}
    \scL_n(u) \geq \frac{c_\varphi}{n} \|  P_J u_J \|_2^2 = \frac{c_\varphi}{n} \| P_J \sfR (B)_J \|_2^2 \geq c_\varphi c_{\rm bulk}^2 \theta^2.
  \end{equation}

  \item $|I| < \delta n$ and $\| B \|_\F \leq \theta$.  Again put $J = [n]\setminus I$. Now \eqref{eq:admissible-hybrid} implies $\| H \|_\F^2 \geq 1/4 -\theta^2 \geq 3/16$.  Lemmas \ref{lem:head-smallball} and \ref{lem:projected-bulk-escape}, followed by the reverse triangle inequality, give
  \begin{equation}\label{eq:head-dominating-bound}
    \begin{aligned}
    \| \sfK (H,b)_J + \sfR (B)_J \|_2 \geq &~ \| \sfK (H,b)_J\|_2 - \| \sfR (B)_J \|_2  \\
    \geq&~ a_0 \sqrt{n} (2\| H \|_F^2 + b^2)^{1/2} - C_0 \sqrt{n} \| B \|_F^2 \geq c_{\rm head} \sqrt{n}.
    \end{aligned}
  \end{equation}
Hence \eqref{eq:lower-bound-risk} gives $\scL_n(u) \geq c_\varphi \| u_J \|_2^2/n \geq c_{\varphi} c_{\rm head}^2$.

\end{enumerate}
Importantly, the estimates \eqref{eq:head-dominating-bound} and \eqref{eq:tail-dominating-bound} hold with high probability \emph{simultaneously} over all subsets $J \subseteq [n]$ with $|J| \geq (1-\delta) n$ and all $(H,B,b)$ satisfying \eqref{eq:admissible-hybrid}. The constants are chosen such that the probability of these events is at least $1 - e^{-c_\gamma d^2}$.

The remainder of this section makes this argument precise. We prove Lemma \ref{lem:head-smallball} and Lemma \ref{lem:projected-bulk-escape} and then assemble them into the complete proof of Proposition~\ref{prop:hybrid-gap}.

\begin{lemma}[Uniform head small ball]\label{lem:head-smallball}
There are universal constants $a_0,\delta_0,c_0 >0$ such that, for all sufficiently large $d$ and every $1 \le r \le k_d$, with probability at least $1-e^{-c_0n}$, 
\begin{equation}\label{eq:smallball}
  \bigl\| \sfK (H,b)_J \bigr\|_2
\;\ge\; a_0 \sqrt n\, \bigl(2\|H\|_{\F}^2 + b^2\bigr)^{1/2} , 
\end{equation}
simultaneously for all $H\in\Sym{r}$, $b\in\R$, and $J\subseteq[n]$ with $|J|\ge(1-\delta_0)n$. 
\end{lemma}

\begin{proof}
Let $q = r(r+1)/2 + 1$ and let $v \in \R^q$ be the vector consisting of the coordinates of $H$ followed by $b$; set $Y_v (g) = g^\top H g - \Tr H - b$ with $g \sim \cN(0, I_r)$. Let $f$ denote the $q$-dimensional feature vector consisting of the coordinates of $g g^\top - I_r$ followed by the constant $-1$.  Then $Y_v = v^\top f$ and $\sfK(H,b)_i = Y_v (g_i) = v^\top f_i$ with $f_i$ the feature vector corresponding to $g_i$. 

Basic Gaussian properties give the variance identity
\begin{equation}
\E Y_v^2 = 2 \|H\|_{\F}^2 + b^2 ,
\end{equation}
and $\|Y_v\|_{L^4} \le 3 \|Y_v\|_{L^2}$ by Gaussian hypercontractivity for polynomials of degree at most two. For every $v \neq 0$, Paley--Zygmund inequality then gives
\begin{equation}\label{eq:PZ}
 \Prob\Bigl\{ |Y_v|^2 \geq \frac{1}{2} \bigl(2\|H\|_{\F}^2 +
b^2\bigr) \Bigr\} \;\ge\; \frac{1}{4}\frac{(\E|Y_v|^2)^2}{\E|Y_v|^4} = \frac{1}{324} =: p_0.
\end{equation}
The sets $\{g : |Y_v(g)|^2 \ge \tfrac12 (2\|H\|_{\F}^2 + b^2)\}$ are inverse images of unions of two affine half-spaces of $\R^q$.  If $\Pi_{\mathrm{hs}}(N)$ denotes the shatter coefficient of affine half-spaces in $\R^q$ (VC dimension $q+1$), Sauer's lemma gives, for $N \ge q + 1$, $\Pi_{\mathrm{hs}}(N) \le (eN/(q+1))^{q+1}$, and a union of two half-spaces is specified by a pair, so
\begin{equation}\label{eq:shatter}
\Pi_{\mathrm{union}}(N) \le \Pi_{\mathrm{hs}}(N)^2
\le \Bigl(\frac{eN}{q+1}\Bigr)^{2(q+1)} .
\end{equation}
The Vapnik--Chervonenkis uniform deviation inequality \cite[Theorem~12.5]{devroye2013probabilistic} states that for any class $\cC$ of measurable sets and empirical measure $\Prob_n$ of $n$ i.i.d.\ samples, $\Prob\{\sup_{C \in \cC} |\Prob_n(C) - \Prob(C)| > t\} \le 8\, \Pi_{\cC}(2n)\, e^{-n t^2/32}$.  Applying this with $t = p_0/2$, using \eqref{eq:PZ} and \eqref{eq:shatter}, we get
\begin{equation}\label{eq:VC-bound}
\Prob\left\{
\inf_{v \neq 0} \frac1n \#\Bigl\{i :\, |Y_v(g_i)| \ge 2^{-1/2}
\bigl(2\|H\|_{\F}^2 + b^2\bigr)^{1/2}\Bigr\} < \frac{p_0}{2}
\right\}
\le 8 \Bigl(\frac{2en}{q+1}\Bigr)^{2(q+1)} e^{-p_0^2 n/128}
\le e^{-c_0 n},
\end{equation}
where we used $q = o(n)$ in the last inequality.  Choose $\delta_0 = p_0/4$.  On the complementary event, deleting at most $\delta_0 n$ rows leaves at least $p_0 n/4$ of the counted coordinates; summing their squares proves \eqref{eq:smallball} with $a_0 = \sqrt{p_0/8}$ (the case $v = 0$ being trivial). 
\end{proof}

We write $H_2(t):=-t\log t-(1-t)\log(1-t)$ for binary entropy.  Set $\Delta=\sqrt\gamma-\tfrac12>0$ and choose $0<\delta<\delta_0$ small enough that
\begin{equation}\label{eq:delta-choice}
\sqrt{(1-\delta)\gamma}-\frac12\ge\frac{3\Delta}{4},
\qquad
\gamma H_2(\delta)<\frac{\Delta^2}{128},
\qquad
\delta<\frac12.
\end{equation}

Fix $\{ g_i \}_{i \leq n}$ for which \eqref{eq:smallball} holds for every $H,b$ and every $J\subseteq[n]$ with $|J|\ge(1-\delta)n$. Note that $\sfR(B)$ is independent of this realization. Let $F = [f_1,\ldots,f_n]^\top \in \R^{n \times q_r}$ where $f_i$ is the feature vector associated to $g_i$ as defined in the proof of Lemma \ref{lem:head-smallball}, with $q_r=r(r+1)/2+1$.  For every $J\subseteq[n]$ with $|J|\ge(1-\delta)n$, the small-ball bound \eqref{eq:smallball} ensures $F_J := (f_j)_{j \in J} \in \R^{|J| \times q_r}$ has full column rank, so the orthogonal projection off the head range is
\begin{equation}\label{eq:PJ-def}
P_J=I_{|J|}-F_J(F_J^\top F_J)^{-1}F_J^\top.
\end{equation}
Thus $P_J\sfK(H,b)_J=0$ for every $(H,b)$: the projection removes the entire
head contribution and leaves only the Gaussian bulk.

\begin{lemma}[Uniform projected Gaussian bulk]
\label{lem:projected-bulk-escape}
Condition on $\{ g_i \}_{i \leq n}$ satisfying the small-ball condition \eqref{eq:smallball}.  There are a universal constant $C_0>0$ and a constant $c_{\mathrm{bulk}}>0$ depending only on $\gamma$, such that, for all sufficiently large $d$ and every $0\le r\le k_d$, with conditional probability at least $1-e^{-\Delta^2d^2/64}-e^{-n/4}$,
\begin{align}
\|P_J\sfR(B)_J\|_2 \ge c_{\mathrm{bulk}}\sqrt n\,\|B\|_{\F}, \qquad \| \sfR (B)_J \|_2 \leq C_0 \sqrt n \| B \|_\F,
\label{eq:bulk-lower}
\end{align}
simultaneously for all $B \succeq 0$, and $J\subseteq[n]$ with $|J|\ge(1-\delta)n$.
\end{lemma}

\begin{proof} 
  Recall that $\sfR(B)_i = \sqrt{m} \Tr (G_i B)$ with $G_i \stackrel{\rm iid}{\sim} \GOE m$. Let $v$ denote the $D_m$-dimensional coordinate vector of $B$, and $\gamma_i$ that of $G_i$. Then we can write $\sfR(B) =  \Gamma v$, where $\Gamma$ is an $n\times D_m$ matrix with independent $\cN(0,2)$ entries. Let $\Gamma_J$ be the submatrix of $\Gamma$ with rows indexed by $J$. Let $U_J\in\R^{|J|\times s_J}$ have orthonormal columns spanning $\operatorname{range}(P_J)$. Hence $U_J^\top\Gamma_J$ has independent $\cN(0,2)$ entries and
\begin{equation}\label{eq:sJ}
s_J=|J|-q_r\ge(1-\delta)n-q_r.
\end{equation}

Lemma~\ref{lem:psd-width} and \eqref{eq:chi-mean} give
\[
w(\psdsphere m)\le\sqrt{\frac{D_m}{2}},
\qquad
\E\|g_{s_J}\|_2\ge\sqrt{s_J-1}.
\]
Uniformly for $r\le k_d=o(d)$, we have $m/d\to1$ and
$q_r/d^2\to0$.  Therefore, by \eqref{eq:delta-choice},
\begin{equation}\label{eq:gap-mean}
\begin{aligned}
\E\|g_{s_J}\|_2-w(\psdsphere m)
&\ge\sqrt{(1-\delta)n-q_r-1}-\sqrt{\frac{D_m}{2}} =d \Big(\sqrt{(1-\delta)\gamma}-\frac12 + o(1) \Big)
\ge\frac{\Delta d}{2}.
\end{aligned}
\end{equation}
Apply Lemma~\ref{lem:escape} to
$2^{-1/2}U_J^\top\Gamma_J$ and $\psdsphere m$, with deviation
$\Delta d/4$.  For each fixed $J$, conditionally on the head,
\begin{equation}\label{eq:bulk-escape}
\Prob\left\{
\inf_{B\in\psdsphere m}\|P_J\sfR (B)_J\|_2<c_1d
\ \middle|\ (g_i)\right\}
\le e^{-\Delta^2d^2/32},
\end{equation}
where $c_1>0$ depends only on $\gamma$ and we used
$\|P_J\sfR (B)_J\|_2=\|U_J^\top\Gamma_J v \|_2$.

There are at most
\begin{equation}\label{eq:deletion-count}
\sum_{j\le\delta n}\binom nj\le e^{nH_2(\delta)}
\end{equation}
admissible $J$ sets.  Since
$nH_2(\delta)\le\Delta^2d^2/128$ for all large $d$, the conditional union
bound in \eqref{eq:bulk-escape} is at most
\begin{equation}\label{eq:deletion-union}
\exp\left\{nH_2(\delta)-\frac{\Delta^2d^2}{32}\right\}
\le\exp\left\{-\frac{3\Delta^2d^2}{128}\right\}.
\end{equation}
Homogeneity and $n/d^2\to\gamma$ now give \eqref{eq:bulk-lower} with
$c_{\mathrm{bulk}}=c_1/(\sqrt\gamma+1)$.

For the second bound in \eqref{eq:bulk-lower}, Lemma~\ref{lem:gauss-rect}, with $t=\sqrt n$, gives
\[
\|\Gamma\|_{\op} \le\sqrt2(2\sqrt n+\sqrt{D_m})\le C_0 \sqrt n
\]
outside an event of probability $2e^{-n/2}$.  On the complementary event, and using that $m^2 \leq 4n$ for large enough $n$, $\| \sfR (B)_J \|_2 \leq \| \sfR (B) \|_2 \leq C_0 \sqrt n \| B \|_\F$ for all $B \in \sfS^m$. 
\end{proof}

\begin{proof}[Proof of Proposition~\ref{prop:hybrid-gap}]
We now implement the argument described at the beginning of the section. Fix $0\le r\le k_d$, put $m=d-r$.  Let $\cE_r$ be the event on
which  
\begin{equation}\label{eq:head-good-prob}
    \bigl\| \sfK (H,b)_J \bigr\|_2
\;\ge\; a_0 \sqrt n\, \bigl(2\|H\|_{\F}^2 + b^2\bigr)^{1/2} 
\end{equation}
holds simultaneously for every $H\in\Sym r$, $b\in\R$, and every $J\subseteq[n]$ with $|J|\ge(1-\delta)n$.  Since $\delta<\delta_0$, Lemma~\ref{lem:head-smallball} gives $\Prob(\cE_r^c)\le e^{-c_0n}$.
For $r=0$, the corresponding head estimate holds deterministically, after decreasing $a_0$ if necessary, because $\sfK(0,b)=-b\one$.

Condition on a realization $(g_i)_{i\le n}\in\cE_r$.  For every $J\subseteq[n]$ with $|J|\ge(1-\delta)n$, let $P_J$ be the orthogonal projection defined in \eqref{eq:PJ-def}.  In particular, $P_J \sfK(H,b)_J=0$  for every $H\in\Sym r$ and $b\in\R$. By Lemma~\ref{lem:projected-bulk-escape}, conditionally on the head vectors, outside an event of probability at most $\rho_d := e^{-\Delta^2d^2/64}+e^{-n/4}$, the estimates
\begin{equation}\label{eq:bulk-estimates-used}
\|P_J\sfR(B)_J\|_2
\ge c_{\mathrm{bulk}}\sqrt n\,\|B\|_{\F},
\qquad
\|\sfR(B)_J\|_2
\le \|\sfR(B)\|_2
\le C_0\sqrt n\,\|B\|_{\F}
\end{equation}
hold simultaneously for every $B\succeq0$ and every admissible $J$.
Since $n/d^2\to\gamma$, there is a constant $c_\gamma'>0$, independent
of $r$, such that $\rho_d\le e^{-c_\gamma'd^2}$ for all sufficiently large $d$.

Work on the event in \eqref{eq:bulk-estimates-used}.  Define
\begin{equation}\label{eq:bulk-threshold}
\theta :=
\min\left\{\frac14,\frac{a_0}{8C_0}\right\}.
\end{equation}
Fix $H\succeq0$, $B\succeq0$, and $b\in\R$ satisfying \eqref{eq:admissible-hybrid}  and fix $J\subseteq[n]$ with $|J|\ge(1-\delta)n$.  Put $u = \sfK (H,b) + \sfR (B)$. If $\|B\|_{\F}\ge\theta$, then by definition of $P_J$ and the first estimate in \eqref{eq:bulk-estimates-used}, we have
\begin{equation}\label{eq:case-bulk}
\begin{aligned}
\|u_J\|_2
&\ge \|P_Ju_J\|_2 =\|P_J \sfR(B)_J\|_2 \ge c_{\mathrm{bulk}}\sqrt n\,\|B\|_{\F} \ge c_{\mathrm{bulk}}\theta\sqrt n.
\end{aligned}
\end{equation}
If instead $\|B\|_{\F}<\theta$, then $\|H\|_{\F}^2 \ge\frac14-\theta^2 \ge\frac{3}{16}$. The head small-ball estimate \eqref{eq:head-good-prob}, the second estimate in \eqref{eq:bulk-estimates-used}, and the reverse triangle inequality imply
\begin{align}
\|u_J\|_2 \ge
\|\sfK (H,b)_J\|_2-\|\sfR(B)_J\|_2 \ge&~ a_0\sqrt n\, \bigl(2\|H\|_{\F}^2+b^2\bigr)^{1/2} -C_0\sqrt n\,\|B\|_{\F} \notag\\
\ge&~  a_0\sqrt n\sqrt{\frac38} -C_0\theta\sqrt n \ge a_0\left(\frac{\sqrt6}{4}-\frac18\right)\sqrt n =:c_{\mathrm{head}}\sqrt n.
\label{eq:case-head}
\end{align}
Consequently, with $c_4:= \min\left\{ c_{\mathrm{bulk}}\theta,\, c_{\mathrm{head}} \right\}>0$, we have
\begin{equation}\label{eq:joint-lower}
\| \bigl(\sfK(H,b)+\sfR(B)\bigr)_J\|_2 \ge c_4\sqrt n
\end{equation}
simultaneously for every admissible $(H,B,b)$ and every $J\subseteq[n]$ with $|J|\ge(1-\delta)n$.

It remains to deduce from \eqref{eq:joint-lower} that the risk is bounded away from $0$. Fix an admissible $(H,B,b)$ and define $I=\{i\in[n]:|u_i|>1\}$. If $|I|\ge\delta n$, then the lower bound in \eqref{eq:phi-sandwich} gives $\Loss_n(u)\ge c_\varphi|I|/n\ge c_\varphi\delta$.  If $|I|<\delta n$, set $J = [n]\setminus I$. Then $|J|\ge(1-\delta)n$ and \eqref{eq:joint-lower} gives $\Loss_n(u)\ \geq c_\varphi \|u_J\|_2^2/n\ge c_\varphi  c_4^2$.  Thus, setting $e_\gamma := c_\varphi\min\{\delta,c_4^2\}>0$, we conclude that the infimum in \eqref{eq:hybrid-gap} is at least $e_\gamma$ on the joint event where \eqref{eq:head-good-prob} and \eqref{eq:bulk-estimates-used} hold.  In particular, for every head realization $(g_i)_{i\le n} \in \cE_r$, we have the conditional estimate
\begin{equation}\label{eq:hybrid-conditional}
\begin{aligned}
&\Prob_G\left\{
\inf_{(H,B,b) \text{ admissible}}
\Loss_n\bigl(\sfK(H,b)+\sfR(B)\bigr)
<e_\gamma
\ \middle|\ (g_i)_{i\le n}
\right\} \le \rho_d \le e^{-c_\gamma'd^2}.
\end{aligned}
\end{equation}
 Integrating \eqref{eq:hybrid-conditional} over the head vectors, we obtain \eqref{eq:hybrid-gap} with probability at least $1 - e^{-c_0n}-e^{-c_\gamma'd^2}$. Since $n/d^2\to\gamma$, there is a constant $c_\gamma>0$ such that $e^{-c_0n}+e^{-c_\gamma'd^2}\le e^{-c_\gamma d^2}$ for all sufficiently large $d$.  This completes the proof.
\end{proof}

\subsection{Deterministic nets and transfer}\label{sec:unsat-nets}

We now prove Proposition~\ref{prop:true-net-gap} for the deterministic net $\cI_d$ defined in Section~\ref{sec:unsat-setup}.  For each fixed net subspace we condition on its head coordinates and use Proposition~\ref{prop:translated-univ} to transfer the hybrid-model risk gap proved in Proposition \ref{prop:hybrid-gap} to the ellipsoid model. Combining bounded differences with the entropy bound \eqref{eq:total-entropy} yields the uniform risk gap over the net $\cI_d$. 

For a deterministic $U\in\fU_r$, decompose each sample into its independent head and bulk components,
\begin{equation}\label{eq:decompose}
x_i=g_i+z_i\in U\oplus U^\perp,
\qquad g_i\sim \cN(0,I_U),\quad z_i\sim \cN(0,I_{U^\perp}).
\end{equation}
For $H\in\fH_U$ and $b\in\frB_\tau$, set
\begin{equation}\label{eq:offsets-def}
a_i=g_i^\top H g_i-\Tr H-b
\end{equation}
and define the conditional ellipsoid-model value
\begin{equation}\label{eq:Eell-g}
\scE_{U,H,b} ( \{ W_i\} )
:=\inf_{B\in\cC_U(H)}\frac1n\sum_{i=1}^n
\varphi\bigl(a_i+ \sqrt{m} \Tr ( BW_i )\bigr), \qquad W_i := \frac{z_i z_i^\top - I_m}{\sqrt m}.
\end{equation}
Similarly, we consider the infimum $\scE_{U,H,b} ( \{ G_i\})$ with the bulk term $z_i^\top Bz_i-\Tr B$ replaced by $\sqrt m\,\Tr(G_i B)$ with $G_i \sim_{\rm iid} \GOE m$.

We condition on the small-ball event $\cE_U$ in Lemma \ref{lem:head-smallball}, that is, $(g_i)_{i\le n} \in \cE_U$ if
\begin{equation}\label{eq:head-event}
\bigl\| \sfK(H, b)_J \bigr\|_2 \ \ge\ a_0 \sqrt n (2\|H\|_{\F}^2 + b^2)^{1/2},
\end{equation} 
for all $H \in \Sym r$, $b \in \R$, and $J \subseteq [n]$ with $|J| \ge (1-\delta) n$, where $\sfK(H,b)_i = g_i^\top H g_i - \Tr H - b$ is the head contribution.

To prove Proposition~\ref{prop:true-net-gap}, we first transfer the hybrid-model risk gap to a conditional risk gap for each fixed head subspace.  McDiarmid's inequality then upgrades that risk gap to high probability, and the $e^{o(d^2)}$ net entropy permits a union bound.

\begin{lemma}[Conditional risk gap for a fixed head subspace]
\label{lem:conditional-true-mean-gap}
Fix $0\le r\le k_d$ and a deterministic $U\in\fU_r$.  Let $\cE_U$ be the
head event in \eqref{eq:head-event}.  Uniformly in $r$ and $U$,
\[
\Prob(\cE_U^c)\le e^{-c_\gamma d^2}.
\]
Moreover, for all sufficiently large $d$, on $\cE_U$, simultaneously for
every $H\in\fH_U$ and $b\in\frB_\tau$ for which
$\cC_U(H)\ne\varnothing$,
\begin{equation}\label{eq:Eell-mean}
\E \left[ \scE_{U,H,b} ( \{ W_i\} )  \mid (g_i)_{i\le n} \right] \ge\frac{e_\gamma}{2}.
\end{equation}
\end{lemma}

\begin{proof}[Proof of Lemma~\ref{lem:conditional-true-mean-gap}]
Fix $U\in\fU_r$.  Since $\delta<\delta_0$,
Lemma~\ref{lem:head-smallball} gives, uniformly for $r\le k_d$,
\begin{equation}\label{eq:head-event-prob}
\Prob(\cE_U^c) \le e^{-c_\gamma d^2} .
\end{equation}
Now condition on a head realization $(g_i)_{i\le n}$ in $\cE_U$. 
For every $H\in\fH_U$ and $b\in\frB_\tau$ with $\cC_U(H)\ne\varnothing$, each $B\in\cC_U(H)$ makes $(H,B,b)$ admissible in \eqref{eq:hybrid-conditional}, and with probability at least $1-e^{-c_\gamma' d^2}$,
\[
\scE_{U,H,b} ( \{ G_i\}) \ge e_\gamma.
\]
Since
$0\le \scE_{U,H,b} ( \{ G_i\}) \le1$,
\begin{equation}\label{eq:Egoe-mean}
\E \big[ \scE_{U,H,b} ( \{ G_i\}) \mid (g_i)_{i\le n} \big]
\ge e_\gamma\Prob \{\scE_{U,H,b} ( \{ G_i\}) \ge e_\gamma  \mid (g_i)_{i\le n}\}
\ge e_\gamma\bigl(1-e^{-c_\gamma ' d^2}\bigr).
\end{equation}

Proposition~\ref{prop:translated-univ}, with the verifications
\eqref{eq:B-schatten}--\eqref{eq:uniform-aspect} and its uniformity over
arbitrary deterministic offsets, implies, conditionally on every fixed
realization of the head,
\begin{equation}\label{eq:univ-applied}
\bigl| \E \big[ \scE_{U,H,b} ( \{ W_i\}) \mid (g_i)_{i\le n} \big]
- \E \big[ \scE_{U,H,b} ( \{ G_i\}) \mid (g_i)_{i\le n} \big] \bigr| =
o(1) ,
\end{equation}
uniformly in $r,U,H,b$ and $(g_i)_{i\le n}$.  For all sufficiently large $d$,
\eqref{eq:Egoe-mean}--\eqref{eq:univ-applied} then give \eqref{eq:Eell-mean} on $\cE_U$.
\end{proof}

\begin{proof}[Proof of Proposition~\ref{prop:true-net-gap}]
For fixed $U,H,b$ and $(g_i)_{i\le n}$, the value $\scE_{U,H,b} $ changes by at most $1/n$ when one bulk row is replaced, because $0 \le \varphi \le 1$.  McDiarmid's inequality
\cite[Theorem~6.2]{boucheron2013concentration} and \eqref{eq:Eell-mean} therefore give
\begin{equation}\label{eq:mcdiarmid}
\Prob_z\bigl\{ \scE_{U,H,b} ( \{ W_i\})  < e_\gamma/4
\ \big|\ (g_i)_{i\le n} \bigr\}
\;\le\; \exp\bigl\{- e_\gamma^2\, n / 8\bigr\} .
\end{equation}
 Write $U(t)$ for the subspace of a tuple $t\in\cI_d$.  For fixed $(r, U)$, let $\cF_U = \sigma\bigl((P_U x_i)_{i \le n}\bigr)$.  Then $\cE_U \in \cF_U$, while conditionally on $\cF_U$ the bulk components $(P_{U^\perp} x_i)_{i \le n}$ are independent standard Gaussians on $U^\perp$.  Consequently, by
\eqref{eq:head-event-prob} and \eqref{eq:mcdiarmid},
\begin{align*}
\Prob\bigl\{\text{\eqref{eq:net-gap} fails for some } t \in \cI_d\bigr\}
&\le \sum_{(r,U) \in \frV_d} \Prob(\cE_U^c)
+ \sum_{t \in \cI_d}
\E\Bigl[\mathbf 1_{\cE_{U(t)}}\,
\Prob\bigl\{\text{failure at } t \mid \cF_{U(t)}\bigr\}\Bigr]\\
&\le |\frV_d|\, e^{-c_\gamma d^2}
+ |\cI_d|\, e^{-e_\gamma^2 n/8} \;=\; o(1) ,
\end{align*}
where the last equality follows from \eqref{eq:total-entropy} and $n/d^2\to\gamma$. 
\end{proof}

\section*{Acknowledgments and use of AI}

TM would like to thank Basil Saeed for suggesting this problem and for helpful discussions. The authors are grateful to Afonso Bandeira for his support and encouragement.

We used modern AI tools in this work. Specifically, we used ChatGPT 5.4 and 5.5 to explore several possible proof strategies. The approach presented in this paper was suggested by the authors and motivated directly by the dual formulation in \cite{bandeira2024fitting,bandeira2025exact} and by the author's previous work on conditional Gaussian equivalence \cite{wen2025does}. Given an earlier draft, GPT 5.6 helped repair and complete several arguments, including the tightened head-tail decomposition in Lemma \ref{lem:regularity} and the decomposition used in the proof of Proposition~\ref{prop:hybrid-gap}, which ultimately led to the completion of the proofs. The authors carefully verified all the results, and they take full responsibility for the content of this paper.

\bibliographystyle{amsalpha}
\bibliography{biblio}

\end{document}

%% file: def.tex
\newcommand{\R}{\mathbb{R}}
\newcommand{\E}{\mathbb{E}}
\renewcommand{\P}{\mathbb{P}}
\newcommand{\Prob}{\mathbb{P}}
\newcommand{\Sym}[1]{\mathsf{S}^{#1}}
\newcommand{\Sd}{\mathsf{S}^d}
\newcommand{\Sdp}{\mathsf{S}^d_+}
\newcommand{\Symp}[1]{\mathsf{S}^{#1}_+}
\newcommand{\Tr}{\operatorname{Tr}}
\newcommand{\Id}{I_d}
\newcommand{\diag}{\operatorname{diag}}
\newcommand{\rank}{\operatorname{rank}}
\newcommand{\sgn}{\operatorname{sgn}}
\newcommand{\ip}[2]{\left\langle #1,#2\right\rangle}
\newcommand{\op}{\mathrm{op}}
\newcommand{\F}{\mathrm{F}}
\newcommand{\sphereF}{\mathbb{S}_{\F}}
\newcommand{\psdsphere}[1]{\mathbb{S}^{+}_{\F}(#1)}
\newcommand{\GOE}[1]{\operatorname{GOE}(#1)}
\newcommand{\Gr}[2]{\operatorname{Gr}(#1,#2)}
\newcommand{\one}{\mathbf{1}}
\newcommand{\cA}{\mathcal{A}}
\newcommand{\cB}{\mathcal{B}}
\newcommand{\cC}{\mathcal{C}}
\newcommand{\cD}{\mathcal{D}}
\newcommand{\cE}{\mathcal{E}}
\newcommand{\cF}{\mathcal{F}}
\newcommand{\cG}{\mathcal{G}}
\newcommand{\cH}{\mathcal{H}}
\newcommand{\cI}{\mathcal{I}}
\newcommand{\cK}{\mathcal{K}}
\newcommand{\Loss}{\mathscr{L}}
\newcommand{\cN}{\mathcal{N}}
\newcommand{\cQ}{\mathcal{Q}}
\newcommand{\cR}{\mathcal{R}}
\newcommand{\cS}{\mathcal{S}}
\newcommand{\cT}{\mathcal{T}}

\newcommand{\cW}{\mathcal{W}}
\newcommand{\cY}{\mathcal{Y}}
\newcommand{\fG}{\mathfrak{G}}
\newcommand{\fU}{\mathfrak{U}}
\newcommand{\fH}{\mathfrak{H}}
\renewcommand{\S}{\mathbb{S}}
\newcommand{\de}{\mathrm{d}}
\newcommand{\eps}{\varepsilon}
\newcommand{\scF}{\mathscr{F}}
\newcommand{\scE}{\mathscr{E}}

\def\sfR{\mathsf{R}}
\def\sfK{\mathsf{K}}
\def\frB{\mathfrak{B}}
\def\frV{\mathfrak{W}}
\def\sfP{\mathsf{P}}